\documentclass[12pt,reqno]{amsart}
\usepackage{amsmath}
\usepackage{amsfonts}
\usepackage{amssymb, amsthm, epsfig}
\usepackage{mathrsfs}
\usepackage{hyperref}
\usepackage{cite}
\usepackage[utf8]{inputenc}
\usepackage[T1]{fontenc}
\usepackage{geometry}
\usepackage{graphicx}
\usepackage{verbatim}
\usepackage{color}

\theoremstyle{plain}
\newtheorem{thm}{Theorem}[section]

\newtheorem{lem}[thm]{Lemma}
\newtheorem{prop}[thm]{Proposition}

\theoremstyle{definition}
\newtheorem{defn}{Definition}[section]
\theoremstyle{remark}
\newtheorem{rem}{Remark}[section]

\numberwithin{equation}{subsection}%{section}

\allowdisplaybreaks

\newcommand{\abs}[1]{\left\vert#1\right\vert}
\newcommand{\set}[1]{\left\{#1\right\}}

\newcommand{\qnt}[1]{\left(#1\right)}

\newcommand{\lb}[1]{\left\{#1\right.}

\def\parbar{{\partial\mkern-10.5mu/}}
\DeclareMathOperator{\sgn}{sgn}
\DeclareMathOperator*{\dist}{dist}

\newcommand{\me}{\mathrm{e}}
\newcommand{\dif}{\mathrm{d}}

\DeclareSymbolFont{lettersA}{U}{pxmia}{m}{it}
\DeclareMathSymbol{\piup}{\mathord}{lettersA}{"19}

\makeatletter

\newcommand{\Rmnum}[1]{\expandafter\@slowromancap\romannumeral#1@}
\makeatother
\newcommand{\Real}{\mathbb R}

\begin{document}
\title[Hypersonic flow onto a large curved wedge and the dissipation of shock wave]
{Hypersonic flow onto a large curved wedge and the dissipation of shock wave}

\author{Dian Hu}

\address{School of Mathematics, East China University of Science and Technology,
Shanghai, 200237, China.}
\email{\tt hudianaug@qq.com}

\author{Aifang Qu}

\address{Department of Mathematics, Shanghai Normal University,
Shanghai, 200234, China.}
\email{\tt afqu@shnu.edu.cn}

\keywords{Hypersonic flow; piecewise smooth potential flow; global attached shock; shock wave dissipation}
\subjclass[2010]{35B20, 35D30, 35Q31, 35L65, 76J20, 76L05, 76N10}
\date{\today}

\begin{abstract}
The flow field with a Mach number larger than 5 is named hypersonic flow. In this paper, we explore the existence of smooth flow field after shock for hypersonic potential flow past a curved smooth wedge with neither smallness assumption on the height of the wedge nor that it is a BV perturbation of a line. The asymptotic behaviour of the shock is also analysed. We prove that for given Bernoulli constant of the incoming flow, there exists a sufficient large constant such that if the Mach number of the incoming flow is larger than it, then there exists a global shock wave attached to the tip of the wedge together with a smooth flow field between it and the wedge. The state of the flow after shock is in a neighbourhood of a curve that is determined by the wedge and the density of the incoming flow. If the slope of the wedge has a positive limit as $x$ goes to infinity, then the slope of the shock tends to that of the self-similar case that the same incoming flow past a straight wedge with slope of the limit. Specifically, we demonstrate that if the slope of the wedge is parallel to the incoming flow at infinity, the strength of the shock will diminish to zero at infinity. The restrictions on the surface of a wedge have been greatly relaxed compared to the previous works on supersonic flow past wedges. The method employed in this paper is characteristic decomposition, and the existence of the solution is obtained by finding an invariant domain of the solution based on geometry structures of the governing equations. The ideas and methods presented here may be applicable to other problems.
\end{abstract}

\maketitle
\tableofcontents

\section{Introduction}\label{sect:Introduction}
We are concerned with the problems of attached shock arising in hypersonic flows past wedges with large data. These problems are fundamental not only in aerodynamics but also in the mathematical theory of multidimensional conservation laws since their solutions play the role of building blocks of general solutions to the multidimensional Euler equations for compressible fluids \cite{Anderson2006AIAA} \cite{CourantFriedrichs1976AMS}. When a supersonic flow passes a wedge with an angle less than a critical value, an attached shock appears ahead of it. As the Mach number of the upcoming flow increases, the shock fronts approach the wedge, the region between the shock and the wedge becomes narrower and narrower, which is called a thin shock layer. The flow is called hypersonic if its Mach number is bigger than 5. The hypersonic flows show different features from those of supersonic flows. The thin shock layer experiences small variations as the Mach number of the incoming flow increases beyond a certain value. This important feature was observed in physical and engineering. However, many fundamental issues for hypersonic flow past obstacles have not been understood, including the global structure, stability and asymptotic behavior. Therefore, it is essential to establish the global existence and structural stability of solutions of hypersonic flow past an obstacle in order to understand fully the Mach number independence principle. On the other hand, there has been no rigorous mathematical results on the global existence and stability of the solution, including the case of potential flow which is widely used in aerodynamics (cf. \cite{Anderson2006AIAA} \cite{CourantFriedrichs1976AMS}). One of the main reasons is that the problems involve several challenging difficulties in the analysis of nonlinear partial differential equations such as solutions with large data, free boundary problems, and singularity similar as vacuum when the shock layer is thin. In this paper we develop a rigorous mathematical approach to overcome these difficulties and establish a global theory of existence and stability for hypersonic flow past a curved wedge for potential flow. The techniques and ideas developed here will be useful for other nonlinear problems involving similar difficulties.

On the $(x, y)$-plane, the steady Euler equations for potential flow are
\begin{equation}\label{eq:EulerEquations}
\lb{
\begin{aligned}
&(\rho u)_x + (\rho v)_y = 0,\\
&v_x-u_y=0,
\end{aligned}
}
\end{equation}
where the first line is the conservation of mass and the second is the irrotational condition. The unknowns $\rho$ and $(u, v)$ denote respectively the density of the mass and the velocity of the flow. We use the notation $\vec{U}:=(u, v)$. For polytropic gas, the pressure of the flow is
$$p(\rho)=A\rho^\gamma,$$ 
where $A>0$ is a constant and $\gamma\in(1, 3)$ the adiabatic index of the gas. The sound speed $c(\rho)$ is given by
\begin{equation}
c^2(\rho)=\frac{\dif p}{\dif \rho}=A\gamma\rho^{\gamma-1}.
\end{equation}
The density $\rho$ and the velocity $(u, v)$ are connected by the Bernoulli’s law
\begin{equation}\label{eq:BernoulliLaw}
\frac{u^2+v^2}{2}+\frac{c^2}{\gamma-1}=\mathbf{B}=\frac{\bar{q}^2}{2}
\end{equation}
with $\mathbf{B}$ being the Bernoulli constant determined by the incoming flow and $\bar{q}$ is the limit speed. The Mach number of the flow is defined by
\begin{equation}\label{eq:Mach}
\mathbf{M}=\frac{q}{c}
\end{equation}
with $q=\sqrt{u^2+v^2}$. For \eqref{eq:BernoulliLaw}, it can be solved that
\begin{equation}\label{eq:cuv}
c(u, v)=\sqrt{\frac{\gamma-1}{2}(\bar{q}^2-u^2-v^2)},\quad\mbox{and}\quad\rho(u, v) = \qnt{\frac{\gamma-1}{2\gamma}(\bar{q}^2 - u^2 - v^2)}^{\frac{1}{\gamma-1}}.
\end{equation}
Thus, in the following, we keep in mind that $c$ and $\rho$ are functions of $(u, v)$.

In the upper half of the $(x, y)$ plane, there is a wedge $\textsf{W}$ with a given $C^2$ smooth boundary function $y=f(x)$, where $f(x)>0$ for $x>0$ and $f(0)=0$. Here, we are mainly concerned with the following two kinds of wedge respectively.
\begin{description}
\item[Case A: Wedge without convexity] The function $f\in C^2(\Real^+)$ satisfies
\begin{equation}\label{eq:SupersonicWallCondition}
\underline{f}\leq f'\leq\bar{f}<\sqrt{\frac{2}{\gamma-1}},
\end{equation}
and
\begin{equation}\label{eq:WallCondition2nd}
\abs{f''(x)}\leq\frac{B}{\qnt{1+x}^{1+a}}
\end{equation}
where $\underline{f}$, $\bar{f}$, $a$ and $B$ are any positive constants with
\begin{equation}\label{eq:SupersonicWallConditionf}
0<\underline{f}<\bar{f}<\sqrt{\frac{2}{\gamma-1}}.
\end{equation}

\item[Case B: Bullet-like wedge] The function $f\in C^3_{loc}(\Real^+)$ satisfies
\begin{equation}\label{eq:ConvexWall}
f''\leq 0;
\end{equation}
\begin{equation}\label{eq:ConvexWallOrigin}
0<f'(0)<\sqrt{\frac{2}{\gamma-1}};
\end{equation}
and
\begin{equation}\label{eq:AsyCondition}
\lim\limits_{x\rightarrow+\infty}f'(x)\in[0, \sqrt{\frac{3-\gamma}{\gamma-1}}).
\end{equation}
\end{description}
By \eqref{eq:WallCondition2nd} and \eqref{eq:ConvexWall}, for both {\bf Case A} and {\bf Case B}, we can define $f_\infty$ as follows
\begin{equation}\label{eq:finfty}
f_\infty:=\lim\limits_{x\rightarrow+\infty}f'(x).
\end{equation}
It is worth noting that the assumptions on $f$ in {\bf Case A} mean no small BV bounded or even any convexity.

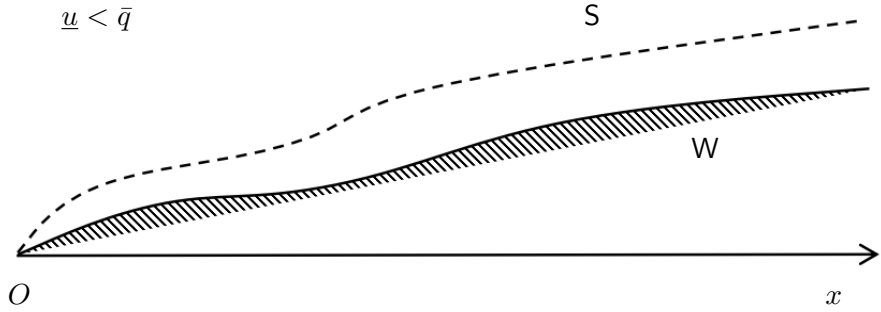
\begin{figure}[htbp]
\setlength{\unitlength}{1bp}
\begin{center}
\begin{picture}(190,190)
\put(-60,-10){\includegraphics[scale=0.4]{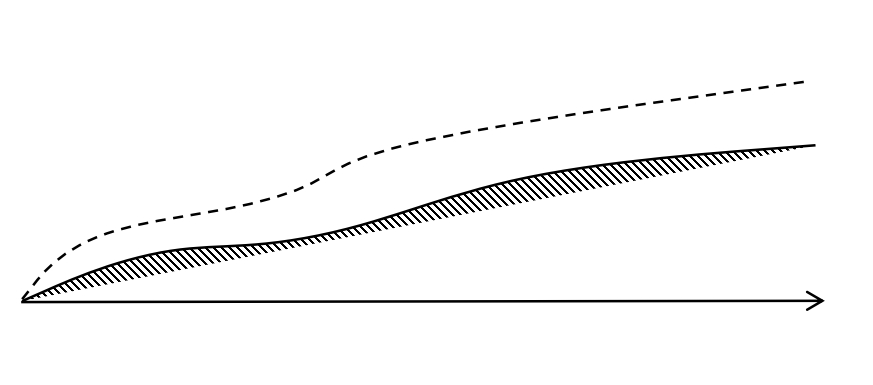}}
\put(-35,110){\small$\underline{u}<\bar{q}$}
\put(160,110){\small$\textsf{S}$}
\put(200,60){\small$\textsf{W}$}
\put(-55,5){\small$O$}
\put(250,5){\small$x$}
\end{picture}
\end{center}
\caption{supersonic flow onto a curved wedge}\label{fig:convexwedge}
\end{figure}

Our problem can be formulated as follows: {\bf Problem 1 (Initial boundary value problem):} Given $\mathbf{B}$, seek a solution of system \eqref{eq:EulerEquations} with given velocity $\qnt{\underline{u}, \underline{v}}$ of the incoming flow at $x=0$
\begin{equation}\label{eq:initialdata}
 (\rho, u, v)=(\underline{\rho}, \underline{u}, 0)\quad\mbox{on}\quad x=0,
\end{equation}
and the slip boundary condition along the wedge boundary $\mathsf{W}$:
\begin{equation}\label{eq:VelocityWallCondition}
v(x, y)=u(x, y)f'(x)\quad\mbox{on}\quad y=f(x),
\end{equation}
where $f$ satisfies either {\bf Case A} or {\bf Case B}.

For a supersonic flow past the wedge, a shock presents ahead of the wedge. Since the incoming flow is uniform and the boundary of the wedge is $C^2$ smooth, we seek piecewise smooth solutions of the form
\begin{equation}\label{eq:PiecewiseSolution}
(\rho, u, v)(x, y)=\begin{cases}
(\underline{\rho}, \underline{u}, 0), & x>0,~y>\phi(x),\\
(\rho, u, v),& x>0,~f(x)<y<\phi(x),
\end{cases}
\end{equation}
where $y=\phi(x)$ is the discontinuity of the solution, namely, the shock front, along which it holds the Rankine-Hugoniot (R-H) condition
\begin{equation}\label{eq:RHCondition}
\lb{
\begin{aligned}
&[\rho u]\phi' = [\rho v];\\
&[v]\phi' = -[u],
\end{aligned}
}
\end{equation}
where the bracket means the jump of the corresponding quantities on the discontinuity curve $y=\phi(x)$, for example, $[v]|_{y=\phi(x)}=v(x, \phi(x)-)-v(x, \phi(x)+)$. Then {\bf Problem 1} can be summarised as follows.

{\bf Problem 2 (Free boundary value problem):} Determine the flow fields $(u, v)$ and their asymptotic behaviour in the domain $\Omega$ bounded by the wedge $\textsf{W}$ and the free boundary $\textsf{S}: y=\phi(x)$, satisfying \eqref{eq:EulerEquations}, \eqref{eq:BernoulliLaw}, \eqref{eq:VelocityWallCondition}, \eqref{eq:PiecewiseSolution} and \eqref{eq:RHCondition} for a given smooth function $f$ satisfying either {\bf Case A} (\eqref{eq:SupersonicWallCondition} \eqref{eq:WallCondition2nd}) or {\bf Case B} (\eqref{eq:ConvexWall} \eqref{eq:ConvexWallOrigin}). Meanwhile, the entropy condition $p>\underline{p}$ should hold along the shock.

Due to the appearance of the wedge, the gas is compressed when it past the line $x=0$, and a shock presents in front of the wedge. Following the von Neumann criterion, the local existence of an attached shock has been thoroughly researched by many mathematicians and system theories have been developed to this kind of problems, provided that the wedge angle is less than a critical value. The shock front is a small perturbation of the one resulting in the same incoming flow past a straight wedge with inclined angle $\arctan f'(0)$.

We develop a rigorous mathematical approach to extend the local theory to a global theory for solutions of an attached shock. The shock approaches to the wedge as the Mach number of the incoming flow $\underline{\mathbf{M}}$ tends to infinity. The solution is supersonic and $C^1$ smooth in the domain bounded by the shock and the wedge. The shock fronts is $C^2$ smooth. Our main results are stated as follows

\begin{thm}\label{thm:MainTheorem1}
For the supersonic incoming flow with fixed Bernoulli's constant $\mathbf{B}$, there exists $\mathbf{M}_1>1$ only depending on the wall and $\mathbf{B}$, such that, for $\underline{\mathbf{M}}>\mathbf{M}_1$, the problem \eqref{eq:EulerEquations}, \eqref{eq:initialdata} and \eqref{eq:VelocityWallCondition} admits a global piecewise smooth solution $(\rho, u, v)$ with an attached shock $y=\phi(x)$ satisfying \eqref{eq:RHCondition}, provided that the wedge boundary $f\in C^2(\Real^+)$ satisfies \eqref{eq:SupersonicWallCondition}-\eqref{eq:SupersonicWallConditionf}.

In addition, we have the asymptotic result
\begin{equation}\label{eq:ApproximateVelocity}
\lim\limits_{x\rightarrow+\infty\atop(x, y)\in\Omega}(u, v)(x, y)=(u_\infty, v_\infty),
\end{equation}
where $(u_\infty, v_\infty)$ is the constant supersonic solution for corresponding problem involving the straight wedge $f(x)=f_\infty x$. Additionally, we have
\begin{equation}\label{eq:ApproximateShock}
\lim\limits_{x\rightarrow+\infty}\phi'(x)=\frac{\underline{u}-u_\infty}{v_\infty}.
\end{equation}
Furthermore, the shock approaches the wedge as the Mach number of the incoming flow increases.
\end{thm}

\begin{figure}[htbp]
\setlength{\unitlength}{1bp}
\begin{center}
\begin{picture}(190,190)
\put(-60,-10){\includegraphics[scale=0.4]{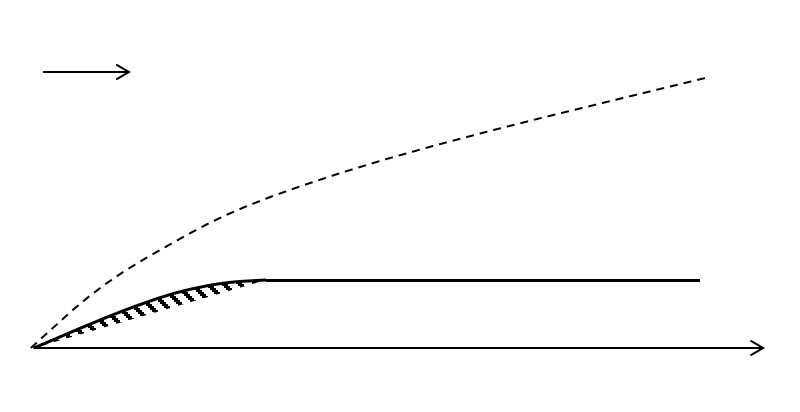}}
\put(-43,104){\small$\underline{u}<\bar{q}$}
\put(160,123){\small$\textsf{S}$}
\put(200,55){\small$\textsf{W}$}
\put(-58,2){\small$O$}
\put(250,5){\small$x$}
\end{picture}
\end{center}
\caption{supersonic flow onto a bullet wedge}\label{fig:bulletwedge}
\end{figure}
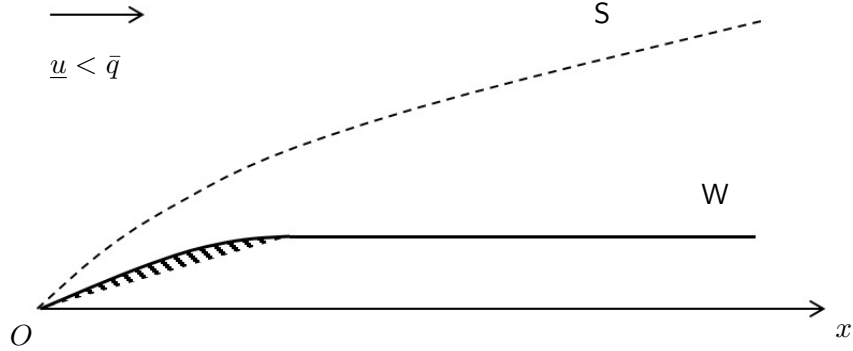

\begin{thm}\label{thm:MainTheorem2}
For the supersonic incoming flow with Mach number $\underline{\mathbf{M}}$ and fixed Bernoulli's constant $\mathbf{B}$, if the wedge boundary $f\in C^3_{loc}(\Real^+)$ satisfies \eqref{eq:ConvexWall}-\eqref{eq:AsyCondition}, then there exists $\mathbf{M}_1>1$, such that, for $\underline{\mathbf{M}}>\mathbf{M}_1$, the problem \eqref{eq:EulerEquations}, \eqref{eq:initialdata} and \eqref{eq:VelocityWallCondition} admits a global piecewise smooth solution $(u, v)$ with an attached shock $y=\phi(x)$ satisfying \eqref{eq:RHCondition} with the asymptotic results \eqref{eq:ApproximateVelocity} and \eqref{eq:ApproximateShock}. In particular, as illustrated in Figure \ref{fig:bulletwedge}, for the wedge of bullet-like shape where $f_\infty=0$ and adiabatic index $\gamma=2$, the shock is extinct at infinity and we have
\begin{equation}\label{eq:ApproximateVelocityF}
\lim\limits_{x\rightarrow+\infty\atop(x, y)\in\Omega}(u, v)=(\underline{u}, 0),
\end{equation}
and
\begin{equation}\label{eq:ApproximateShockF}
\lim\limits_{x\rightarrow+\infty}\phi'(x)=\frac{\underline{c}}{\sqrt{\underline{u}^2-\underline{c}^2}}.
\end{equation}
\end{thm}

The problem of given incoming flow past an obstacle has been extensively studied in various models across numerous literatures. For the case where the wall boundary is a straight line, extensive research has been performed, as exemplified in \cite{CourantFriedrichs1976AMS}. Subsequently, many authors have investigated the nonlinear perturbation problem, see \cite{Chen2017SCM,ChenFang2009DCDS,ChenFang2017AM,ChenLi2008JDE,ChenZhangZhu2006ARMA,Chen1998CAMS,Chen1998JPDE,Chen2003ZAMP,ChenFang2007JDE,ChenXinYin2002CMP,CuiYin2009JDE,LiWittYin2014CMP,LienLiu1999CMP,ZhangCui2012AAM,YinZhou2009JDE,XuYin2009SIAM,Schaeffer1976DMJ,Zhang1999SIAM,Zhang2003JDE} and the references therein. All these studies commonly assume that the solid wall boundary is a small perturbation of a straight line and the flow fields behind the shock are perturbations of constant states.

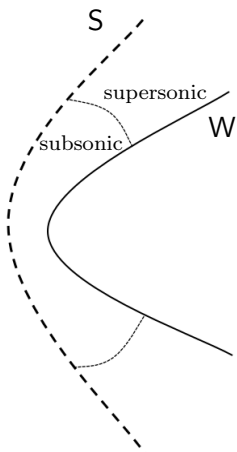
\begin{figure}[htbp]
	\setlength{\unitlength}{1bp}
	\begin{center}
		\begin{picture}(190,190)
		\put(-10,-10){\includegraphics[scale=0.4]{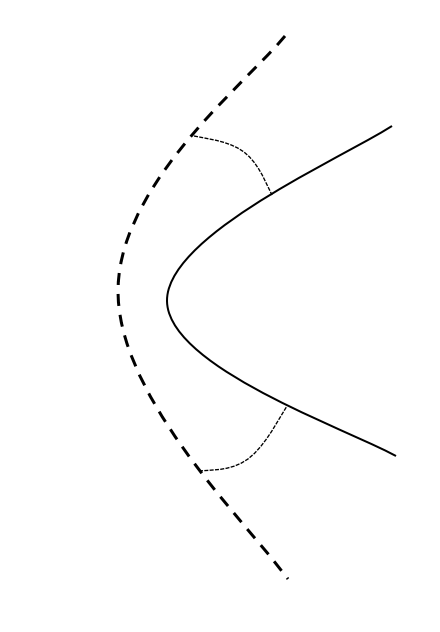}}
		\put(61,139){\tiny$\mbox{supersonic}$}
		\put(37,120){\tiny$\mbox{subsonic}$}
		\put(55,165){\small$\textsf{S}$}
		\put(100,125){\small$\textsf{W}$}
		\end{picture}
	\end{center}
	\caption{detached shock}\label{fig:bluntbody}
\end{figure}

As illustrated in Figure \ref{fig:bluntbody}, an interesting issue arises in the context of supersonic flow over a blunt body. For an obstacle with a boundary where $f'(0)=+\infty$, the shock front is detached. The downstream flow fields around the head of the body are subsonic. If the slope of the wall boundary at infinity is less than a critical value, then the downstream flow becomes supersonic there. The downstream flow field exhibits a wide range of variations. This wave structure has been consistently observed across numerous experiments. Mathematically, determining the position of shock front and the distribution of the flow field behind them is a challenging problem with significant importance in both theory and application, see \cite{Chen2010QAM}. This wave structure is characterized by three features: (1) shock, (2) solid wall, and (3) flow fields behind the shock with large variations. To our knowledge, in steady flow scenarios, there are few results that simultaneously encompass all these three features. For self-similar case, such problem has been studied in \cite{BaeChenFeldman2009IM,EllingLiu2008CPAM,ChenFeldman2010AM,ChenFeldman2018AMS,ChenFeldmanXiang2020ARMA}.

%In \cite{HuZhang2019SIAM}, the authors introduce a limit case for this problem. They consider the incoming flow reaching a limit speed $\underline{u}=\bar{q}$. According to \eqref{eq:cuv}, the incoming density $\underline{\rho}=0$. A limit solution can be introduced by setting $\mathsf{S}=\mathsf{W}$ and combining \eqref{eq:VelocityWallCondition} and \eqref{eq:RHCondition} to determine the downstream flow field. This solution addresses the problem in the distributional sense \cite[Section 11.1.1]{Evans2010GSM} and can be defined for more general obstacles. It provides the simplest example of simultaneously containing above three features. On the other hand, the shock-free boundary is often treated using the continuity method. Starting with a known specific solution, a general solution is derived through {\it a priori} estimates and continuous variations, see \cite{BaeXiang2023AAM,ChenFeldman2018AMS,ChenFeldmanXiang2020ARMA,EllingLiu2008CPAM}. Following this approach, we aim to use this limit solution as the starting point for the continuity method to construct a general solution that simultaneously incorporates these three features. In fact, we intend to construct a solution where $\underline{u}$ is sufficiently close to $\bar{q}$, treating the limit solution as the asymptotic state. At this point, apart from at the vertex, the shock wave detaches from the solid wall. Moreover, this constructed solution should asymptotically approach the limit solution as $\underline{u}$ approaches $\bar{q}$.

In \cite{HuZhang2019SIAM}, the authors introduce a limit case for this problem by considering the incoming flow reaching a limit speed $\underline{u}=\bar{q}$. According to \eqref{eq:cuv}, this implies that the incoming density $\underline{\rho}=0$. A limit solution can be obtained by setting $\mathsf{S}=\mathsf{W}$ and combining \eqref{eq:VelocityWallCondition} and \eqref{eq:RHCondition} to determine the downstream flow field. This solution addresses the problem in the distributional sense \cite[Section 11.1.1]{Evans2010GSM} and can be formulated for more general obstacles. It provides the simplest example that simultaneously contains the three aforementioned features. On the other hand, the shock-free boundary is often tackled using the continuity method, as seen in \cite{BaeXiang2023AAM,ChenFeldman2018AMS,ChenFeldmanXiang2020ARMA,EllingLiu2008CPAM}. Starting with a known specific solution, a general solution is derived through {\it a priori} estimates and continuous variations. Following this approach, we aim to use the limit solution as the starting point for the continuity method to construct a general solution that simultaneously incorporates these three features for $\underline{u}<\bar{q}$. Specifically, we intend to construct a solution where $\underline{u}$ is sufficiently close to $\bar{q}$, treating the limit solution as the asymptotic state. At this stage, apart from at the vertex, the shock wave detaches from the solid wall. Furthermore, this constructed solution should asymptotically approach the limit solution as $\underline{u}$ approaches $\bar{q}$.

For the issue of shock wave dissipation, there are some examples for scalar conservation laws in one space dimension, see \cite[Section 3.4.1 Example 3]{Evans2010GSM}. However, regarding conservation laws of system, such issues are rarely studied. By comparison, there are numerous results regarding singularity formation, see \cite{ChenChenZhu2021SIAM,ChenPanZhu2017SIAM,ChenYoungZhang2013JHDE,Christodoulou2022JMP,ChristodoulouMiao2014SMM,John1974CPAM,Kong2002TAMS,Lax1964JMP,Liu1979JDE,MiaoYu2017IM,Nakane1988SIAM,Rammaha1989PAMS}.

Some efforts have been made mathematically for the hyperbolic flow past obstacles. So far, tangible results have been achieved primarily by studying special flow patterns. The limiting hyperbolic Euler flow, as Mach number goes to infinity in the case that the adiabatic index of the flow $\gamma$ goes to one, past a straight wedge was studied in Qu-Yuan-Zhao \cite{QuYuanZhao2020ZAMM} by proposing a definition of solution in the frame of Radon measure, namely, Radon measure solution, and the Newton' sine squared laws in hyperbolic flow was firstly proved. Later, the Newton-Busemann pressure law and their generalized versions have been derived for the limiting flow past general wedges and cones in \cite{JinQuYuan2021CPAA} \cite{JinQuYuan2023CAMC} \cite{QuYuan2023CAMB}. In these works, the wedge angle and the cone angle are fixed, and the shock layer is singular as it degenerate to a curve or a surface with mass and thus momentum as well as energy concentrated on it. These solutions are useful since they can be adapted to relatively simple boundary conditions. The hypersonic similarity was first discovered by Tsien \cite{Tsien1939JSR} for steady potential flow. For a fixed similarity parameter defined by the product of the Mach number and the wedge angle, the flow are governed approximately by the same equation. Recently, Kuang-Xiang-Zhang proved the validation of this in \cite{KuangXiangZhang2020AIPAN} for potential flow past a slender wedge by proving that the shock solution structures after scaling are consistent. The convergence of the approximation was studied in \cite{KuangXiangZhang2023CVPDE} and the hypersonic similarity for full Euler equations is discussed in \cite{ChenKuangXiangZhang2024AM}.

Our proof originates from a fundamental idea. Under the assumption of an invariant Bernoulli constant, as the incoming flow continuously accelerates, approaching the limit speed, the corresponding density of the gas will decay to zero. Consequently, the associated mass flux diminishes as well. In this context, we can consider an asymptotic state defined as a limit solution, where the shock front coincides with the solid wall. The flow field after the shock is defined through a combination of the Rankine-Hugoniot conditions and the solid wall boundary conditions (see Figure \ref{fig:convexwedgeS}). Then, our estimates rely on three fundamental steps. First, owing to the conservation of mass and the entropy condition, we can control the distance between the shock front and the solid wall. This provides us with an estimate of the shock (see Theorem \ref{thm:NarrowEstimate}). Second, under the assumption that the derivatives of the flow field are uniformly bounded, as the shock front approaches the solid wall, the flow field's velocity can be approximated by the velocities obtained through the combination of the shock conditions and the solid wall conditions. This approximation allows for an estimate of the flow field (see Theorem \ref{thm:SpecialSolutionPerturbationEstimate}). Third, as we approach the limit solution, an analysis of the characteristics' reflection provides us with an estimate for the derivatives of the solution (see Theorem \ref{thm:SonicCharacteristicBoundInductive}). We point out here that the first two steps are directly reflected in the definition of limit solution, which is the logic of conservation laws and has not appeared in previous proofs. The third logic stems from the equation.

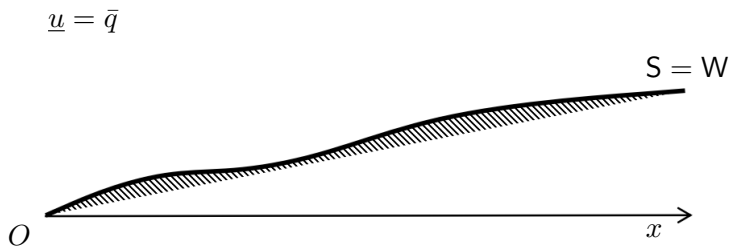
\begin{figure}[htbp]
\setlength{\unitlength}{1bp}
\begin{center}
\begin{picture}(190,190)
\put(-60,-10){\includegraphics[scale=0.4]{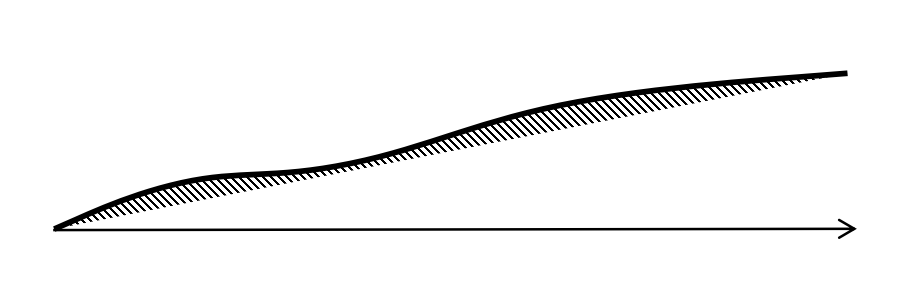}}
\put(-43,84){\small$\underline{u}=\bar{q}$}
\put(180,65){\small$\textsf{S}=\textsf{W}$}
\put(-58,2){\small$O$}
\put(180,5){\small$x$}
\end{picture}
\end{center}
\caption{flow with limit speed onto a curved wedge}\label{fig:convexwedgeS}
\end{figure}

%这个讨论源于一个基本的观点。在假定不变的伯努利常数的前提下，如果入射流持续加速，接近极限速度，对应的气体密度将衰减为零。因此，相应的质量通量也将衰减为零。以这种方式，我们可以考虑一个渐近状态，定义为极限解，此时的激波前与实体壁面重合。激波后的流场通过联合应用Rankine-Hugoniot条件和实体壁面边界条件在实体壁面的边界上定义。从某种意义上说，这个极限解，在分布意义上，是满足条件A的一个守恒定律解。

%我们的估计来源于三个基本逻辑。首先由于质量守恒和熵条件，我们可以控制激波到固壁的距离，从而得到激波的估计。第二，在流场导数一致有界的假设下，随着激波和固壁越来越近，流场的速度近似等于激波条件和固壁条件的联立，从而得到速度的估计。第三，如果解靠近于极限解，我们可以通过特征反射，得到导数的先验估计。需要主意的是，这里前两个逻辑在极限解的构造中有所体现，来自于守恒律的逻辑。而第三个逻辑则是来自于方程。

We need to analyse a solution with large variations, in the presence of both free and fixed boundaries. Due to the complexity of the objects we are calculating, the presentation of our calculations is crucial. Our approach relies on two key techniques. The first is characteristic decomposition. In reference \cite{LiZhangZheng2006CMP}, the authors introduce the technique of characteristic decomposition for quasi-linear hyperbolic systems, clearly demonstrating the differential relationships among various quantities of the system, and further constructing global solutions for rarefaction wave interactions that occur in high-dimensional Riemann problems, see \cite{ChenZheng2010IUMJ,HuLi2020ARMA,Lai2019IUMJ,LaiSheng2016SIAM,LaiSheng2021ARMA,LiXiao2012JDE,LiYangZheng2011JDE,LiZheng2009ARMA,LiZheng2010CMP,LiZheng2011ARMA}. We employ their technique to investigate the equations and their corresponding boundary relations. The second technique is the application of symbolic computing software Mathematica. As the incoming and post-shock flows near the limit speeds, due to the significant singularity of various quantities at this time, the calculation is difficult. In this part, we only consider the case where the adiabatic constant is $\gamma=2$. There are two benefits to doing this. First, in this case, the state equation, Bernoulli's equation, and the Rankine-Hugoniot conditions have polynomial forms, which streamline the calculations. Second, it ensures that various quantities maintain good differentiability near the limit speed. Utilizing computational software Mathematica under these conditions, we are well-equipped to derive the desired computation. We believe that the methods discussed herein can be further extended to yield more general results.

The rest of this paper is organized as follows. In Section \ref{sect:BasicDifferentialRelations}, we first decompose both the system and boundary conditions by making use of the characteristics and reformulate the problem into one in which $\rho$, $(u, v)$ and their derivatives are expressed by $c$ and the directive derivation along the characteristic, that is $\bar{\partial}^\pm c$. In Section \ref{sect:LimitSolution}, the limit solution is introduced for the case $\underline{u}=\bar{q}$. A particular curve was specified by analysing the jump conditions, the Rankine-Hugoniot, that the shock satisfies. Estimates are obtained here for the states curve $(u, v)$ that could be connected to a given hypersonic flow by a shock by analysing the continuity of a family of parametric curves with parameter relating to the Mach number of the incoming flow. In Section \ref{subsect:SomeaprioriEstimates}, we present some a priori estimates for the solution, whose states locate around the limit solution. The proof of Theorem \ref{thm:MainTheorem1} is given in Section \ref{sect:GlobalShockwithoutConvexity}. We establish the a priori estimates for the flow with boundary of {\bf Case A} in this section and get the global existence of solution. The asymptotic behaviour of the solution is analysed in Section \ref{subsect:AsymptoticBehaviour}. Furthermore, we analyse the higher order derivatives of the solution and obtain their limit as the Mach number of the incoming flow goes to infinity in Section \ref{subsect:HigherOrderProperty}. Finally, for {\bf Case B}, the proof of Theorem \ref{thm:MainTheorem2} is given in Section \ref{sect:ConvexWedge}. In Section \ref{subsect:SignofShockStates}, we present the delicate states estimates around the shock polar. Then, under the convex assumption, we show the sign estimates of the solution and prove the global existence. The shock is proved extinguished at infinity for $f_\infty=0$.

\section{Basic differential relations}\label{sect:BasicDifferentialRelations}
In this section, we introduce characteristic decomposition. Specifically, we consider the differentiation of various quantities along characteristic directions, as well as the higher-order dependencies of these derivatives. Then, we consider the differential relations derived from the boundary conditions. The idea here comes from a series of works in \cite{LiZhangZheng2006CMP,LiYangZheng2011JDE,ChenZheng2010IUMJ}.
\subsection{Preliminary}\label{subsect:Preliminary}
For the states along the shock, we can eliminate $\phi'$ in \eqref{eq:RHCondition} to have the shock polar equation
\begin{equation*}
(\rho u-\underline{\rho}\underline{u})(u-\underline{u})+\rho v^2=0.
\end{equation*}
Set the incoming mass flux
\begin{equation}\label{eq:epsilon}
\epsilon:=\underline{\rho}\underline{u}.
\end{equation}
Owing to \eqref{eq:BernoulliLaw}, for $\underline{u}>\underline{c}$, $\epsilon$ is a decreasing function of $\underline{u}$. Hence, for fixed Bernoulli constant $\mathbf{B}$, we have
\begin{equation*}
\underline{u}\rightarrow\bar{q},
\end{equation*}
and
\begin{equation*}
\underline{\mathbf{M}}\rightarrow+\infty,
\end{equation*}
as
\begin{equation*}
\epsilon\rightarrow0+.
\end{equation*}
In view of above, we have
\begin{prop}
Theorem \ref{thm:MainTheorem1} and Theorem \ref{thm:MainTheorem2} hold if and only if their conclusions are valid for sufficiently small $\epsilon$.
\end{prop}

Combining \eqref{eq:epsilon} with \eqref{eq:cuv}, for $\underline{u}>\underline{c}$, we can solve $\underline{u}$ and $\underline{\rho}$ as functions of $\epsilon$. The states $(u, v)$ along the shock should satisfy
\begin{equation}\label{eq:ShockPolar}
G(\vec{U}, \epsilon)=0,
\end{equation}
with
\begin{equation}\label{eq:ShockPolarGExpression}
G(\vec{U}, \epsilon):=(u-\frac{\epsilon}{\rho})(u-\underline{u}(\epsilon))+v^2.
\end{equation}
In above, we can solve $\epsilon$ uniquely from $(u, v)$ for $\epsilon<<1$. To see this, we differentiate above with respect to $\epsilon$ to have
\begin{equation}\label{eq:Ge}
G_\epsilon=-\frac{\underline{u}-u}{\rho}-\underline{u}'\qnt{\frac{\epsilon}{\rho}-u}.
\end{equation}
By entropy condition, we have
\begin{equation*}
\underline{u}-u>0.
\end{equation*}
By \eqref{eq:ShockPolar}, we have
\begin{equation*}
u-\frac{\epsilon}{\rho}=-\frac{v^2}{u-\underline{u}}>0.
\end{equation*}
It can be directly verified that
\begin{equation*}
\underline{u}'<0,
\end{equation*}
for $\epsilon<<1$ and $\bar{q}-\underline{u}<<1$. These three indicates $G_\epsilon<0$, so $\epsilon$ can be solved as a function of $(u, v)$. 

In the following discussion, sometimes for the sake of convenience, we list both $\vec{U}$ and $\epsilon$, but in fact, they share a constraint condition \eqref{eq:ShockPolar}.
\subsection{Characteristic decomposition}\label{subsect:CharacteristicDecomposition}
If the flow field is $C^1$ smooth, then we can rewrite \eqref{eq:EulerEquations} in the following non-conservation form
\begin{equation}\label{eq:EulerEquationsNC}
\begin{cases}
(c^2-u^2)u_x-uv(u_y+v_x)+(c^2-v^2)v_y=0;\\
u_y-v_x=0,
\end{cases}
\end{equation}
or the matrix form
\begin{equation}\label{eq:EulerEquationsM}
\begin{pmatrix}u\\v\end{pmatrix}_x + \begin{pmatrix}\frac{-2uv}{c^2-u^2}& \frac{c^2-v^2}{c^2-u^2}\\-1& 0\end{pmatrix}\begin{pmatrix}u\\v\end{pmatrix}_y=0.
\end{equation}
If $u^2+v^2>c^2$, then \eqref{eq:EulerEquationsM} is a hyperbolic system with two genuinely nonlinear characteristics
\begin{equation}\label{eq:lambda}
\lambda_\pm(u, v):=\frac{uv \pm c\sqrt{u^2+v^2-c^2}}{u^2-c^2}.
\end{equation}
Multiply \eqref{eq:EulerEquationsM} with $l_\pm=(1,\lambda_\mp)$ on the left and rewriting it in the characteristic forms yields
\begin{equation}\label{eq:EulerEquationsC}
\partial^\pm u +\lambda_\mp \partial^\pm v=0,
\end{equation}
where $\partial^\pm=\partial_x+\lambda^\pm\partial_y$. Since the expressions \eqref{eq:lambda} cannot be defined for the case $u=c$, for convenience, in the following discussion, we normalize the derivatives $\partial^\pm$ by defining $\bar{\partial}^\pm$ and introduce the angles $\alpha, \beta, \omega, \tau$ and some constants as follows
\begin{equation}\label{eq:Angles}
\begin{split}
&\bar{\partial}^+:=\cos\alpha\partial_x+\sin\alpha\partial_y,\quad \bar{\partial}^-:=\cos\beta\partial_x+\sin\beta\partial_y,\quad \bar{\partial}^0:=\cos\tau\partial_x+\sin\tau\partial_y,\\
&\tan\alpha:=\lambda_+,\quad\tan\beta:=\lambda_-,\quad\omega:=(\alpha-\beta)/2,\quad\tau:=(\alpha+\beta)/2,\\
&\nu:=(\gamma+1)/[2(\gamma-1)],\quad\Xi:=m-\tan^2\omega,\quad m:=(3-\gamma)/(\gamma+1),\quad \kappa:=\frac{\gamma-1}{2},
\end{split}
\end{equation}
where $\omega$ is the Mach angle and $\tau$ is the velocity angle. In this paper, we have the angle ranges as follows
\begin{equation}\label{eq:AnglesRange}
\alpha\in[0, \pi], \beta\in[-\frac{\pi}{2}, \frac{\pi}{2}], \omega\in(0, \frac{\pi}{2}), \tau\in(0, \frac{\pi}{2}).
\end{equation}
Then, \eqref{eq:EulerEquationsC} can be rewritten as
\begin{equation}\label{eq:EulerEquationsCN}
\bar{\partial}^\pm u +\lambda_\mp \bar{\partial}^\pm v=0.
\end{equation}

We have the following two lemmas, which can be found in \cite{Hu2018JMAA}.
\begin{lem}[Characteristic expressions]
For the $C^1$ solution of system \eqref{eq:EulerEquationsCN}, we have the characteristic expressions as follows
\begin{equation}\label{eq:SonicExpressions}
\begin{cases}
c\bar{\partial}^+\beta=-2\nu\tan\omega\bar{\partial}^+c,\quad c\bar{\partial}^+\alpha=-\nu\Xi\sin(2\omega)\bar{\partial}^+c,\\
c\bar{\partial}^-\alpha=2\nu\tan\omega\bar{\partial}^-c,\quad c\bar{\partial}^-\beta=\nu\Xi\sin(2\omega)\bar{\partial}^-c,\\
c\bar{\partial}^\pm\omega=\Big(\kappa^{-1}\sin^2\omega+1\Big)\tan\omega\bar{\partial}^\pm c,\\
\bar{\partial}^+u=\kappa^{-1}\sin\beta\bar{\partial}^+c,\quad \bar{\partial}^+v=-\kappa^{-1}\cos\beta\bar{\partial}^+c,\\
\bar{\partial}^-u=-\kappa^{-1}\sin\alpha\bar{\partial}^-c,\quad \bar{\partial}^-v=\kappa^{-1}\cos\alpha\bar{\partial}^-c,\\
c\bar{\partial}^+\tau=-\frac{\sin(2\omega)}{2\kappa}\bar{\partial}^+c,\quad c\bar{\partial}^-\tau=\frac{\sin(2\omega)}{2\kappa}\bar{\partial}^-c.
\end{cases}
\end{equation}
\end{lem}

\begin{lem}[Characteristic decompositions]
For the $C^2$ solution of the system \eqref{eq:EulerEquationsCN}, we have the characteristic decompositions as follows
\begin{equation}\label{eq:CharacteristicDecomposition}
\begin{cases}
c\bar{\partial}^-\bar{\partial}^+c
=\bar{\partial}^+c\Big( \nu(1+\tan^2\omega)\bar{\partial}^+c+\frac{\nu(\tan^2\omega-1)^2+2\tan^2\omega}{\tan^2\omega+1}\bar{\partial}^-c \Big),\\
c\bar{\partial}^+\bar{\partial}^-c
=\bar{\partial}^-c\Big( \nu(1+\tan^2\omega)\bar{\partial}^-c+\frac{\nu(\tan^2\omega-1)^2+2\tan^2\omega}{\tan^2\omega+1}\bar{\partial}^+c \Big).
\end{cases}
\end{equation}
\end{lem}

\subsection{Differential relations on the wedge boundary.}\label{subsect:DifferentialRelationW}
Denote by $\kappa_w$ the curvature of the wall $\mathsf{W}$. Then, in view of \eqref{eq:VelocityWallCondition}, we have
\begin{equation}\label{eq:WallCurvature}
\kappa_w(x, f(x)):=\frac{f''}{(1+(f')^2)^{3/2}}=\bar{\partial}^0\tau.
\end{equation}
Since $\bar{\partial}^0\tau=\frac1{2\cos\omega}\qnt{\bar{\partial}^+\tau+\bar{\partial}^-\tau}$, owing to \eqref{eq:SonicExpressions} and by the fact that $c=q\sin\omega$, we have
\begin{lem}
On the wedge boundary $\mathsf{W}$, it holds the differential relation as follows
\begin{equation}\label{eq:WallRelationCharacteristic}
\bar{\partial}^-c-\bar{\partial}^+c=(\gamma-1)q\kappa_w \quad \mbox{on } \quad \mathsf{W}.
\end{equation}
\end{lem}

\subsection{Differential relations on the shock fronts.}\label{subsect:DifferentialRelationS}
Next, we intend to derive the relation of $\bar{\partial}^\pm c$ on the shock $y=\phi(x)$ by \eqref{eq:ShockPolar}. By \eqref{eq:RHCondition}, define the angle of shock $s$ as follows
\begin{equation}\label{eq:sExpression}
s(\vec{U}, \epsilon):=\arctan \phi'(x)=\arctan\frac{\rho v}{\rho u-\epsilon}.
\end{equation}
Then the derivative along the shock is defined by
\begin{equation}\label{eq:sderivative}
\bar{\partial}^s=\cos s\partial_x+\sin s\partial_y.
\end{equation}
A directive calculation yields
\begin{equation}\label{eq:sderivativet}
\bar{\partial}^s=t_+\bar{\partial}^++t_-\bar{\partial}^-,
\end{equation}
where
\begin{equation}\label{eq:sderivativetexpression}
t_+=\frac{\sin(\beta-s)}{\sin(\beta-\alpha)}\quad\mbox{and}\quad t_-=\frac{\sin(s-\alpha)}{\sin(\beta-\alpha)}.
\end{equation}

Noting that it holds \eqref{eq:ShockPolar} along the shock, we can apply $\bar{\partial}^s$ on it to yield
\begin{equation}\label{eq:RHConditionS}
G_u\bar{\partial}^su+G_v\bar{\partial}^sv=0\quad\mbox{on}\quad \mathsf{S}.
\end{equation}
Combining \eqref{eq:SonicExpressions} and \eqref{eq:sderivative} gives
\begin{equation*}
\begin{cases}
\bar{\partial}^su=t_+\bar{\partial}^+u+t_-\bar{\partial}^-u=
\kappa^{-1}(t_+\sin\beta\bar{\partial}^+c-t_-\sin\alpha\bar{\partial}^-c),\\
\bar{\partial}^sv=t_+\bar{\partial}^+v+t_-\bar{\partial}^-v=
-\kappa^{-1}(t_+\cos\beta\bar{\partial}^+c -t_-\cos\alpha\bar{\partial}^-c),
\end{cases} \mbox{on }\quad \mathsf{S}.
\end{equation*}
Inserting above into \eqref{eq:RHConditionS} yields the following equation of $\bar{\partial}^\pm c$
\begin{equation}\label{eq:ShockRelation1}
t_+(G_u\sin\beta-G_v\cos\beta)\bar{\partial}^+c+t_-(-G_u\sin\alpha+G_v\cos\alpha)\bar{\partial}^-c=0, \quad {\rm on}\quad \mathsf{S}.
\end{equation}

Let
\begin{equation}\label{eq:kexpression}
\tan k:=-\frac{G_u}{G_v},
\end{equation}
where $k=k(\vec{U}, \epsilon)$ is the angle of the tangent line at the point $(u, v)$ on the shock polar \eqref{eq:ShockPolar} and belongs to $[-\frac{\pi}{2}, 0]$ in this paper. Inserting it into \eqref{eq:ShockRelation1}, we have
\begin{equation*}
\sin(\beta-s)\cos(\beta-k)\bar{\partial}^+c+\sin(\alpha-s)\cos(\alpha-k)\bar{\partial}^-c=0\quad \mbox{on} \quad \mathsf{S}.
\end{equation*}
Then we immediately arrive at the following conclusion.
\begin{lem}
It holds on the shock $\mathsf{S}$ the following differential relations
\begin{equation}\label{eq:ShockRelation}
\bar{\partial}^+c=g(\vec{U}, \epsilon)\bar{\partial}^-c,\quad \mbox{on} \quad \mathsf{S},
\end{equation}
where
\begin{equation}\label{eq:gexpression}
g(\vec{U}, \epsilon):=\frac{\sin(s-\alpha)\cos(k-\alpha)}{\sin(\beta-s)\cos(k-\beta)}.
\end{equation}
\end{lem}
$g(\vec{U}, \epsilon)$ is continuous with respect to both $\vec{U}$ and $\epsilon$ under the constrain \eqref{eq:ShockPolar}. As previously, we use the notation $g=g(\vec{U}, \epsilon)$ in the sequel without confusion. Therefore, Problem 2 induces the derivatives relations \eqref{eq:CharacteristicDecomposition}, \eqref{eq:WallRelationCharacteristic} and \eqref{eq:ShockRelation}.

\section{A limit case}\label{sect:LimitSolution}
In this section, we introduce the limit solution for the case where $\underline{u}=\bar{q}$ ($\underline{\rho} \underline{u}=\epsilon=0$) and present some useful estimates. This solution is first introduced in \cite{Hu2018JMAA,HuZhang2019SIAM} to consider the hypersonic flow past a bend obstacle.
\subsection{The limit solution}
For the limit speed $\bar{q}$ in \eqref{eq:BernoulliLaw}, since $1<\gamma<3$, for the corresponding density $\bar{\rho}$, pressure $\bar{p}$ and sonic $\bar{c}$, we have
\begin{equation}\label{eq:LimitStates}
\bar{\rho}=\bar{p}=\bar{c}=0.
\end{equation}
The Mach number 
\begin{equation}
\overline{\mathbf{M}}=+\infty.
\end{equation}
In view of above, if the incoming flow has a limit speed $\underline{u}=\bar{q}$, then the R-H condition \eqref{eq:RHCondition} has the form
\begin{equation}\label{eq:ShockPolarS}
G(\vec{U}, 0)=u(u-\bar{q})+ v^2=0.
\end{equation}
On the $(u, v)$ plane, it is a circle passing through the origin and centered at $(\frac{\bar{q}}{2}, 0)$. For such a case, we have many simplifications. First, for the shock $y=\phi(x)$, since $\underline{\rho}=\rho(\bar{q})=\bar{\rho}=0$, it yields
\begin{equation*}
\phi'=\frac{[\rho v]}{[\rho u]}=\frac{\rho v-\underline{\rho}\underline{v}}{\rho u-\underline{\rho}\underline{u}}=\frac{\rho v-0}{\rho u-0}=\frac{v}{u}.
\end{equation*}
This leads to
\begin{equation}\label{eq:stauS}
s(\vec{U}, 0)=\tau(\vec{U}).
\end{equation}

We can construct a limit solution as follows
\begin{defn}
We define a nonzero flow field $(u_s, v_s)$ along $\mathsf{W}$ by combining \eqref{eq:VelocityWallCondition} and \eqref{eq:ShockPolarS} directly as follows
\begin{equation}\label{eq:LimitSolution}
(u_s, v_s)(x, f(x)):=\qnt{\frac{\bar{q}}{1+(f')^2}, \frac{\bar{q}f'}{1+(f')^2}}.
\end{equation}
Meanwhile, letting $\mathsf{S}$ coincide with $\mathsf{W}$ directly, we define the limit solution for {\bf Problem 2} by $(u_s, v_s)$ and $\mathsf{S}$. For this solution, the domain $\Omega$ degenerates to the curve $\mathsf{W}$. Moreover, we have 
\begin{equation}\label{eq:SpeedS}
q_s=\sqrt{u_s^2+v_s^2}=\frac{\bar{q}}{\sqrt{1+(f')^2}},~ c^2_s=\frac{\gamma-1}{2}\frac{f'^2}{1+f'^2}\bar{q}^2,~ \rho_s=\qnt{\frac{\gamma-1}{2\gamma}\frac{f'^2}{1+f'^2}\bar{q}^2}^{\frac{1}{\gamma-1}}.
\end{equation}
Therefore, it can be easily deduced that the flow $(u_s, v_s)$ is subsonic if
\begin{equation*}
f'>\sqrt{\frac{2}{\gamma-1}}\qnt{\mbox{or\ } u_s<\frac{\gamma-1}{\gamma+1}\bar{q}};
\end{equation*}
it is supersonic if
\begin{equation*}
0\leq f'<\sqrt{\frac{2}{\gamma-1}}\qnt{\mbox{or\ }u_s>\frac{\gamma-1}{\gamma+1}\bar{q}};
\end{equation*}
and it can not be defined if $f'<0$.
\begin{rem}
For the limit case $\underline{u}=\bar{q}$, we can define another interesting limit solution
\begin{equation}\label{eq:LimitSolutionBX}
(u_o, v_o):=(0, 0),
\end{equation}
with the shock $\mathsf{S}$ defined as follows
\begin{equation}\label{eq:LimitSolutionShockBX}
x=-S_o
\end{equation}
where $S_o>0$ is any positive constant. Based on this limit solution, Bae and Xiang establish the global existence of a detached transonic shock for a supersonic flow onto a blunt body in \cite{BaeXiang2023AAM}.
\end{rem}
\end{defn}
\subsection{Differential relations for the limit solution}
For the limit case, we calculate the equation \eqref{eq:ShockRelation}. First, by \eqref{eq:stauS}, recalling \eqref{eq:sderivativet}, we have
\begin{equation}\label{eq:ShockRelationT}
t_\pm=\frac{\sin(-\omega)}{\sin(-2\omega)}=\frac{\sin\omega}{\sin(2\omega)}=\frac{1}{2\cos\omega}>0.
\end{equation}
Moreover, the shock polar \eqref{eq:ShockPolar} becomes the circle \eqref{eq:RHConditionS}, and we have 
\begin{equation}\label{eq:kLimit}
k=2\tau-\frac{\pi}{2}.
\end{equation} 
Therefore, combine above two, we have 
\begin{lem}
For the limit case $\epsilon=0$, \eqref{eq:ShockRelation} has the following form
\begin{equation}\label{eq:ShockRelationS}
\bar{\partial}^+c=\frac{\sin\beta}{\sin\alpha}\bar{\partial}^-c=g(\vec{U}, 0)\bar{\partial}^-c,
\end{equation}
where $(u, v)$ locates on the shock polar \eqref{eq:ShockPolarS}.
\end{lem}
In order to estimate $g(\vec{U}, \epsilon)$, we need the calculation for $g(\vec{U}, 0)=\frac{\sin\beta}{\sin\alpha}$ as follows
\begin{thm}\label{thm:gRangeS}
For $(u, v)$ on the shock polar \eqref{eq:ShockPolarS},

(1) If $v\geq0$ and 
\begin{equation}\label{eq:gEstimateF}
\frac{\gamma-1}{\gamma+1}\bar{q}<u<\bar{q},
\end{equation}
then we have
\begin{equation}\label{eq:gEstimateS}
\sin\alpha>\abs{\sin\beta};
\end{equation}

(2) If $v\geq0$ and 
\begin{equation}\label{eq:gEstimateF1}
\frac{\gamma-1}{2}\bar{q}<u<\bar{q},
\end{equation}
then we have
\begin{equation}\label{eq:gEstimateS1}
0<\frac{\sin\beta}{\sin\alpha}\leq K<1,
\end{equation}
where $K$ is a constant.
\end{thm}

\begin{proof}
(1) First, we can directly compute that \eqref{eq:gEstimateF} and $v\geq0$ lead to
$$\tau, \omega\in(0, \frac{\pi}{2}).$$
Then, 
\begin{equation*}
\begin{cases}
\sin\alpha+\sin\beta=2\sin\tau\cos\omega>0,\\
\sin\alpha-\sin\beta=2\cos\tau\sin\omega>0,
\end{cases}
\end{equation*}
indicate that
\begin{equation*}
-\sin\alpha<\sin\beta<\sin\alpha.
\end{equation*}
(1) is proved.

(2) In order to compute the sign of $\sin\beta$, we compare $\frac{v}{q}=\sin\tau$ with $\frac{c}{q}=\sin\omega$. In fact, according to \eqref{eq:BernoulliLaw} and \eqref{eq:ShockPolarS}, we have
\begin{equation*}
v^2=u(\bar{q}-u),\quad q^2=u\bar{q},\quad\mbox{and}\quad c^2=\frac{\gamma-1}{2}\bar{q}(\bar{q}-u).
\end{equation*}
So
\begin{equation*}
\begin{split}
\frac{v^2}{q^2}-\frac{c^2}{q^2}&=\frac{(u-\frac{\gamma-1}{2}\bar{q})(\bar{q}-u)}{u\bar{q}}.
\end{split}
\end{equation*}
In view of above, $\bar{q}>u>\frac{\gamma-1}{2}\bar{q}$ yields $\frac{v}{q}>\frac{c}{q}$, which indicates $\sin\tau>\sin\omega$. Since $\omega, \tau\in(0, \frac{\pi}{2})$, we have $\frac{\pi}{2}>\tau>\omega>0$. So
\begin{equation*}
\sin\alpha>\abs{\sin\beta}=\sin\beta=\sin\qnt{\tau-\omega}>0.
\end{equation*}
Moreover, we have
\begin{equation*}
\frac{\sin^2\tau}{\sin^2\omega}=\frac{v^2}{c^2}=\frac{2}{\gamma-1}\frac{u}{\bar{q}}.
\end{equation*}
In addition, on the shock polar \eqref{eq:ShockPolarS}, as $u\rightarrow\bar{q}$, we have
\begin{equation*}
\tau, \omega\rightarrow0.
\end{equation*}
So, for $\gamma\in(1, 3)$, we can conclude that
\begin{equation*}
\lim_{u\rightarrow\bar{q}}\frac{\tau}{\omega}=\lim_{u\rightarrow\bar{q}}\frac{\sin\tau}{\sin\omega}= \lim_{u\rightarrow\bar{q}}\sqrt{\frac{2}{\gamma-1}\frac{u}{\bar{q}}}=\sqrt{\frac{2}{\gamma-1}}>1.
\end{equation*}
We finally get
\begin{equation}\label{eq:gLimitS}
\begin{split}
\lim_{u\rightarrow\bar{q}}\frac{\sin\beta}{\sin\alpha}&=\lim_{u\rightarrow\bar{q}}\frac{\sin\qnt{\tau-\omega}}{\sin\qnt{\tau+\omega}}\\
&=\lim_{u\rightarrow\bar{q}}\frac{\tau-\omega}{\tau+\omega}=\qnt{\sqrt{\frac{2}{\gamma-1}}-1}/ \qnt{\sqrt{\frac{2}{\gamma-1}}+1}\in(0,1)
\end{split}
\end{equation}
To conclude, for the function $\frac{\sin\beta}{\sin\alpha}$ defined on the shock polar \eqref{eq:ShockPolarS}, we have 

(1)$$\frac{\sin\beta}{\sin\alpha}\in(0, 1)\quad\mbox{for}\quad \frac{\gamma-1}{2}\bar{q}<u<\bar{q};$$

(2)$$\frac{\sin\beta}{\sin\alpha}=0\quad\mbox{for}\quad u=\frac{\gamma-1}{2}\bar{q};$$

(3)$$\lim_{u\rightarrow\bar{q}}\frac{\sin\beta}{\sin\alpha}=\qnt{\sqrt{\frac{2}{\gamma-1}}-1}/ \qnt{\sqrt{\frac{2}{\gamma-1}}+1}\in(0,1).$$ 
Thus, by the continuity of $\frac{\sin\beta}{\sin\alpha}$, we achieve \eqref{eq:gEstimateS1}. The proof is complete.
\end{proof}
\begin{rem}\label{rem:gRangeS}
For the limit solution $\qnt{u_s, v_s}$ in \eqref{eq:LimitSolution}, we can calculate directly that \eqref{eq:gEstimateF} and \eqref{eq:gEstimateF1} are equivalent to
\begin{equation}\label{eq:gEstimateFf}
0<f'<\sqrt{\frac{2}{\gamma-1}},
\end{equation}
and
\begin{equation}\label{eq:gEstimateF1f}
0<f'<\sqrt{\frac{3-\gamma}{\gamma-1}},
\end{equation}
respectively.
\end{rem}

Let
\begin{equation*}
\mathcal{C}_s:=\set{(u_s, v_s) | u_s\in(\frac{\gamma-1}{\gamma+1}\bar{q}, \bar{q}),~G(\vec{U}_s, 0)=0~\mbox{ and }~\frac{v_s}{u_s}\in[\underline{f}, \bar{f}]}.
\end{equation*}
It is obvious that $\mathcal{C}_s$ is compact and it is a part of the circle \eqref{eq:ShockPolarS} on the $(u, v)$ plane. 
\begin{lem}\label{lem:omegaRangeS}
If $0<\underline{f}<\bar{f}<\sqrt{\frac{2}{\gamma-1}}$, then there exist positive constants $\omega_0\in(0, \frac{\pi}{16})$, $\rho_m$ and $q_m$, for any point $(u, v)\in\mathcal{C}_s$, we have
\begin{equation*}
\alpha-\tau=\alpha-s=\tau-\beta=s-\beta=\omega\in (\omega_0, \frac{\pi}{2}-\omega_0),
\end{equation*}
\begin{equation}\label{eq:RhoLowerS}
\rho\in[\rho_m, \qnt{\frac{\gamma-1}{2\gamma}\bar{q}^2}^{\frac{1}{\gamma-1}}],
\end{equation}
and
\begin{equation}\label{eq:qRangeS}
q\in [q_m, \bar{q}],
\end{equation}
where $\omega_0, \rho_m$ and $q_m$ depend only on $\bar{f}$ and $\underline{f}$.
\end{lem}

Denote by $\mathcal{B}(P, \delta)$ the neighbourhood of point $P$ with radius $\delta$ in $\mathbb{R}^2$ space and $\mathcal{C}_s^\delta$ the $\delta$ neighbourhood of $\mathcal{C}_s$. 
\begin{lem}\label{lem:BasicWavePattern}
If $0<\underline{f}<\bar{f}<\sqrt{\frac{2}{\gamma-1}}$, then there exists $\delta_\omega>0$ such that for any $\vec{U}_s\in\mathcal{C}_s$, $\vec{U}\in \mathcal{B}((u_s, v_s), \delta_\omega)$ and $\epsilon\leq\delta_\omega$, we have
\begin{equation}\label{eq:CharacteristicPerturbation}
\abs{\alpha(\vec{U})-\alpha(\vec{U}_s)}+\abs{\beta(\vec{U})-\beta(\vec{U}_s)}+\abs{s(\vec{U})-s(\vec{U}_s)} +\abs{\tau(\vec{U})-\tau(\vec{U}_s)}\leq \frac{\omega_0}{100},
\end{equation}
and
\begin{equation}\label{eq:omegaPerturbation}
\omega(\vec{U})\in[\frac{\omega_0}{2}, \frac{\pi}{2}-\frac{\omega_0}{2}].
\end{equation}
\end{lem}

\begin{proof}
Since the functions $\alpha(u, v), \beta(u, v), \tau(u, v)$ and $s(u, v, \epsilon)$ are all continuous with respect to their arguments, it is obvious.
\end{proof}

\smallskip

\begin{lem}\label{lem:PerturbationConstants}
If $0<\underline{f}<\bar{f}<\sqrt{\frac{2}{\gamma-1}}$, then there exists a positive constant $\delta_{\bar{K}}$, which is depending on $\underline{f}$ and $\bar{f}$, to ensure the following estimate
\begin{equation}\label{eq:gRange}
\abs{g(\vec{U}, \epsilon)}\leq\bar{K}<1
\end{equation}
for all $\vec{U}\in\mathcal{C}_s^{\delta_{\bar{K}}}$ and $\epsilon<\delta_{\bar{K}}$.
\end{lem}
\begin{proof}
By the first part of Theorem \ref{thm:gRangeS}, there exists a constant $\bar{K}\in(0, 1)$, such that
\begin{equation*}
\abs{g(\vec{U}_s, 0)|_{\vec{U}_s\in \mathcal{C}_s}}\leq\frac{3\bar{K}-1}{2},
\end{equation*}
provided $0<\underline{f}<\bar{f}<\sqrt{\frac{2}{\gamma-1}}$. For any $(u, v)\in\mathcal{C}_s^{\delta_{\bar{K}}}$, similar to the proof in Lemma \ref{lem:BasicWavePattern}, we can get the conclusion by the continuity of the function $g(\vec{U}, \epsilon)$.
\end{proof}

\section{Global attached shock to a curved wedge without convexity}\label{sect:GlobalShockwithoutConvexity}

\subsection{Some a priori estimates}\label{subsect:SomeaprioriEstimates}
In this subsection, we present some basic estimates for the flow field with states around $\mathcal{C}_s$. We use the notation $M_s$ to denote the quantities that depend only on the solution $(u_s, v_s)$ and the wedge boundary $\mathsf{W}$.

\subsubsection{$(\xi, \eta)$ Coordinates}
For the convenience of discussion, we introduce a useful coordinate. It follows from \cite[Section 14.6]{gilbarg2001elliptic} that under the assumptions \eqref{eq:SupersonicWallCondition}-\eqref{eq:WallCondition2nd} on the wedge boundary, there exists a constant $M_{\mathsf{W}}>0$, such that for any point $P$ in the domain
\begin{equation*}
\Omega_0=\set{(x, y)|f(x)\leq y\leq f(x)+M_{\mathsf{W}}x; x\geq0},
\end{equation*}
there exists a unique point $P_w\in \mathsf{W}$, such that
\begin{equation*}
\abs{P_wP}=\dist(P, \mathsf{W}).
\end{equation*}
Let $\eta(P)=\abs{P_wP}$ and $\xi(P)$ be the length of the path connecting $O$ with $P_w$ along $\mathsf{W}$. Then, for any $P(x, y)\in\Omega_0$, it defines a map as follows
\begin{equation}\label{eq:eqltofxxi}
(x, y)\in\Omega_0\mapsto (\xi, \eta)\in \mathbb{R}_+\times \mathbb{R}_+.
\end{equation}
Clearly, $P_wP$ is perpendicular to $\mathsf{W}$ at $P_w$. Here and in the sequel, for any point $P$, the subscript $w$ stands for the projection of $P$ on $\mathsf{W}$, and the superscript $s$ stands for the intersection of $\mathsf{S}$ and $P_wP$. Denote the intersection of $P_wP$ and $\mathsf{S}$ by $P^s$. It is obvious that the transformation has the following property.
\begin{lem}
Under the assumptions \eqref{eq:SupersonicWallCondition} and \eqref{eq:WallCondition2nd} on the wedge boundary and in $\Omega_0$, $x$ and $\xi$ are equivalent parameters. Indeed, there exists $M_s>1$ such that
\begin{equation}\label{eq:arcparameterestimate}
\frac{\xi}{x}\in[1, M_s].
\end{equation}
\end{lem}

\subsubsection{Thickness of the shock layer}
By conservation of mass, we first have the following estimate about the distance between the shock $y=\phi(x)$ and the wedge $y=f(x)$. Recall that for any $P\in\Omega_0$, we have $P\in P^sP_w$, and $\xi\qnt{P}=\xi\qnt{P_w}=\xi\qnt{P^s}$. Estimates about the thickness of the shock layer, that is the length of the section $P^sP_w$, can be obtained as follows.

\begin{thm}[Narrow estimate]\label{thm:NarrowEstimate}
Under the assumption \eqref{eq:SupersonicWallCondition}, there exists $\delta_m>0$, if $0<\epsilon\leq\delta_m$ and the solution of the problem $(u, v)\in\mathcal{C}_s^{\delta_m}$, then we have
\begin{enumerate}
\item there exist two constants $m_0$ and $m_1$ to ensure the following estimates
\begin{equation}\label{eq:MassFluxRange}
m_0\leq\rho(\vec{U}) \vec{U}\cdot\vec{b}(\xi)\leq m_1,
\end{equation}
for $\xi\in\Real^+$;
\item 
\begin{equation}\label{eq:NarrowEstimate}
\epsilon M_1\xi(P)\leq\abs{P^sP_w}\leq \epsilon M_0\xi(P)
\end{equation}
where $M_{0,1}$ are constants depending only on $\underline{f}$ and $\bar{f}$.
\end{enumerate}
\end{thm}

\begin{proof}
(i) By the definition of $(u_s, v_s)$, its direction vector takes the form as follows
\begin{equation*}
\qnt{\cos\tau(\vec{U}_s(\xi)), \sin\tau(\vec{U}_s(\xi))}=\vec{b}(\xi).
\end{equation*}
So, in view of \eqref{eq:SupersonicWallCondition}, there exists a positive constant $b_m$, such that
\begin{equation}\label{eq:AngleLowerS}
\qnt{\cos\tau(\vec{U}_s(\xi')), \sin\tau(\vec{U}_s(\xi'))}\cdot\vec{b}(\xi)=\vec{b}(\xi')\cdot\vec{b}(\xi)\in[b_m, 1].
\end{equation}
where $\xi, \xi'\in\Real^+$. Therefore, by \eqref{eq:RhoLowerS} and \eqref{eq:qRangeS}, there exist two positive constants $m_0$ and $m_1$, such that
\begin{equation*}
\frac{m_1}{2}\geq\rho(\vec{U}_s(\xi'))\vec{U}_s(\xi')\cdot\vec{b}(\xi)=\rho(q_s(\xi'))q_s(\xi')\vec{b}(\xi')\cdot\vec{b}(\xi)\geq2m_0,
\end{equation*}
for any $\xi, \xi'\in\Real^+$. So we have
\begin{equation}\label{eq:MassFluxRangeS}
\frac{m_1}{2}\geq\rho(\vec{U}_s)\vec{U}_s\cdot\vec{b}(\xi)\geq2m_0,
\end{equation}
where $\vec{U}_s\in\mathcal{C}_s$ and $\xi\in\Real^+$. By a perturbation discussion, we could prove \eqref{eq:MassFluxRange} by \eqref{eq:MassFluxRangeS}.

(ii) Assume that $(u, v)\in\mathcal{C}_s^{\delta_m}$ is the solution of the problem, then by \eqref{eq:MassFluxRange}, we have
\begin{equation}\label{eq:MassFluxRangeP}
m_1\geq\qnt{\rho(u, v)}(P)\cdot \vec{b}(P_w)\geq m_0>0,
\end{equation}
for any $P\in \Omega_0$. Thus we obtain the conclusion of item (i).

Next, by the conservation of mass (see Figure \ref{fig:ConMass}),we have
\begin{equation}\label{eq:ConservationofMass}
\underline{\rho}\underline{u}y(P^s)=\int_{P\in P_wP^s}\rho (u, v)(P)\cdot \vec{b}(P_w)\dif\eta_P.
\end{equation}

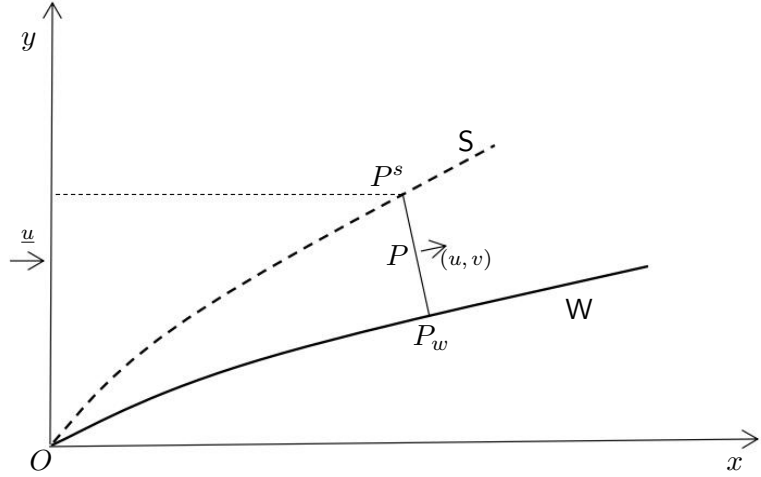
\begin{figure}[htbp]
\setlength{\unitlength}{1bp}
\begin{center}
\begin{picture}(190,190)
\put(-40,-10){\includegraphics[scale=0.54]{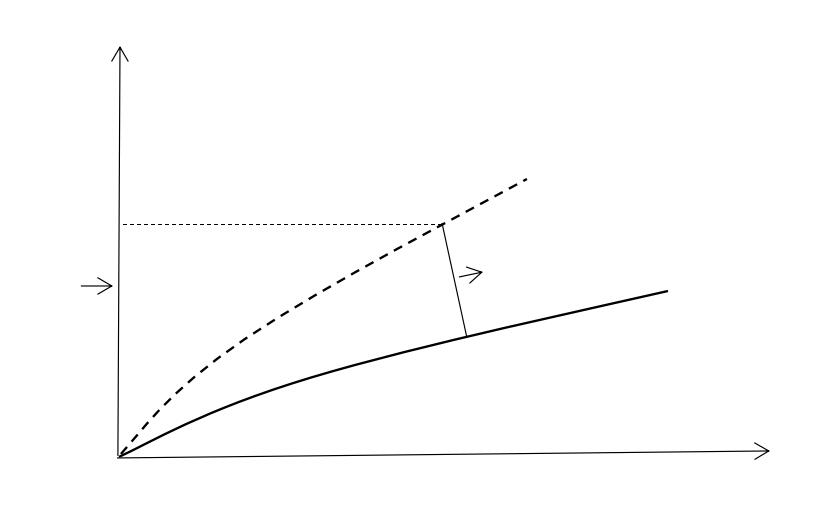}}
\put(-3,168){\small$y$}
\put(-3,95){\tiny$\underline{u}$}
\put(160,128){\small$\textsf{S}$}
\put(133,85){\small$P$}
\put(153,85){\tiny$(u, v)$}
\put(143,55){\small$P_w$}
\put(127,115){\small$P^s$}
\put(200,65){\small$\textsf{W}$}
\put(0,8){\small$O$}
\put(260,8){\small$x$}
\end{picture}
\end{center}
\caption{conservation of mass}\label{fig:ConMass}
\end{figure}

For the left-hand side of the above equality, since by Lemma \ref{lem:PerturbationConstants} the slope $\phi'(x)=\phi'(\vec{U}, \epsilon)$ of the shock-front is controlled by $\underline{\phi}$ and $\bar{\phi}$, we have on the one hand
\begin{equation*}
y(P^s)\leq \bar{\phi}x(P^s)\leq \bar{\phi}x(P_w)\leq \bar{\phi}\xi(P_w).
\end{equation*}
On the other hand,
\begin{equation*}
y(P^s)\geq y(P_w)\geq \underline{\phi}x(P_w)\geq \frac{1}{M_s}\underline{\phi}\xi(P_w).
\end{equation*}
Applying above two to the left hand term of \eqref{eq:ConservationofMass} yields
\begin{equation}\label{eq:CoordinatesInequality}
\frac{1}{M_s}\epsilon \underline{\phi} \xi(P_w)\leq\underline{\rho}\underline{u}y(P^s)\leq\epsilon \bar{\phi}\xi(P_w),
\end{equation}
where we use the arc parameter equivalency \eqref{eq:arcparameterestimate}.

For the right-hand side of \eqref{eq:ConservationofMass}, by \eqref{eq:MassFluxRangeP}, we have
\begin{equation}\label{eq:CoordinatesInequality1}
\begin{split}
m_0\abs{P_wP^s}\leq\int_{P\in P_wP^s}\rho (u, v)(P)\cdot \vec{b}(P_w)\dif\eta\leq m_1\abs{P_wP^s}.
\end{split}
\end{equation}
Thus, we can substitute \eqref{eq:CoordinatesInequality} and \eqref{eq:CoordinatesInequality1} into \eqref{eq:ConservationofMass} to get
\begin{equation*}
\frac{\epsilon\underline{\phi}}{m_1 M_s}\xi(P_w)\leq\abs{P_wP^s}\leq \frac{ \epsilon\bar{\phi}}{m_0}\xi(P_w).
\end{equation*}
Let $M_0=\frac{\bar{\phi}}{m_0}$ and $M_1=\frac{\underline{\phi}}{m_1 M_s}$, we have \eqref{eq:NarrowEstimate}. The proof is complete.
\end{proof}

\subsubsection{A priori estimates related to the characteristics in the shock layer}
In order to estimate the length of the characteristics between the shock $\mathsf{S}$ and the wall $\mathsf{W}$, we consider a standardised domain $\Omega$. It has a wall boundary $\mathsf{W}$ and a shock boundary $\mathsf{S}$. Let $P_w\in\mathsf{W}$ and $P^s\in\mathsf{S}$. We employ an arclength parameter $\xi$ on $\mathsf{W}$ with $\xi(P_w)=0$ and $\abs{P_wP^s}=1$. $\eta$ is the distance function with respect to $\mathsf{W}$. Assume a $C^-$ characteristic starting from $P^s$ intersects with a $C^+$ characteristic starting from $P_w$ at $Q$. We remind that $Q_w$ is the projection of $Q$ on the wall and $Q^s$ is the intersection of $QQ_w$ and $\mathsf{S}$. Then we have the following characteristic estimates result.

\begin{figure}[htbp]
\setlength{\unitlength}{1bp}
\begin{center}
\begin{picture}(200,180)
\put(-30,-10){\includegraphics[scale=0.7]{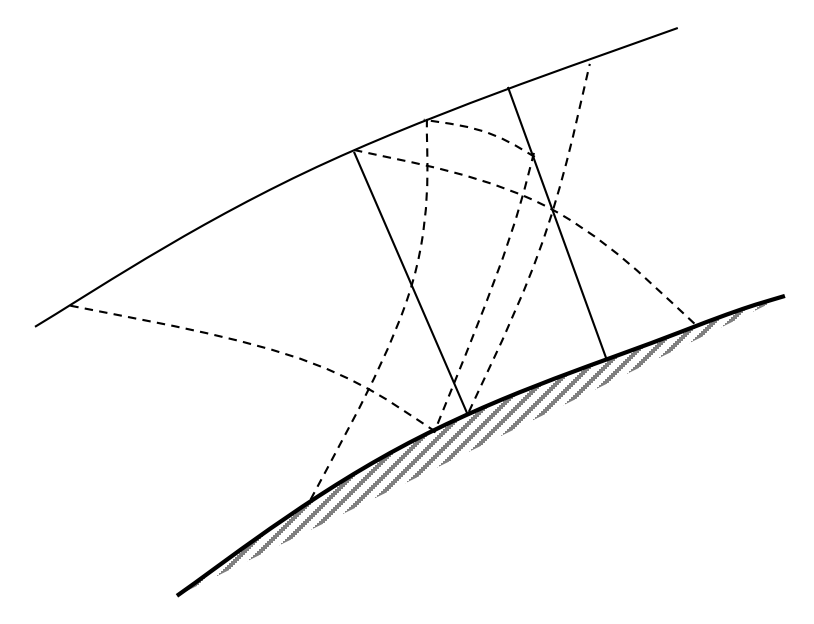}}
\put(108,44){\tiny$P_w$}
\put(94,58){\tiny$\underline{R}_1$}
\put(54,38){\tiny$\underline{R}_2$}
\put(180,168){\small$\textsf{S}$}
\put(162,101){\tiny$C^-$}
\put(145,130){\tiny$C^+$}
\put(145,164){\tiny$P^1$}
\put(150,60){\tiny$Q_w$}
\put(140,114){\tiny$Q$}
\put(131,130){\tiny$R$}
\put(120,155){\tiny$Q^s$}
\put(94,145){\tiny$\bar{R}_1$}
\put(-14,90){\tiny$\bar{R}_2$}
\put(176,85){\tiny$P_1$}
\put(200,75){\small$\textsf{W}$}
\put(69,135){\tiny$P^s$}
\end{picture}
\end{center}
\caption{the characteristics in the standardised domain}\label{fig:standardisedDomain}
\end{figure}
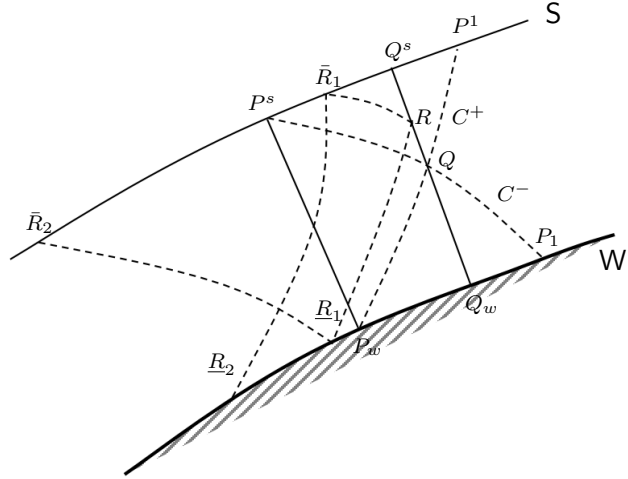

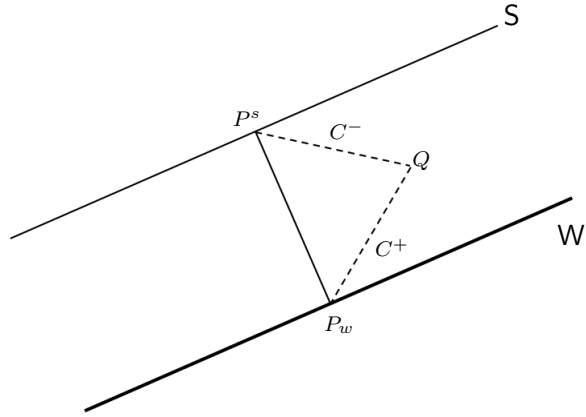
\begin{figure}[htbp]
\setlength{\unitlength}{1bp}
\begin{center}
\begin{picture}(200,180)
\put(-30,-10){\includegraphics[scale=0.7]{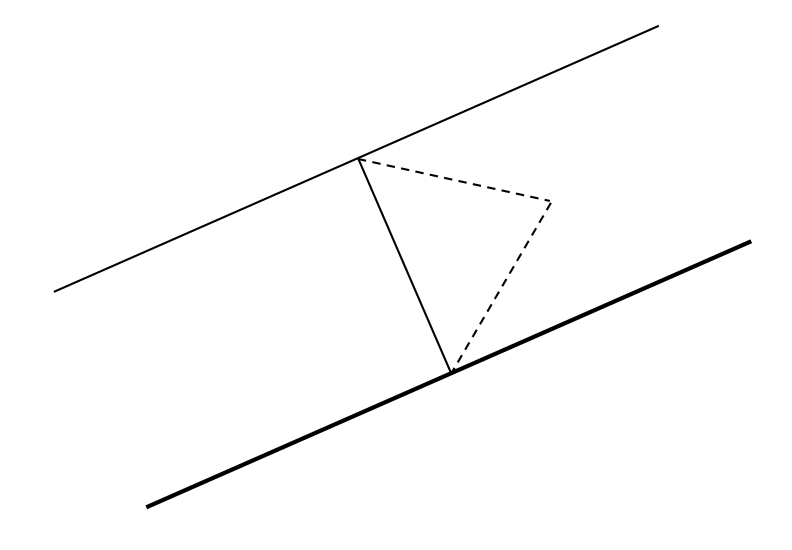}}
\put(103,33){\tiny$P_w$}
\put(170,148){\small$\textsf{S}$}
\put(122,61){\tiny$C^+$}
\put(136,95){\tiny$Q$}
\put(105,105){\tiny$C^-$}
\put(190,65){\small$\textsf{W}$}
\put(69,110){\tiny$P^s$}
\end{picture}
\end{center}
\caption{the characteristics in the standardised domain for the limit case}\label{fig:standardisedDomainS}
\end{figure}
\begin{thm}[Characteristic estimates]\label{thm:characteristicestimate}
As illustrated in Figure \ref{fig:standardisedDomain}, for any point $R$ in the domain bounded by $\widehat{P^sP_wQ_wQ^s}$, consider the $\pm$characteristics emanating backward from $R$ and the reflected parts $\widehat{R\bar{R}_1}$, $\widehat{\bar{R}_1\underline{R}_2}$, $\widehat{R\underline{R}_1}$ and $\widehat{\underline{R}_1\bar{R}_2}$. For $\delta_\omega$ defined in Lemma \ref{lem:BasicWavePattern}, if
\begin{equation}\label{eq:EstimatesO1}
(u, v)\in \mathcal{B}(\vec{U}_s, \delta_\omega),
\end{equation}
satisfies Problem 2 in $\Omega$ with $\epsilon<\delta_\omega$ for some fixed $\vec{U}_s\in\mathcal{C}_s$, then we have

(1) the length of $\widehat{R\bar{R}_1}, \widehat{\bar{R}_1\underline{R}_2}, \widehat{R\underline{R}_1}$ and $\widehat{\underline{R}_1\bar{R}_2}$ are each less than some constant $M_l$, which is only depending on $\omega_0$;

(2)
\begin{equation}
\xi(Q)\geq \frac{1}{8\tan\qnt{\frac{\pi}{2}-\omega_0}}=\frac{1}{8}\tan\omega_0;
\end{equation}

(3) $\bar{R}_2$ and $\underline{R}_2$ must locate in the backward part of $P^sP_w$ and
\begin{equation}
\xi(\bar{R}_2), \xi(\underline{R}_2)\in[-\frac{8}{\tan\omega_0}, 0].
\end{equation}
Here $\omega_0$ is the constant determined in Lemma \ref{lem:omegaRangeS}.
\end{thm}
\begin{proof}
We give an outline of the proof. For any point $\vec{U}_s\in\mathcal{C}_s$ and $\epsilon=0$, we have
\begin{equation*}
\alpha(\vec{U}_s)-\tau(\vec{U}_s)=\alpha(\vec{U}_s)-s(\vec{U}_s)=\tau(\vec{U}_s)-\beta(\vec{U}_s)=s(\vec{U}_s)-\beta(\vec{U}_s)= \omega(\vec{U}_s) \in(\omega_0, \frac{\pi}{2}-\omega_0),
\end{equation*}
and
\begin{equation*}
\tau(\vec{U}_s)=s(\vec{U}_s).
\end{equation*}
In view of above, as illustrated in Figure \ref{fig:standardisedDomainS}, for the constant states $\vec{U}_s$, the standardised domain becomes a domain bounded by a pair of parallel lines $\mathsf{W}$ and $\mathsf{S}$. The characteristics become two groups of parallel lines having an angle $\omega(\vec{U}_s)$ with $\mathsf{W}$ and $\mathsf{S}$. Thus in view of Lemma \ref{lem:omegaRangeS} and Lemma \ref{lem:BasicWavePattern}, the wave pattern in Figure \ref{fig:standardisedDomain} is only a perturbation of that in Figure \ref{fig:standardisedDomainS}. The conclusion is obvious.

Indeed, the proof is based on the Project Theorem. First, for any point $\vec{U}_s\in\mathcal{C}_s$ and $\epsilon=0$, we denote
\begin{equation*}
s_0(\vec{U}_s)=s(\vec{U}_s, 0),
\end{equation*}
and, for any point $\vec{U}\in \mathcal{B}(\vec{U}_s, \delta_\omega)$ and $\epsilon\leq \delta_\omega$, we denote
\begin{equation*}
s_\epsilon(\vec{U})=s(\vec{U}, \epsilon).
\end{equation*}

Then, for the first part, by the Project Theorem, (here abuse notation, $'\rm{ds}'$ denote for the length parameter along the curves), we have
\begin{equation}\label{eq:CharacteristicEstimates1}
\begin{split}
1=\abs{P_wP^s}=&\int_{\widehat{P_wP^1}}\sin\qnt{\alpha(\vec{U}(P))-\tau(\vec{U}_s(P_w))}\rm{ds}\\
&-\int_{\widehat{P^sP^1}}\sin\qnt{s_\epsilon(\vec{U}(P))-\tau(\vec{U}_s(P_w))}\rm{ds}\\
&\begin{cases}
\geq\abs{\widehat{P_wP^1}}\sin\frac{\omega_0}{2}-\abs{\widehat{P^sP^1}}\sin\frac{\omega_0}{50};\\
\leq\abs{\widehat{P_wP^1}}\sin2\omega_0+\abs{\widehat{P^sP^1}}\sin\frac{\omega_0}{50},
\end{cases}
\end{split}
\end{equation}
and
\begin{equation}\label{eq:CharacteristicEstimates2}
\begin{split}
0=&\int_{\widehat{P_wP^1}}\cos\qnt{\alpha(\vec{U}(P))-\tau(\vec{U}_s(P_w))}\rm{ds}\\
&-\int_{\widehat{P^sP^1}}\cos\qnt{s_\epsilon(\vec{U}(P))-\tau(\vec{U}_s(P_w))}\rm{ds}\\
&\begin{cases}
\leq\abs{\widehat{P_wP^1}}\cos\frac{\omega_0}{2}-\abs{\widehat{P^sP^1}}\cos\frac{\omega_0}{50};\\
\geq\abs{\widehat{P_wP^1}}\cos2\omega_0-\abs{\widehat{P^sP^1}},
\end{cases}
\end{split}
\end{equation}
where, by the perturbation assumption \eqref{eq:CharacteristicPerturbation} and Lemma \ref{lem:omegaRangeS}, we apply the estimates
\begin{equation*}
\alpha(\vec{U}(P))-\tau(\vec{U}_s(P_w))\in[\frac{\omega_0}{2}, 2\omega_0],
\end{equation*}
and
\begin{equation*}
\abs{s_\epsilon(\vec{U}(P))-\tau(\vec{U}_s(P_w))}=\abs{s_\epsilon(\vec{U}(P))-s_0(\vec{U}_s(P_w))}\leq\frac{\omega_0}{100}
\end{equation*}
to get the inequalities.

By the first of \eqref{eq:CharacteristicEstimates2}, we have
\begin{equation*}
\abs{\widehat{P_wP^1}}\geq\abs{\widehat{P^sP^1}}\frac{\cos\frac{\omega_0}{50}}{\cos\frac{\omega_0}{2}}.
\end{equation*}
Insert above into the first of \eqref{eq:CharacteristicEstimates1}, we have
\begin{equation*}
\abs{\widehat{P^sP^1}}\leq1/\qnt{\frac{\cos\frac{\omega_0}{50}\sin\frac{\omega_0}{2}}{\cos\frac{\omega_0}{2}}-\sin\frac{\omega_0}{50}}\leq\frac{1}{\sin\frac{\omega_0}{4}}.
\end{equation*}
In view of the second of \eqref{eq:CharacteristicEstimates2}, above yields
\begin{equation*}
\abs{\widehat{P_wP^1}}\leq\frac{1}{\cos2\omega_0\qnt{\frac{\cos\frac{\omega_0}{50}\sin\frac{\omega_0}{2}}{\cos\frac{\omega_0}{2}}-\sin\frac{\omega_0}{50}}}\leq\frac{1}{\sin\frac{\omega_0}{4}}.
\end{equation*}
Similarly, by the second of \eqref{eq:CharacteristicEstimates2}, we have
\begin{equation*}
\cos2\omega_0\abs{\widehat{P_wP^1}}\leq\abs{\widehat{P^sP^1}}.
\end{equation*}
Insert above into the second of \eqref{eq:CharacteristicEstimates1}, we have
\begin{equation*}
\abs{\widehat{P^sP^1}}\geq1/\qnt{\frac{\sin2\omega_0}{\cos2\omega_0}+\sin\frac{\omega_0}{50}}\geq\frac{1}{4\sin\omega_0}.
\end{equation*}
In view of the first of \eqref{eq:CharacteristicEstimates2}, above yields
\begin{equation*}
\abs{\widehat{P_wP^1}}\geq\frac{\cos\frac{\omega_0}{50}}{\cos\frac{\omega_0}{2}\qnt{\frac{\sin2\omega_0}{\cos2\omega_0}+\sin\frac{\omega_0}{50}}}\geq\frac{1}{4\sin\omega_0}.
\end{equation*}
In above inequalities, we repeat applying the fact that $\omega_0\in(0, \frac{\pi}{16})$.

This discussion comes from a simple observation. By Mean Value Theorem, the equations \eqref{eq:CharacteristicEstimates1} and \eqref{eq:CharacteristicEstimates2} can be written as
\begin{equation}\label{eq:CharacteristicEstimates1P}
1=\abs{P_wP^s}=\sin\qnt{\omega(\vec{U}_s)+\Delta\omega_1}\abs{\widehat{P_wP^1}}-\sin\qnt{\Delta\omega_2}\abs{\widehat{P^sP^1}},
\end{equation}
and
\begin{equation}\label{eq:CharacteristicEstimates2P}
0=\cos\qnt{\omega(\vec{U}_s)+\Delta\omega_3}\abs{\widehat{P_wP^1}}-\cos\qnt{\Delta\omega_4}\abs{\widehat{P^sP^1}}.
\end{equation}
where $\abs{\Delta\omega_i}\leq\frac{\omega_0}{100}$. So $\abs{\widehat{P_wP^1}}$ and $\abs{\widehat{P^sP^1}}$ can be viewed as a perturbation of the solution $\abs{\widehat{P_wP^1}}$ and $\abs{\widehat{P^sP^1}}$ for the following equations
\begin{equation}\label{eq:CharacteristicEstimates3P}
1=\abs{P_wP^s}=\sin\omega(\vec{U}_s)\abs{\widehat{P_wP^1}}-\sin0\abs{\widehat{P^sP^1}},
\end{equation}
and
\begin{equation}\label{eq:CharacteristicEstimates4P}
0=\cos\omega(\vec{U}_s)\abs{\widehat{P_wP^1}}-\cos0\abs{\widehat{P^sP^1}},
\end{equation}
which is corresponding to case where $\vec{U}(P)\equiv \vec{U}_s(P_w)$ and $\epsilon=0$. The other cases for the first part can be discussed similarly.

For the second part, since $Q$ is the intersection of $C^+$ and $C^-$, as \eqref{eq:CharacteristicEstimates1} and \eqref{eq:CharacteristicEstimates2}, we have
\begin{equation}\label{eq:CharacteristicEstimates3}
\begin{split}
1=\abs{P_wP^s}=&\int_{\widehat{P_wQ}}\sin\qnt{\tau(\vec{U}_s(P_w))-\beta(\vec{U}(P))}\rm{ds}\\
&+\int_{\widehat{P^sQ}}\sin\qnt{\alpha(\vec{U}(P))-\tau(\vec{U}_s(P_w))}\rm{ds},
\end{split}
\end{equation}
and
\begin{equation}\label{eq:CharacteristicEstimates4}
\begin{split}
0=&\int_{\widehat{P_wQ}}\cos\qnt{\tau(\vec{U}_s(P_w))-\beta(\vec{U}(P))}\rm{ds}\\
&-\int_{\widehat{P^sQ}}\cos\qnt{\alpha(\vec{U}(P))-\tau(\vec{U}_s(P_w))}\rm{ds}.
\end{split}
\end{equation}
$Q_w$ is the projection of $Q$, we have
\begin{equation}\label{eq:CharacteristicEstimates5}
\begin{split}
0=&\int_{\widehat{P_wQ}}\cos\qnt{\alpha(\vec{U}(P))-\tau(\vec{U}(Q_w))}\rm{ds}\\
&-\int_{\widehat{P_wQ_w}}\cos\qnt{\tau(\vec{U}(P))-\tau(\vec{U}_s(P_w))}\rm{ds},
\end{split}
\end{equation}
and $\abs{\widehat{P_wQ_w}}=\xi(Q)$. By \eqref{eq:CharacteristicEstimates3} and \eqref{eq:CharacteristicEstimates4}, we get the estimate for $\abs{\widehat{P_wQ}}$. Then, we estimate $\abs{\widehat{P_wQ_w}}=\xi(Q)$ by \eqref{eq:CharacteristicEstimates5}. The third part is similar.
\end{proof}

\subsection{Extension of the local solution and uniform estimates}\label{subsect:StandisedProblem}
In this subsection, we first introduce the inductive estimates of the states and their derivatives and establish the local existence of the solution with these estimates. Then, by continuity method, we prove the global existence of the solution by showing the inductive estimates can be extended to infinity.
\subsubsection{Inductive estimates and the local solution}
By \eqref{eq:SupersonicWallCondition} and \eqref{eq:WallCondition2nd}, we set $M_b>0$ as follows
\begin{equation*}
M_b:=\sup_{\xi_0\in\Real^+}\max_{\xi\in[\frac{\xi_0}{2}, \xi_0+\frac{\tan\omega_0\xi_0}{2}]}\xi\abs{\kappa_w(\xi)},
\end{equation*}
$M_\pm$ be the solution of
\begin{equation}\label{eq:pmBound}
\begin{cases}
\frac{\bar{K}+1}{2}(M_-+1)+1=M_+;\\
\frac{3-\bar{K}}{2}M_++M_b+1=M_-,
\end{cases}
\end{equation}
and $\delta_0$
\begin{equation*}
\delta_0:=\min\set{\delta_\omega, \delta_m, \delta_{\bar{K}}}.
\end{equation*}
To conclude, the constants $\delta_0, \omega_0, m_0, m_1, \bar{K}, M_b$ and $M_\pm$ defined above are all depending on the states set $\mathcal{C}_s$ or the wall $\textsf{W}$. In particular, $\omega_0, m_0, m_1, \bar{K}, M_b$ and $M_\pm$ can be expressed by $M_s$.

Set the {\bf inductive estimates}:
\begin{equation}\label{eq:SonicCharacteristicBoundInductive}
\abs{\bar{\partial}^\pm c(\xi)}\leq \frac{M_\pm}{\xi},
\end{equation}
and
\begin{equation}\label{eq:VelocityBoundInductive}
(u, v)(\xi)\in \mathcal{C}_s^{\frac{\delta_0}{3}}.
\end{equation}

\begin{thm}\label{thm:InitialTheorem}
There exists a positive constant $\epsilon_O$, such that, for any $\epsilon\leq\epsilon_O$, Problem 2 admits a local solution with the inductive estimates \eqref{eq:SonicCharacteristicBoundInductive} and \eqref{eq:VelocityBoundInductive} in $\xi\leq\xi_\epsilon$, where $\xi_\epsilon$ is a constant depending on $\epsilon$.
\end{thm}
\begin{proof}
By \cite[Theorem 3.1.1.]{LiYu1985DUMS}, we can directly establish local existence of $C^1$ solution for Problem 2. The solution $\vec{U}$ and $\vec{U}_s$ obey the algebra equations 
\begin{equation}\label{eq:OriginEquation}
\begin{cases}
G(\vec{U}(O), \epsilon)=0;\\
(\frac{v}{u})(O)=f'(0),
\end{cases}
\end{equation} 
and 
\begin{equation}\label{eq:OriginEquationS}
\begin{cases}
G(\vec{U}_s(O), 0)=0;\\
(\frac{v_s}{u_s})(O)=f'(0),
\end{cases}
\end{equation}
respectively. Then, by the continuity of $G$, above two lead to
\begin{equation*}
\lim\limits_{\epsilon\rightarrow0}\vec{U}(O)=\vec{U}_s(O)\in \mathcal{C}_s.
\end{equation*}
So there exists $\epsilon_O$, for any $\epsilon\leq\epsilon_O$, we have \eqref{eq:VelocityBoundInductive} around the origin. Next, since $\lim\limits_{\xi\rightarrow0}\frac{M_\pm}{\xi}=+\infty$, it holds \eqref{eq:SonicCharacteristicBoundInductive} locally as well. The proof is complete.
\end{proof}
We show two important estimates as follows.

\begin{lem}\label{lem:DerivativeEquiv}
If the $C^1$ solution $(u, v)\in\mathcal{C}_s^{\delta_\omega}$, then we have
\begin{equation}\label{eq:VelocityDerivativeBoundC}
\abs{\bar{\partial}^\pm(u, v)}\leq M_U\abs{\bar{\partial}^\pm c}.
\end{equation}
Furthermore, if \eqref{eq:SonicCharacteristicBoundInductive} holds in $\Omega_0$ for $\xi(x, y)\in[0, \xi_0]$, we have
\begin{equation}\label{eq:VelocityDerivativeBound}
\abs{\nabla_{(x, y)}u(P)}+\abs{\nabla_{(x, y)}v(P)}\leq\frac{M_U}{\xi(P)}, \quad\mbox{for}\quad P\in \Omega_0\quad\mbox{ and }\quad\xi(P)\leq \xi_0.
\end{equation}
\end{lem}

\begin{proof}
For $(u, v)\in\mathcal{C}_s^{\delta_\omega}$, it follows from Lemma \ref{lem:PerturbationConstants} that the vectors $(\cos\alpha, \sin\alpha)$ and $(\cos\beta, \sin\beta)$ have an angle, which is $2\omega$, in the interval $[2\omega_0, \pi-2\omega_0]$. So, for any smooth function $\psi$, we have the following derivative estimates
\begin{equation*}
\abs{\bar{\partial}^\pm\psi}\leq\abs{\nabla\psi}\leq \frac{2}{\sin\omega_0}\qnt{\abs{\bar{\partial}^+\psi}+\abs{\bar{\partial}^-\psi}}.
\end{equation*}
In particular, for the velocity $(u, v)$,
\begin{equation*}
\abs{\nabla_{(x, y)}u(P)}+\abs{\nabla_{(x, y)}v(P)}\leq \frac{2}{\sin\omega_0}\qnt{\abs{\bar{\partial}^+u}+\abs{\bar{\partial}^-u}+\abs{\bar{\partial}^+v}+\abs{\bar{\partial}^-v}}.
\end{equation*}
By the sonic expressions \eqref{eq:SonicExpressions}, for $(u, v)\in \mathcal{C}_s^{\delta_\omega}$,
\begin{equation*}
\abs{\bar{\partial}^\pm u}+\abs{\bar{\partial}^\pm v}\leq M_U\abs{\bar{\partial}^\pm c}.
\end{equation*}
Combining these two inequalities, we have
\begin{equation*}
\abs{\nabla_{(x, y)}u(P)}+\abs{\nabla_{(x, y)}v(P)}\leq \frac{M_U}{\sin\omega_0}\qnt{\abs{\bar{\partial}^+c}+\abs{\bar{\partial}^-c}}\leq\frac{M_U}{\xi(P)}.
\end{equation*}
The proof is complete.
\end{proof}

\begin{thm}[Perturbation estimate]\label{thm:SpecialSolutionPerturbationEstimate}
Under the same assumptions as in Lemma \ref{lem:DerivativeEquiv}, there exists $\epsilon_U>0$ such that for any $\epsilon\in\qnt{0, \epsilon_U}$, we have
\begin{equation}\label{eq:SpecialSolutionPerturbationEstimate3}
\abs{\vec{U}(P)-\vec{U}_s(P_w)}\leq \epsilon M_s\leq \frac{\delta_0}{3}.
\end{equation}
\end{thm}
\begin{proof}
In view of Lemma \ref{lem:DerivativeEquiv}, \eqref{eq:VelocityDerivativeBound} is available. So, by the mean value theorem, we have
\begin{equation}\label{eq:EstimateBetweenSW}
\begin{split}
&\abs{(u, v)(P)-(u, v)(P^s)}+\abs{(u, v)(P)-(u, v)(P_w)}\\
&\quad\quad\quad\leq \sup_{\tilde{P}\in P_wP^s}\abs{\nabla_{(x, y)}(u, v)(\tilde{P})}\abs{P_wP^s}\\
&\quad\quad\quad\leq \epsilon M_s,
\end{split}
\end{equation}
where we use \eqref{eq:NarrowEstimate} to estimate $\abs{P_wP^s}$.

On the one hand, recall that
\begin{equation*}
G(\vec{U}(P^s), \epsilon)=0.
\end{equation*}
So, by \eqref{eq:EstimateBetweenSW} and above, on the $(u, v)$ plane, the state $(u, v)(P)$ locates around the shock polar \eqref{eq:ShockPolar}. On the other hand, by the solid boundary relation,
\begin{equation*}
\frac{v(P_w)}{u(P_w)}=f'(P_w).
\end{equation*}
and \eqref{eq:EstimateBetweenSW}, on the $(u, v)$ plane, the state $(u, v)(P)$ locates around the line
\begin{equation}\label{eq:SolidWallStatesLine}
\frac{v}{u}=f'(P_w).
\end{equation}

In view of above two, the state $(u, v)(P)$ must locates around the intersection of the shock polar \eqref{eq:ShockPolar} and the line \eqref{eq:SolidWallStatesLine}.

In addition, let $\vec{U}_\epsilon^*=(u_\epsilon^*, v_\epsilon^*)$ be the unique supersonic state given by
\begin{equation*}
\begin{cases}
G(\vec{U}_\epsilon^*, \epsilon)=0,\\
\frac{v_\epsilon^*}{u_\epsilon^*}=f'(P_w),
\end{cases}
\end{equation*}
for small $\epsilon$.

Therefore, we have
\begin{equation}\label{eq:uvPerturbation1}
\abs{(u, v)(P)-(u_\epsilon^*, v_\epsilon^*)}\leq \epsilon M_s.
\end{equation}
Finally, recalling that $\vec{U}_s(P_w)=(u_s, v_s)(P_w)$ is the supersonic solution of
\begin{equation*}
\begin{cases}
G(\vec{U}_s(P_w), 0)=0,\\
\frac{v_s(P_w)}{u_s(P_w)}=f'(P_w),
\end{cases}
\end{equation*}
and $G(\vec{U}, \epsilon)$ is a perturbation of $G(\vec{U}, 0)$, we have
\begin{equation}\label{eq:uvPerturbation2}
\abs{(u_s, v_s)(P_w)-(u_\epsilon^*, v_\epsilon^*)}\leq \epsilon M_s.
\end{equation}
Therefore, combining \eqref{eq:uvPerturbation1} with \eqref{eq:uvPerturbation2}, we get \eqref{eq:SpecialSolutionPerturbationEstimate3}. The proof is complete.
\end{proof}

\begin{figure}[htbp]
\setlength{\unitlength}{1bp}
\begin{center}
\begin{picture}(200,180)
\put(-30,-10){\includegraphics[scale=0.7]{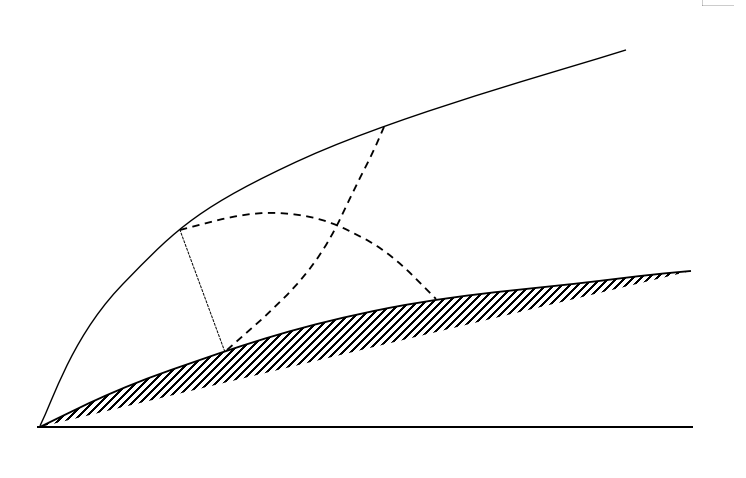}}
\put(89,110){\small$\bar{\Omega}_1$}
\put(109,70){\small$\underline{\Omega}_1$}
\put(69,75){\small$\Omega_1$}
\put(160,158){\small$\textsf{S}$}
\put(77,93){\tiny$C^-$}
\put(77,63){\tiny$C^+$}
\put(200,75){\small$\textsf{W}$}
\put(150,73){\small$\underline{P}_1$}
\put(32,93){\small$P^s$}
\put(120,142){\small$\bar{P}_1$}
\put(47,44){\small$P_w$}
\end{picture}
\end{center}
\caption{the characteristics between the shock and wedge (1)}\label{fig:characteristicSW}
\end{figure}

\subsubsection{Finite Growth Estimate}
By Theorem \ref{thm:InitialTheorem}, we know that it holds \eqref{eq:SonicCharacteristicBoundInductive} and \eqref{eq:VelocityBoundInductive} for $\xi\leq\xi_\epsilon$. We will show that if the solution $(u, v)$ satisfies \eqref{eq:SonicCharacteristicBoundInductive} and \eqref{eq:VelocityBoundInductive} for $\xi\leq\xi_0$ with $\xi_0>0$, then there exists $\Delta$ depending on $\xi_0$ such that the solution can be extended to $\xi\leq \xi_0+\Delta$ with \eqref{eq:SonicCharacteristicBoundInductive} and \eqref{eq:VelocityBoundInductive}. As illustrated in Figure \ref{fig:characteristicSW}, with given smooth data on $P_wP^s=\set{P|\xi(P)=\xi_0,~0\leq\eta(P)\leq\eta(P^s)}$ and the slipping boundary condition on the wedge $\mathsf{W}$, we consider the standardised free boundary value problem in $\xi\geq\xi_0$. We have the following finite growth estimates.

\begin{thm}[Finite growth estimates]\label{thm:finitegrowth}
Assume that Problem 2 admits a $C^1$ smooth solution in $\xi\in[0, \xi_0]$ with \eqref{eq:SonicCharacteristicBoundInductive} and \eqref{eq:VelocityBoundInductive}, then there exists a positive constant $\epsilon_F<\min\set{\delta_0, \epsilon_U}$, so that we can solve free boundary value problem in $\Omega_1\cup\bar{\Omega}_1\cup\underline{\Omega}_1$ for $\epsilon<\epsilon_F$ with the estimates
\begin{equation}\label{eq:SemiclassicalEstimate1}
\begin{split}
\abs{\bar{\partial}^\pm c|_{\Omega_1\cup\bar{\Omega}_1\cup\underline{\Omega}_1}}\leq \frac{M_s\qnt{M_++M_-+M_b}}{\xi_0},
\end{split}
\end{equation}
and
\begin{equation}\label{eq:SemiclassicalEstimate2}
\abs{\vec{U}(P) - \vec{U}(P')}\leq \frac{\delta_0}{3},
\end{equation}
where $P\in P_wP^s$ and $P'\in\Omega_1\cup\bar{\Omega}_1\cup\underline{\Omega}_1$ are connected via a characteristic.
\end{thm}

\begin{proof}
To see the impact of $\epsilon$, we employ a standard scaling transformation. Apply the coordinates transformation
\begin{equation}\label{eq:Standise}
(x, y)\mapsto (\tilde{x}, \tilde{y}):=\frac{1}{\abs{P^sP_w}(\xi_0)}(x-x(P_w), y-y(P_w))=(\frac{x-x(P_w)}{\abs{P^sP_w}(\xi_0)}, \frac{y-y(P_w)}{\abs{P^sP_w}(\xi_0)}).
\end{equation}
To utilise the decay \eqref{eq:SonicCharacteristicBoundInductive}, we keep the parameter for $\xi(P_w)=\xi_0$. In the following, we omit the tilde $'\textasciitilde’$.

Recalling the narrow estimate \eqref{eq:NarrowEstimate}, under this transformation, we can assume without loss of generality that
\begin{equation}\label{eq:StandisedInitialBoundaryLength}
\abs{P^sP_w}(\xi_0)=1;
\end{equation}
\begin{equation}\label{eq:StandisedInitialValueDerivativeBound}
\abs{\bar{\partial}^\pm c|_{P^sP_w}}\leq \epsilon M_0\xi_0\frac{M_\pm}{\xi_0}=\epsilon M_0 M_\pm;
\end{equation}
and
\begin{equation}\label{eq:StandisedWallBoundaryCurvatureBound}
\abs{\kappa_{\tilde{w}}|_{[\xi_0, \xi_0+\frac{\xi_0\tan\omega_0}{2\abs{P^sP_w}(\xi_0)}]}}\leq \epsilon M_0\xi_0\abs{\kappa_w|_{[\frac{\xi_0}{2}, \xi_0+\frac{\tan\omega_0\xi_0}{2}]}}\leq \epsilon M_0M_b,
\end{equation}
where we apply the decay \eqref{eq:WallCondition2nd} to get the last inequality.

The proof consists of three wave patterns discussions.

{\bf Pattern 1:} In Figure \ref{fig:characteristicSW}, we have the data of $(u, v)$ along $P_wP^s$ with the estimate\eqref{eq:StandisedInitialValueDerivativeBound}. Then we can solve the flow field in $\Omega_1$.

{\bf Pattern 2:} In Figure \ref{fig:characteristicSW}, assume that: (1) along the $C^+$ characteristic, we have the data of $(u, v)$ solved in {\bf Pattern 1}; (2) along the wall boundary $\mathsf{W}$, we have the boundary value condition \eqref{eq:StandisedWallBoundaryCurvatureBound}. We can solve the flow field in $\underline{\Omega}_1$.

{\bf Pattern 3:} In Figure \ref{fig:characteristicSW}, assume that: (1) along the $C^-$ characteristic, we have the data of $(u, v)$ solved in {\bf Pattern 1}; (2) along the shock front boundary $\mathsf{S}$, it holds the free boundary value condition \eqref{eq:RHCondition}. We can solve the flow field in $\bar{\Omega}_1$.

\medskip

\underline{Consider Pattern 1}. By \cite[Remark 1.4.2]{LiYu1985DUMS}, there exists $\Delta_i>0$, such that the Cauchy problem \eqref{eq:EulerEquations} with initial data $(u, v)|_{P\in P_wP^s}$ satisfying $|\nabla(u, v)(P)|\big|_{P\in P_wP^s}<\Delta_i$ admits a unique $C^1$ smooth solution in the determined domain of $P_wP^s$, namely, $\Omega_1$. By Lemma \ref{lem:DerivativeEquiv}, \eqref{eq:StandisedInitialValueDerivativeBound} and Theorem \ref{thm:SpecialSolutionPerturbationEstimate}, we know that there exists a positive constant $\epsilon_i$ such that for any $\epsilon<\epsilon_i$, we have $|\nabla(u, v)(P)||_{P\in P_wP^s}<\Delta_i$ and $(u, v)(P)|_{P\in P_wP^s}\in \mathcal{B}(\vec{U}_s(P_w),\frac{\delta_0}{3})$. Thus we get the $C^1$ smooth solution in $\Omega_1$. We point out that the domain $\Omega_1$ depends on $\epsilon$.

Recalling that $M_l$ is the maximum of the length for all the $C^\pm$ characteristics obtained by Theorem \ref{thm:characteristicestimate} for $\abs{P_wP^s}=1$, we consider the coefficients of \eqref{eq:CharacteristicDecomposition} and let
\begin{equation*}
M_{\omega}:=\sup _{\vec{U} \in \mathcal{C}_{s}^{\delta_{0}}}\left|\frac{\nu\left(1+\tan ^{2} \omega\right)}{c}+\frac{\nu\left(\tan ^{2} \omega-1\right)^{2}+2 \tan ^{2} \omega}{c\left(\tan ^{2} \omega+1\right)}\right| .
\end{equation*}

By \eqref{eq:StandisedInitialValueDerivativeBound}, choose $\Delta_{i1}$ such that for any $\epsilon<\Delta_{i1}$, we have
\begin{equation*}
M_l M_s \me^{M_l} \sup _{P \in P_w P^s}\left(\left|\bar{\partial}^+ c\right|+\left|\bar{\partial}^- c\right|\right)(P)<\frac{\delta_{0}}{20},
\end{equation*}
and
\begin{equation}\label{eq:Patterns1}
\me^{M_l} \sup _{P \in P_w P^s}\left(\left|\bar{\partial}^+ c\right|+\left|\bar{\partial}^- c\right|\right)(P) M_{\omega}<\frac{1}{3}.
\end{equation}

Claim that for $\epsilon<\Delta_{i1}$, it holds uniformly in $\Omega_1$ that
\begin{equation}\label{eq:Patterns2}
\left(\left|\bar{\partial}^+ c\right|+\left|\bar{\partial}^- c\right|\right) M_{\omega}<1,
\end{equation}
and
\begin{equation*}
\abs{\vec{U}(P)-\vec{U}(P')} \leq \frac{\delta_0}{20},
\end{equation*}
where $P\in P_wP^s$ and $P'\in\Omega_1$ are connected via a characteristic.

Note that it holds \eqref{eq:Patterns1} on $P_wP^s$. By continuity of the solution as well as its first derivative, we know that \eqref{eq:Patterns2} holds in a neighbourhood of $P_wP^s$.
We first prove that \eqref{eq:Patterns2} could not fail in $\Omega_1$ under the assumption \eqref{eq:Patterns1} on $P_wP^s$. Assume to the contrary that there exists a point $P'\in \Omega_{1}$ such that $\left(\left|\bar{\partial}^{+} c\right|+\left|\bar{\partial}^{-} c\right|\right)(P') M_{\omega}=1$, $\left(\left|\bar{\partial}^{+} c\right|+\left|\bar{\partial}^{-} c\right|\right) (Q)M_{\omega}<1$ and $\vec{U}(Q)\in \mathcal{C}_s^{\delta_0}$ for $Q\in\Omega_{1}$ with $\xi(Q)<\xi(P')$.We go on to deduce a contradiction.

Indeed, it follows from \eqref{eq:CharacteristicDecomposition} and \eqref{eq:Patterns2} that
\begin{equation*}
\bar{\partial}^{\mp}\left|\bar{\partial}^{ \pm} c\right| \leq\left|\bar{\partial}^{ \pm} c\right|.
\end{equation*}
Integrating it along the characteristic from $P'$ to $P$, and then the Gronwall inequality leads to
\begin{equation}\label{eq:Patterns3}
\left(\left|\bar{\partial}^{+} c\right|+\left|\bar{\partial}^{-} c\right|\right)(P') \leq \mathrm{e}^{M_{l}} \sup _{P \in P_{w} P^{s}}\left(\left|\bar{\partial}^{+} c\right|+\left|\bar{\partial}^{-} c\right|\right)(P).
\end{equation}
Here we have used the fact that the length of the characteristic between $P$ and $P'$ is bounded by $M_l$ due to Theorem \ref{thm:characteristicestimate}. It follows from \eqref{eq:Patterns1} that \eqref{eq:Patterns3} leads to 
$$\left(\left|\bar{\partial}^{+} c\right|+\left|\bar{\partial}^{-} c\right|\right)(P')M_{\omega}<\frac{1}{3},$$
which is a contradiction to the choice of $P'$.

Second, if there exists a $P'\in\Omega_1$ with smallest $\xi(P')$ such that $\mbox{dist}(\vec{U}(P'), \mathcal{C}_s)=\delta_0$ while \eqref{eq:Patterns2} holds for $\xi\leq \xi(P')$, we get a contradiction once again. In fact, we apply the mean value theorem to have
\begin{equation}\label{eq:finitegrowth1}
\begin{aligned}
\abs{\vec{U}(P')|_{P' \in \Omega_1}-\vec{U}(P)|_{P \in P^{s} P_{w}\qnt{\xi_0}}} & \leq\abs{\widehat{P P'}} \sup _{Q \in \Omega_{1}\cap \{\xi_0\leq \xi_Q\leq \xi_P\}}|\nabla \vec{U}(Q)| \\
&\leq M_{l} \sup _{Q \in \Omega_{1}\cap \{\xi_0\leq \xi_Q\leq \xi_P\}}|\nabla \vec{U}(Q)| \\
& \leq M_{l} M_{s} \sup _{Q \in \Omega_{1}\cap \{\xi_0\leq \xi_Q\leq \xi_P\}}\left(\left|\bar{\partial}^{+} c\right|+\left|\bar{\partial}^{-} c\right|\right)(Q) \\
& \leq M_{l} M_{s} \mathrm{e}^{M_{l}} \sup _{Q \in P_{w} P^{s}}\left(\left|\bar{\partial}^{+} c\right|+\left|\bar{\partial}^{-} c\right|\right)(Q)\\
&<\frac{\delta_{0}}{20},
\end{aligned}
\end{equation}
which implies that $\vec{U}(P')\in \mathcal{C}_{s}^{\frac{\delta_{0}}{3}+\frac{\delta_{0}}{20}}$, a contradiction. Due to \eqref{eq:Patterns2}, it follows from \eqref{eq:StandisedInitialValueDerivativeBound} and \eqref{eq:Patterns3} that
\begin{equation}\label{eq:Patterns4}
\sup_{P \in \Omega_l}
\qnt{\abs{\bar{\partial}^+ c}+\abs{\bar{\partial}^- c}}(P) \leq \epsilon \me^{M_l} M_0 M_\pm\leq \epsilon M_sM_0 M_\pm.
\end{equation}
Rescaling back, we show that \eqref{eq:SemiclassicalEstimate1} is valid in $\Omega_1$.

\medskip

Pattern 2 and Pattern 3 can be treated by the theory of characteristic boundary value problem in \cite[Section 3.3]{LiYu1985DUMS}. We need to check the minimal characterizing number.

\medskip

\underline{Consider Pattern 2}. The boundary value problem can be formulated as: (1) the reformed wall condition of \eqref{eq:VelocityWallCondition} as follows
\begin{equation*}
u+\lambda_-(P_w)v=u+\lambda_-(P_w)v-\lambda_-(P_w)v+\lambda_-(P_w)f'(P_w)u;
\end{equation*}
(2) the characteristic boundary value condition
\begin{equation*}
u+\lambda_+(P_w)v=u+\lambda_+(P_w)v|_{C^+};
\end{equation*}
(3) the equation \eqref{eq:EulerEquationsC}.

By a rotation, we may assume that $(u, v)|_{P_w}$ is the $x$ direction, in another word $v(P_w)=0$. In such a case
\begin{equation*}
\lambda_+(P_w)=-\lambda_-(P_w)>0,\quad f'|_{P_w}=0.
\end{equation*}
We can direct compute the minimal characterizing number as
\begin{equation*}
\theta_{\min}=\abs{\frac{-\lambda_-(P_w)}{\lambda_+(P_w)-\lambda_-(P_w)}}=\frac{1}{2}.
\end{equation*}
By \cite[Theorem 3.3.4]{LiYu1985DUMS}, we get the local solution. By the results in \cite{LiYu1985DUMS}, if $\vec{U}|_{C^+}\subset\mathcal{C}_s^{\frac{\delta_0}{3}+\frac{\delta_0}{20}}$, $\abs{\bar{\partial}^\pm c|_{C^+}}\leq 1$ and $\abs{\kappa_{\tilde{b}}}\leq 1$, we can solve the problem in $\underline{\Omega}_1\cap\set{\xi\in[\xi_0, \xi_0+\Delta_w]}$ where $\Delta_w$ is depending on $\mathcal{C}_s^{\delta_0}$. In our case, with the estimates on the wall \eqref{eq:StandisedWallBoundaryCurvatureBound} and the characteristic boundary estimate \eqref{eq:SemiclassicalEstimate1} on $C^+$, we can rescale $(x, y)=\frac{1}{\epsilon M_s(M_s+M_++M_-+M_b)}(\bar{x}, \bar{y})$. Thus under the $(\bar{x}, \bar{y})$ coordinates, the problem can be solved forward in $\Delta_w$, which is depending only on $\mathcal{C}_s^{\delta_0}$. So in the $(x, y)$ coordinates, the problem can be solved forward in $\frac{\Delta_w}{\epsilon M_s(M_s+M_++M_-+M_b)}$. This is the semi-classic property. Therefore, we can solve this pattern in $\underline{\Omega}_1$ provided $\epsilon<\Delta_w$ for some positive constant $\Delta_w$. The estimates \eqref{eq:SemiclassicalEstimate1} and \eqref{eq:SemiclassicalEstimate2} in $\underline{\Omega}_1$ can be obtained similarly as that in Pattern 1 by \eqref{eq:WallRelationCharacteristic}. Moreover, as \eqref{eq:finitegrowth1} we have
\begin{equation}\label{eq:finitegrowth2}
\abs{\vec{U}(P')-\vec{U}(P)} <\frac{\delta_{0}}{20},
\end{equation}
where $P\in C^+$ and $P'\in\underline{\Omega}_1$ are connected via a -characteristic.

\medskip

\underline{Consider Pattern 3}. To apply \cite[Theorem 3.3.4]{LiYu1985DUMS}, we have the shock free boundary condition
\begin{equation*}
[v]\phi' = -[u],
\end{equation*}
with the shock polar condition
\begin{equation*}
G(\vec{U}, \epsilon)=0\quad\mbox{on}\quad y=\phi(x),
\end{equation*}
which can be reformed as
\begin{equation*}
u+\lambda_+(P^s)v=G(\vec{U}, \epsilon)+(u+\lambda_+(P^s)v)\quad\mbox{on}\quad y=\phi(x).
\end{equation*}
Once more, we rotate to set $v(P^s)=0$. Thus the minimal characterizing number condition can be calculated as follows
\begin{equation*}
\begin{split}
H&=\frac{1}{\lambda_+(P^s)-\lambda_-(P^s)}\qnt{G_u(\vec{U}, \epsilon)+1, G_v(\vec{U}, \epsilon)+\lambda_+(P^s)}(-\lambda_-(P^s), 1)\\
&=\frac{-G_u(\vec{U}, \epsilon)\lambda_-(P^s)+G_v(\vec{U}, \epsilon)}{\lambda_+(P^s)-\lambda_-(P^s)}+1\neq1.
\end{split}
\end{equation*}
The discussion for \eqref{eq:SemiclassicalEstimate1} and \eqref{eq:SemiclassicalEstimate2} in $\bar{\Omega}_1$ are similar as Pattern 2, in the same way, for some positive constant $\Delta_s$, we have 
\begin{equation}\label{eq:finitegrowth3}
\abs{\vec{U}(P')-\vec{U}(P)} <\frac{\delta_{0}}{20},\quad\mbox{for}\quad\epsilon<\Delta_s,
\end{equation}
where $P\in C^-$ and $P'\in\bar{\Omega}_1$ are connected via a +characteristic.
\medskip
Finally, we can summarise the estimates \eqref{eq:finitegrowth1}, \eqref{eq:finitegrowth2} and \eqref{eq:finitegrowth3} in $\Omega_1$, $\bar{\Omega}_1$ and $\underline{\Omega}_1$ to get \eqref{eq:SemiclassicalEstimate2} with the derivative estimates \eqref{eq:SemiclassicalEstimate2} for $\epsilon_F\leq\min\set{\Delta_{i}, \Delta_{i1}, \Delta_w, \Delta_s, \delta_0, \epsilon_U}$. The proof is complete.
\end{proof}

\eqref{eq:SpecialSolutionPerturbationEstimate3} and \eqref{eq:SemiclassicalEstimate2} indicate the following
\begin{equation*}
(u, v)|_{\Omega_1\cup\bar{\Omega}_1\cup\underline{\Omega}_1}\in \mathcal{C}_s^{\frac{2\delta_0}{3}}.
\end{equation*}
\begin{lem}\label{lem:PerturbationVelocity}
If \eqref{eq:SonicCharacteristicBoundInductive} and \eqref{eq:VelocityBoundInductive} hold for $\xi\leq\xi_0$, then there exists $\epsilon_B\leq\epsilon_F$, such that in the domain
$$\xi(P)\in [\xi_0-\abs{P^sP_w}(\xi_0)\frac{8}{\tan\omega_0}, \xi_0+\abs{P^sP_w}(\xi_0)\frac{\tan\omega_0}{8}],$$
we have
\begin{equation}\label{eq:LocalPerturbationP}
\abs{\vec{U}(P)-\vec{U}_s(\xi_0)}\leq\frac{2\delta_0}{3}
\end{equation}
for $\epsilon<\epsilon_B$.
\end{lem}
\begin{proof}
First, by \eqref{eq:SpecialSolutionPerturbationEstimate3}, we have
\begin{equation}\label{eq:SpecialSolutionPerturbationEstimatedelta}
\abs{\vec{U}(P)-\vec{U}_s(P_w)}<\frac{\delta_0}{3}
\end{equation}
for $\xi(P)\in[\xi_0-\abs{P^sP_w}(\xi_0)\frac{8}{\tan\omega_0}, \xi_0]$.

Second, for $\xi\in [\xi_0-\abs{P^sP_w}(\xi_0)\frac{8}{\tan\omega_0}, \xi_0]$ and $\xi_0>0$, there exists $\epsilon_L>0$, such that we have
\begin{equation}\label{eq:equivContinuityS}
\abs{\vec{U}_s(\xi)-\vec{U}_s(\xi_0)}\leq\frac{\delta_0}{3},
\end{equation}
for $\epsilon<\epsilon_L$. To see this, we can apply the narrow estimate \eqref{eq:NarrowEstimate} to have
$$\xi\in[\xi_0-\abs{P^sP_w}(\xi_0)\frac{8}{\tan\omega_0},\xi_0]\subset [(1-\epsilon M_0\frac{8}{\tan\omega_0})\xi_0,\xi_0].$$
By the uniform continuity of $\vec{U}_s(\xi)$ and the decay \eqref{eq:WallCondition2nd}, for sufficiently small $\epsilon_L$, we can show \eqref{eq:equivContinuityS} for $\xi\in[(1-\epsilon M_0\frac{8}{\tan\omega_0})\xi_0, \xi_0]$ and $\xi_0\in\Real^+$. Thus, \eqref{eq:equivContinuityS} is proved.

Combining above two, we have
\begin{equation}\label{eq:LocalPerturbation1}
\abs{\vec{U}(P) - \vec{U}_s(\xi_0)}\leq\abs{\vec{U}(P)-\vec{U}_s(P_w)}+\abs{\vec{U}_s(P_w)-\vec{U}_s(\xi_0)}\leq \frac{2\delta_0}{3}
\end{equation}
for $\xi(P)\in [\xi_0-\abs{P^sP_w}(\xi_0)\frac{8}{\tan\omega_0}, \xi_0]$.

Third, for $\epsilon<\epsilon_F$, we have
\begin{equation*}
\abs{\vec{U}(P)-\vec{U}_s(\xi_0)}\leq\abs{\vec{U}(P)-\vec{U}(P')}+\abs{\vec{U}(P')-\vec{U}_s(\xi_0)}
\end{equation*}
where $\xi(P)\in[\xi_0, \xi_0+\frac{\tan\omega_0}{8}\abs{P^sP_w}(\xi_0)]$ and $\xi(P')=\xi_0$. For the first term on the right hand side, in view of \eqref{eq:SemiclassicalEstimate2} in Theorem \ref{thm:finitegrowth}, it has a bound of $\frac{\delta_0}{3}$. For the second term on the right hand side, by \eqref{eq:SpecialSolutionPerturbationEstimatedelta}, it has a bound of $\frac{\delta_0}{3}$ as well. Therefore, we have
\begin{equation}\label{eq:LocalPerturbation2}
\abs{\vec{U}(P)-\vec{U}_s(\xi_0)}\leq\frac{2\delta_0}{3}
\end{equation}
where $\xi(P)\in[\xi_0, \xi_0+\frac{\tan\omega_0}{8}\abs{P^sP_w}(\xi_0)]$.

Combining \eqref{eq:LocalPerturbation1} and \eqref{eq:LocalPerturbation2} and setting $\epsilon_B=\min\set{\epsilon_F, \epsilon_L}$, we complete the proof.
\end{proof}

\begin{figure}[htbp]
\setlength{\unitlength}{1bp}
\begin{center}
\begin{picture}(200,180)
\put(-30,-10){\includegraphics[scale=0.7]{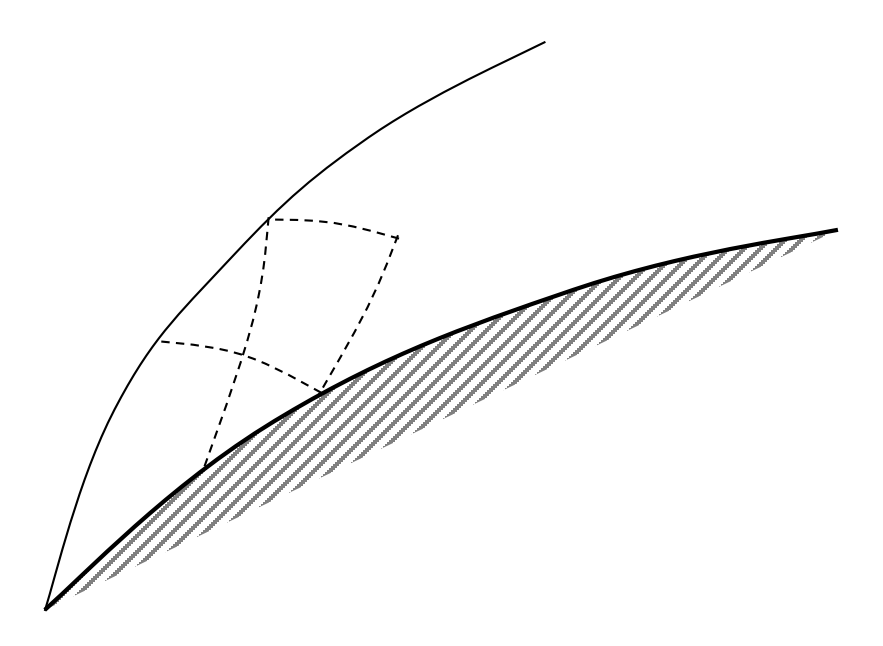}}
\put(93,110){\tiny$Q$}
\put(59,75){\tiny$\underline{Q}_1$}
\put(140,178){\small$\textsf{S}$}
\put(80,83){\tiny$C_1^+$}
\put(30,83){\tiny$C_2^-$}
\put(50,90){\tiny$C_2^+$}
\put(70,120){\tiny$C_1^-$}
\put(220,95){\small$\textsf{W}$}
\put(42,123){\tiny$\bar{Q}_1$}
\put(19,44){\tiny$\underline{Q}_2$}
\put(7,84){\tiny$\bar{Q}_2$}
\end{picture}
\end{center}
\caption{the characteristics between the shock and wedge (2)}\label{fig:CharacteristicReflectPM}
\end{figure}

Next, we intend to show the inductive estimates \eqref{eq:SonicCharacteristicBoundInductive} and \eqref{eq:VelocityBoundInductive} for $$\xi \leq \xi_0 + \frac{1}{8}\tan\omega_0\abs{P^sP_w}(\xi_0).$$ First, we prove \eqref{eq:SonicCharacteristicBoundInductive} in the following
\begin{thm}[Characteristic reflection estimates]\label{thm:SonicCharacteristicBoundInductive}
In the domain 
$$\xi_0\leq\xi\leq \xi_0+\frac{1}{8}\tan\omega_0\abs{P^sP_w}(\xi_0),$$ 
there exists a positive constant $\epsilon_R<\epsilon_B$, so that we have \eqref{eq:SonicCharacteristicBoundInductive} for $\epsilon<\epsilon_R$. This $\epsilon_R$ is only depending on the states $\mathcal{C}_s$ and the wedge boundary $\textsf{W}$ constants $\bar{f}, \underline{f}, a$ and $B$ in \eqref{eq:SupersonicWallCondition} and \eqref{eq:WallCondition2nd}.
\end{thm}

\begin{proof}
As illustrated in Figure \ref{fig:CharacteristicReflectPM}, fix a point $Q$ in the domain $\xi_0\leq\xi\leq \xi_0+\frac{1}{8}\tan\omega_0\abs{P^sP_w}(\xi_0)$. Intend to apply the result in Section \ref{subsect:StandisedProblem} and Theorem \ref{thm:characteristicestimate}, as \eqref{eq:Standise}, we standardise the domain at $\xi=\xi_0$. All the estimates of the derivatives are multiplied by $\abs{P^sP_w}(\xi_0)$. So, in the part $\xi(Q)\in [\xi_0-\abs{P^sP_w}(\xi_0)\frac{8}{\tan\omega_0}, \xi_0]$ and under the scaled domain, as \eqref{eq:StandisedInitialValueDerivativeBound}, we have
\begin{equation}\label{eq:StandisedInitialValueDerivativeBoundInverse}
\abs{\bar{\partial}^\pm c|_{Q^sQ_w(\xi)}}\leq\abs{P^sP_w}(\xi_0)\frac{M_\pm}{\xi} \leq\epsilon M_0\xi_0\frac{M_\pm}{\xi}\leq \epsilon M_s M_\pm,
\end{equation}
where we apply \eqref{eq:SonicCharacteristicBoundInductive} to estimate $\bar{\partial}^\pm c$. Then, in the part $\xi\in[\xi_0, \xi_0+\frac{1}{8}\tan\omega_0\abs{P^sP_w}(\xi_0)]$, by Theorem \ref{thm:finitegrowth}, we have \eqref{eq:SemiclassicalEstimate1}.

As Theorem \ref{thm:characteristicestimate}, we consider the $C^\pm$ characteristics emanating backward from $Q$ and their reflected parts $C_1^-, C_2^+, C_1^+$ and $C_2^-$. By the characteristic decomposition \eqref{eq:CharacteristicDecomposition}, we can integrate $\bar{\partial}^-\bar{\partial}^+ c$ and $\bar{\partial}^+\bar{\partial}^- c$ along $C_1^-$ and $C_2^+$ respectively to have the inequalities
\begin{equation}\label{eq:CharateristicPointEstimates1}
\abs{\bar{\partial}^+ c(Q)}\leq\abs{\bar{\partial}^+ c(\bar{Q}_1)}+ M_s \qnt{\epsilon M_s\qnt{M_++M_-+M_b}}^2,
\end{equation}
and
\begin{equation}\label{eq:CharateristicPointEstimates2}
\abs{\bar{\partial}^- c(\bar{Q}_1)}\leq\abs{\bar{\partial}^- c(\underline{Q}_2)} + M_s \qnt{\epsilon M_s\qnt{M_++M_-+M_b}}^2.
\end{equation}
Indeed, we integrate $\bar{\partial}^-\bar{\partial}^+ c$ along $C_1^-$ to have
\begin{equation*}
\begin{split}
\bar{\partial}^+ c(\bar{Q}_1)&=\bar{\partial}^+ c(Q)+\int_{\widehat{\bar{Q}_1Q}}\bar{\partial}^+c\Big( \nu(1+\tan^2\omega)\bar{\partial}^+c+\frac{\nu(\tan^2\omega-1)^2+2\tan^2\omega}{\tan^2\omega+1}\bar{\partial}^-c \Big)\rm{ds}\\
&\leq \bar{\partial}^+ c(Q)+\abs{\widehat{\bar{Q}_1Q}}\sup_{\widehat{\bar{Q}_1Q}}\qnt{\abs{\bar{\partial}^+c}+\abs{\bar{\partial}^-c}}^2M_s,
\end{split}
\end{equation*}
where we apply \eqref{eq:omegaPerturbation} to estimate the terms of $\tan\omega$ by $M_s$. Thus, inserting the estimates of $\bar{\partial}^\pm c$ \eqref{eq:StandisedInitialValueDerivativeBoundInverse} and \eqref{eq:SemiclassicalEstimate1} in above and adopting the characteristic estimates obtained in the first conclusion of Theorem \ref{thm:characteristicestimate} to control $\abs{\widehat{\bar{Q}_1Q}}$, we finally have \eqref{eq:CharateristicPointEstimates1}. \eqref{eq:CharateristicPointEstimates2} can be deduced in the same way.

On the shock part, it holds \eqref{eq:ShockRelation}, we apply \eqref{eq:gRange} in Lemma \ref{lem:PerturbationConstants} to have
\begin{equation}\label{eq:CharateristicPointEstimates3}
\abs{\bar{\partial}^+ c(\bar{Q}_1)}\leq \bar{K}\abs{\bar{\partial}^- c(\bar{Q}_1)}.
\end{equation}
By \eqref{eq:CharateristicPointEstimates1}-\eqref{eq:CharateristicPointEstimates3}, we have
\begin{equation*}
\abs{\bar{\partial}^+ c(Q)}\leq \bar{K}\abs{\bar{\partial}^- c(\underline{Q}_2)}+ M_s \qnt{\epsilon M_s\qnt{M_++M_-+M_b}}^2.
\end{equation*}
By Theorem \ref{thm:characteristicestimate} and Lemma \ref{lem:PerturbationVelocity}, we have $\xi(\underline{Q}_2)\leq \xi_0$. So it holds the estimate assumption \eqref{eq:SonicCharacteristicBoundInductive} for $\bar{\partial}^-c(\underline{Q}_2)$. We convert above estimate to
\begin{equation*}
\abs{\bar{\partial}^+ c(Q)}\leq\abs{P^sP_w}(\xi_0) \bar{K}\frac{M_-}{\xi(\underline{Q}_2)}+ M_s \qnt{\epsilon M_s\qnt{M_++M_-+M_b}}^2.
\end{equation*}
By a scaling discussion, \eqref{eq:NarrowEstimate} yields
\begin{equation*}
\xi(Q)\leq \xi_0+\frac{1}{8}\tan\omega_0\abs{P^sP_w}(\xi_0)\leq (1+ \frac{\epsilon M_0\tan\omega_0}{8})\xi_0;
\end{equation*}
and the third conclusion of Theorem \ref{thm:characteristicestimate} yields
\begin{equation*}
\xi(\underline{Q}_2)\geq\xi_0-\abs{P^sP_w}(\xi_0)\frac{8}{\tan\omega_0}\geq (1-\epsilon M_0 \frac{8}{\tan\omega_0})\xi_0.
\end{equation*}
So, by choosing $\epsilon$ small enough, we have
\begin{equation*}
\begin{split}
\abs{\bar{\partial}^+ c(Q)}&\leq\abs{P^sP_w}(\xi_0) \bar{K}\frac{M_-}{(1-\epsilon M_0 \frac{8}{\tan\omega_0})\xi_0}+ M_s \qnt{\epsilon M_s\qnt{M_++M_-+M_b}}^2\\
&\leq\frac{\abs{P^sP_w}(\xi_0)}{\xi(Q)} \bar{K}\frac{1+\frac{\epsilon M_0\tan\omega_0}{8}}{1-\epsilon M_0 \frac{8}{\tan\omega_0}}M_-\\
&\quad+\frac{\abs{P^sP_w}(\xi_0)}{\xi(Q)} \frac{M_s \qnt{\epsilon M_s\qnt{M_++M_-+M_b}}^2}{\abs{P^sP_w}(\xi_0) }(1+\frac{\epsilon M_0\tan\omega_0}{8})\xi_0\\
&\leq\frac{\abs{P^sP_w}(\xi_0)}{\xi(Q)} \bar{K}\frac{1+\frac{\epsilon M_0\tan\omega_0}{8}}{1-\epsilon M_0 \frac{8}{\tan\omega_0}}M_-\\
&\quad+\frac{\abs{P^sP_w}(\xi_0)}{\xi(Q)} \frac{M_s \qnt{\epsilon M_s\qnt{M_++M_-+M_b}}^2}{\epsilon M_1\xi_0 }(1+\frac{\epsilon M_0\tan\omega_0}{8})\xi_0\\
&\leq \abs{P^sP_w}(\xi_0) \frac{M_+}{\xi(Q)},
\end{split}
\end{equation*}
where we apply \eqref{eq:NarrowEstimate} to estimate $\abs{P^sP_w}(\xi_0)$ and use the first part of \eqref{eq:pmBound}.

Similarly, we can integrate along $C_1^+$ and $C_2^-$ to have
\begin{equation*}
\abs{\bar{\partial}^- c(Q)}\leq\abs{\bar{\partial}^- c(\underline{Q}_1)}+ M_s \qnt{\epsilon M_s\qnt{M_++M_-+M_b}}^2,
\end{equation*}
and
\begin{equation*}
\abs{\bar{\partial}^+ c(\underline{Q}_1)}\leq\abs{\bar{\partial}^+ c(\bar{Q}_2)} + M_s \qnt{\epsilon M_s\qnt{M_++M_-+M_b}}^2,
\end{equation*}
respectively. At the point $\underline{Q}_1\in\mathsf{W}$, by \eqref{eq:WallRelationCharacteristic}, we have
\begin{equation*}
\abs{\bar{\partial}^- c(\underline{Q}_1)}\leq \abs{\bar{\partial}^+ c(\underline{Q}_1)}+\epsilon M_0M_b.
\end{equation*}
In view of above three, we have
\begin{equation*}
\abs{\bar{\partial}^- c(Q)}\leq\abs{\bar{\partial}^+ c(\bar{Q}_2)}+\epsilon M_0M_b+ M_s \qnt{\epsilon M_s\qnt{M_++M_-+M_b}}^2.
\end{equation*}
Noting that
\begin{equation*}
\xi(\bar{Q}_2)\geq\xi_0-\abs{P^sP_w}(\xi_0)\frac{8}{\tan\omega_0}\geq (1-\epsilon M_0 \frac{8}{\tan\omega_0})\xi_0,
\end{equation*}
we have
\begin{equation*}
\begin{split}
\abs{\bar{\partial}^- c(Q)}&\leq\abs{P^sP_w}(\xi_0) \frac{M_+}{\xi(\bar{Q}_2)}+\epsilon M_0M_b+ M_s \qnt{\epsilon M_s\qnt{M_++M_-+M_b}}^2\\
&\leq\frac{\abs{P^sP_w}(\xi_0)}{\xi(Q)}\frac{1+\frac{\epsilon M_0\tan\omega_0}{8}}{1-\epsilon M_0 \frac{8}{\tan\omega_0}}M_+\\
&\quad+\frac{\abs{P^sP_w}(\xi_0)}{\xi(Q)} \frac{M_s \qnt{\epsilon M_s\qnt{M_++M_-+M_b}}^2}{\abs{P^sP_w}(\xi_0) }(1+\frac{\epsilon M_0\tan\omega_0}{8})\xi_0\\
&\leq\frac{\abs{P^sP_w}(\xi_0)}{\xi(Q)}\frac{1+\frac{\epsilon M_0\tan\omega_0}{8}}{1-\epsilon M_0 \frac{8}{\tan\omega_0}}M_+\\
&\quad+\frac{\abs{P^sP_w}(\xi_0)}{\xi(Q)} \frac{M_s \qnt{\epsilon M_s\qnt{M_++M_-+M_b}}^2}{\epsilon M_1\xi_0 }(1+\frac{\epsilon M_0\tan\omega_0}{8})\xi_0\\
&\leq \abs{P^sP_w}(\xi_0) \frac{M_-}{\xi(Q)},
\end{split}
\end{equation*}
where we apply \eqref{eq:NarrowEstimate} to estimate $\abs{P^sP_w}(\xi_0)$ and use the second of \eqref{eq:pmBound}. So it yields
\begin{equation*}
\abs{\bar{\partial}^\pm c(Q)}\leq \abs{P^sP_w}(\xi_0) \frac{M_\pm}{\xi(Q)}.
\end{equation*}
Rescaling back, we get \eqref{eq:SonicCharacteristicBoundInductive}.
\end{proof}
\begin{lem}\label{lem:SpecialSolutionPerturbationEstimate3}
In the domain $\xi_0\leq\xi\leq \xi_0+\frac{1}{8}\tan\omega_0\abs{P_wP^s}(\xi_0)$, we have $(u, v)(\xi)\in \mathcal{C}_s^{\frac{\delta_0}{3}}$.
\end{lem}
\begin{proof}
By the above theorem, we get \eqref{eq:SonicCharacteristicBoundInductive} in this region. Then, by Theorem \ref{thm:finitegrowth}, the assumptions of Theorem \ref{thm:SpecialSolutionPerturbationEstimate} are valid, we have \eqref{eq:SpecialSolutionPerturbationEstimate3}.
\end{proof}

\subsubsection{The global existence of smooth solution}\label{subsect:GlobalSolutionwithoutConvexity}
\begin{prop}\label{prop:GlobalExistence}
Under the conditions \eqref{eq:SupersonicWallCondition} and \eqref{eq:WallCondition2nd} on the wedge boundary $f$, there exists $\epsilon_0=\min\set{\epsilon_R, \epsilon_O}$, such that for any $\epsilon<\epsilon_0$, Problem 2 admits a global piecewise smooth solution with an attached shock.
\end{prop}
\begin{proof}
By Theorem \ref{thm:InitialTheorem}, for any fixed $0<\epsilon<\epsilon_O$, there always exists $\xi_\epsilon>0$ depending on $\epsilon$, it holds \eqref{eq:SonicCharacteristicBoundInductive} and \eqref{eq:VelocityBoundInductive} for $\xi\in[0, \xi_\epsilon]$. Then the assumptions of Theorem \ref{thm:SonicCharacteristicBoundInductive} and Lemma \ref{lem:SpecialSolutionPerturbationEstimate3} are valid. The inductive estimates can be extended to $\xi_\epsilon\leq\xi\leq \xi_\epsilon+\frac{1}{8}\tan\omega_0\abs{P^sP_w}(\xi_\epsilon)$. By \eqref{eq:NarrowEstimate}, we have $\epsilon M_1\xi_\epsilon\leq\abs{P^sP_w}(\xi_\epsilon)$. Thus these estimates are valid in the domain $\xi_\epsilon\leq\xi\leq \xi_\epsilon+\frac{\epsilon M_1}{8}\xi_\epsilon\tan\omega_0$. For the fixed $\epsilon$, we can repeat the process to extend the estimates to infinity and obtain the global solution. By \eqref{eq:SpecialSolutionPerturbationEstimate3}, the entropy condition is valid as well. The proof is complete.
\end{proof}

\subsection{Asymptotic behaviour}\label{subsect:AsymptoticBehaviour}
In this subsection, we show the asymptotic property of the solution. We first prove the decay of derivatives under the assumption \eqref{eq:WallCondition2nd} as follows
\begin{lem}\label{lem:DerivativeDecay}
The solution obtained in Proposition \ref{prop:GlobalExistence} satisfies 
\begin{equation*}
\lim\limits_{X\rightarrow+\infty}\max_{\xi(P)=X}\xi(P)\abs{\bar{\partial}^\pm c(P)}=0.
\end{equation*}
\end{lem}
\begin{proof}
Let
\begin{equation*}
F(X)=\sup_{\xi(P)\geq X}\xi(P)\max\set{\abs{\bar{\partial}^+ c(P)}, \abs{\bar{\partial}^- c(P)}}.
\end{equation*}
We intend to show $\lim\limits_{X\rightarrow+\infty}F(X)=0$. By inductive estimates \eqref{eq:SonicCharacteristicBoundInductive}, we have
\begin{equation*}
0\leq F(X)\leq \max\set{M_+, M_-},
\end{equation*}
for $X\in\Real^+$.

Denote $F_1=F(X)$ where $X\geq1$. We have
\begin{equation*}
\abs{\bar{\partial}^\pm c(P)}\leq \frac{F_1}{\xi},
\end{equation*}
for $\xi(P)\geq X$. As indicated in Figure \ref{fig:CharacteristicReflect}, for any point $P\in\Omega$, denote by $C_1^-$ the backward $C^-$ characteristic passing through it. This characteristic intersects the shock front $\textsf{S}$ at $\bar{P}_1$. Then, we denote by $C_1^+$ the backward $C^+$ characteristic starting from $\bar{P}_1$ and $\underline{P}_1$ the intersection of $C_1^+$ and $\textsf{W}$. This process continues and eventually approaches the origin.

\begin{figure}[htbp]
\setlength{\unitlength}{1bp}
\begin{center}
\begin{picture}(200,160)
\put(-30,-10){\includegraphics[scale=0.34]{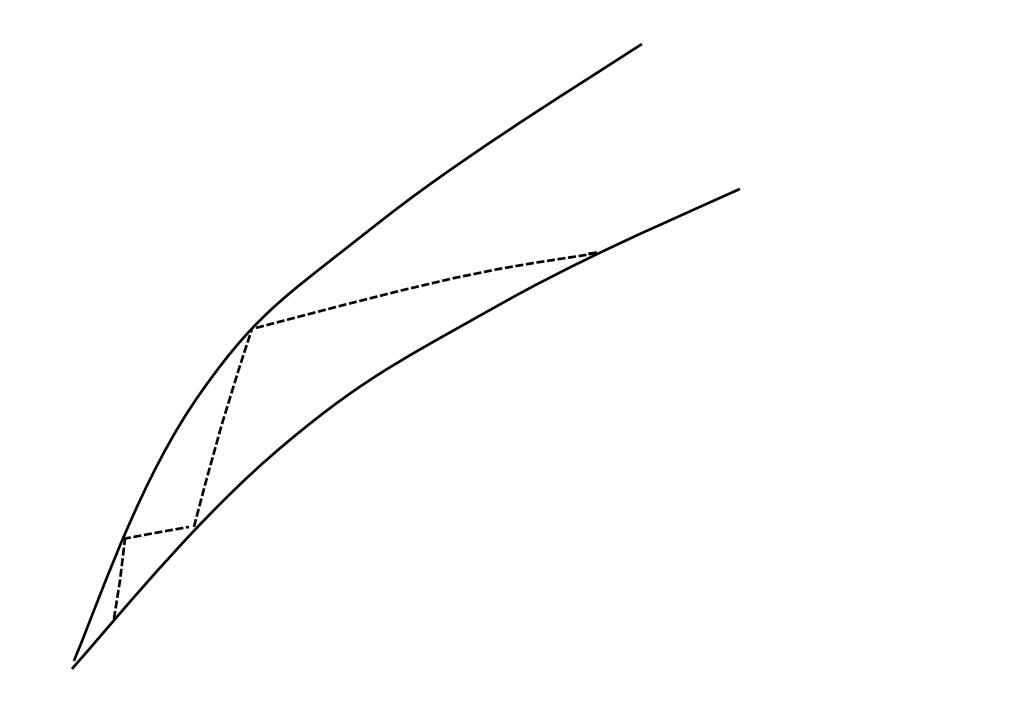}}
\put(-10,38){\small$\bar{P}_2$}
\put(120,158){\small$\textsf{S}$}
\put(60,102){\tiny$C^-_1$}
\put(5,40){\tiny$C^-_2$}
\put(3,27){\tiny$C^+_2$}
\put(30,65){\tiny$C^+_1$}
\put(101,106){\small$P$}
\put(94,101){\small$\bullet$}
\put(145,125){\small$\textsf{W}$}
\put(20,90){\small$\bar{P}_1$}
\put(20,30){\small$\underline{P}_1$}
\put(0,8){\small$\underline{P}_2$}
\end{picture}
\end{center}
\caption{The characteristics between the shock and wedge (3)}\label{fig:CharacteristicReflect}
\end{figure}
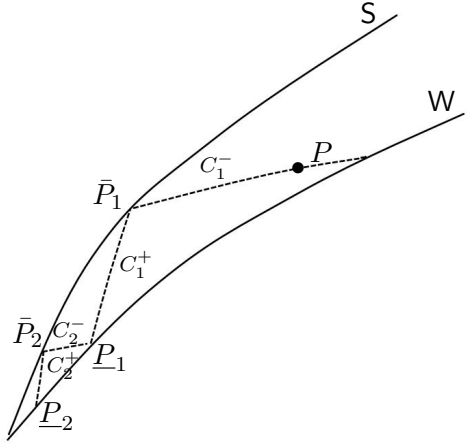

We have
\begin{equation*}
\abs{\bar{\partial}^+ c(P)}\leq \abs{\bar{\partial}^+ c(\bar{P}_1)}+\epsilon M_s\xi(P)\frac{M_sF_1^2}{\xi^2(\bar{P}_1)},
\end{equation*}
\begin{equation*}
\abs{\bar{\partial}^+ c(\bar{P}_1)}\leq \bar{K}\abs{\bar{\partial}^- c(\bar{P}_1)},
\end{equation*}
\begin{equation*}
\abs{\bar{\partial}^- c(\bar{P}_1)}\leq \abs{\bar{\partial}^- c(\underline{P}_1)}+\epsilon M_s\xi(\bar{P}_1)\frac{M_sF_1^2}{\xi^2(\underline{P}_1)},
\end{equation*}
and
\begin{equation*}
\abs{\bar{\partial}^- c(\underline{P}_1)}\leq \abs{\bar{\partial}^+ c(\underline{P}_1)}+\frac{M_b}{\xi^{1+a}(\underline{P}_1)}.
\end{equation*}

Summarising above four, we have
\begin{equation*}
\abs{\bar{\partial}^+ c(P)}\leq \bar{K}\abs{\bar{\partial}^+ c(\underline{P}_1)}+\bar{K}\frac{M_b}{\xi^{1+a}(\underline{P}_1)}+\bar{K}\epsilon M_s\xi(\bar{P}_1)\frac{M_sF_1^2}{\xi^2(\underline{P}_1)}+\epsilon M_s\xi(P)\frac{M_sF_1^2}{\xi^2(\bar{P}_1)}.
\end{equation*}
By Theorem \ref{thm:characteristicestimate}, there exist positive constants $M_s$ and $m_s$ satisfying $M_s>m_s$, so that
\begin{equation*}
\xi(\underline{P}_1),\xi(\bar{P}_1)\in[(1-\epsilon M_s)\xi(P), (1-\epsilon m_s)\xi(P)].
\end{equation*}
Then, we have $\xi(\underline{P}_1)\geq X$ provided $\xi(P)\geq \frac{X}{1-\epsilon M_s}$ and
\begin{equation*}
\begin{split}
\xi(P)\abs{\bar{\partial}^+ c(P)}&\leq \bar{K}\frac{\xi(P)}{\xi(\underline{P}_1)}\xi(\underline{P}_1)\abs{\bar{\partial}^+ c(\underline{P}_1)}+\bar{K}\frac{M_b\xi(P)}{\xi^{1+a}(\underline{P}_1)}\\
&\quad+\bar{K}\epsilon M_s\xi(\bar{P}_1)\frac{M_sF_1^2\xi(P)}{\xi^2(\underline{P}_1)}+\epsilon M_s\xi^2(P)\frac{M_sF_1^2}{\xi^2(\bar{P}_1)}\\
&\leq \bar{K}\frac{1}{1-\epsilon M_s}\xi(\underline{P}_1)\abs{\bar{\partial}^+ c(\underline{P}_1)}+\frac{\bar{K}}{(1-\epsilon M_s)^{1+a}}\frac{M_b}{\xi^a(P)}\\
&\quad+\bar{K}\epsilon M_s(1-\epsilon m_s)\frac{M_sF_1^2}{(1-\epsilon M_s)^2}+\epsilon M_s\frac{M_sF_1^2}{(1-\epsilon M_s)^2}.
\end{split}
\end{equation*}

Noting that $F_1$ is bounded by a constant $M_s$ which is only depending on the wedge boundary $\textsf{W}$ and independent of $\epsilon$, so does $m_s$ and $M_b$, we can choose $\epsilon$ sufficiently small, so that
\begin{equation*}
\begin{split}
\xi(P)\abs{\bar{\partial}^+ c(P)}
&\leq \frac{\bar{K}}{1-\epsilon M_s}F_1+\frac{\bar{K}}{(1-\epsilon M_s)^{1+a}}\frac{M_b}{\xi^a(P)}\\
&\quad+\bar{K}\epsilon M_s\frac{1-\epsilon m_s}{(1-\epsilon M_s)^2}F_1+\epsilon M_s\frac{F_1}{(1-\epsilon M_s)^2}\\
&\leq \frac{\bar{K}+1}{2}F_1+\frac{M_s}{X^a}.
\end{split}
\end{equation*}

Similarly, we have
\begin{equation*}
\xi(P)\abs{\bar{\partial}^- c(P)}\leq \frac{\bar{K}+1}{2}F_1+\frac{M_s}{X^a}.
\end{equation*}
Combining the above two, we have
\begin{equation*}
F(\frac{X}{1-\epsilon M_s})\leq \frac{\bar{K}+1}{2}F(X)+\frac{M_s}{X^a}.
\end{equation*}
Let $X\rightarrow+\infty$, we have
\begin{equation*}
0\leq F_\infty\leq \frac{\bar{K}+1}{2}F_\infty
\end{equation*}
where $F_\infty=\lim\limits_{X\rightarrow+\infty}F(X)$. $\bar{K}\in(0, 1)$ yields $F_\infty=0$ . Therefore,
\begin{equation*}
\lim\limits_{X\rightarrow+\infty}\max_{\xi(P)=X}\xi(P)\abs{\bar{\partial}^\pm c(P)}=\lim\limits_{X\rightarrow+\infty}F(X)=0.
\end{equation*}
The proof is complete.
\end{proof}

\begin{prop}\label{prop:AsymptoticBehaviour}
The solution obtained in Proposition \ref{prop:GlobalExistence} satisfies
\begin{equation*}
\lim\limits_{X\rightarrow+\infty}\max_{\xi(P)=X}\xi(P)\abs{\nabla_{(x, y)} \vec{U}(P)}=0,
\end{equation*}
\begin{equation*}
\lim\limits_{x\rightarrow+\infty}(u, v)=(u_\infty, v_\infty),
\end{equation*}
and 
\begin{equation*}
\lim\limits_{x\rightarrow+\infty}\phi'(x)=\frac{\underline{u}-u_\infty}{v_\infty}.
\end{equation*}
\end{prop}
\begin{proof}
For the first part, the discussion is similar to Lemma \ref{lem:DerivativeEquiv}. For the second part, the proof is similar to Theorem \ref{thm:SpecialSolutionPerturbationEstimate}. Let $\xi(P)=X$, then we have
\begin{equation}\label{eq:EstimateBetweenSWInfty}
\begin{split}
\abs{\vec{U}(P)-\vec{U}(P^s)}+\abs{\vec{U}(P)-\vec{U}(P_w)}&\leq \sup_{\tilde{P}\in P_wP^s}\abs{\nabla_{(x, y)}\vec{U}(\tilde{P})}\abs{P_wP^s}\\
&\leq \epsilon M_s\max_{\xi(P)=X}\xi(P)\abs{\nabla_{(x, y)} \vec{U}(P)}.
\end{split}
\end{equation}
So
\begin{equation*}
\abs{\vec{U}(P)-\vec{U}(P^s)}+\abs{\vec{U}(P)-\vec{U}(P_w)}\rightarrow0,
\end{equation*}
as $\xi(P)\rightarrow+\infty$. On the one hand, $\vec{U}(P^s)$ locates on the shock polar $G(\vec{U}, \epsilon)=0$. On the other hand, $\vec{U}(P_w)$ satisfies
\begin{equation*}
\frac{v}{u}(P_w)=f'\rightarrow f_\infty
\end{equation*}
as $\xi(P)\rightarrow+\infty$. Thus $\vec{U}(P)$ approximates the intersection of shock polar and the line $\frac{v}{u}=f_\infty$ as $\xi(P)\rightarrow+\infty$. The equation \eqref{eq:ApproximateShock} is a direct result of \eqref{eq:ApproximateVelocity}. The proof is complete.
\end{proof}

\subsection{Higher order properties}\label{subsect:HigherOrderProperty}
In this section, we discuss the higher order differentiability of solutions and the approximation of the derivatives.

\begin{thm}\label{thm:HigherSmoothnessoftheSolution}
If
\begin{equation}\label{eq:ThirdOrderRegularW}
f\in C^3([0, \xi_0])
\end{equation}
and it holds \eqref{eq:SupersonicWallCondition} for $\xi\in[0, \xi_0]$, then there exists $\epsilon_1\leq\epsilon_0$, such that Problem 2 can be solved for $0<\epsilon<\epsilon_1$ and $\xi\in[0, \xi_0]$. Moreover, we have
\begin{equation}\label{eq:SecondOrderEstimatesofCharacteristics}
\abs{\bar{\partial}^\pm\bar{\partial}^\pm c}\leq \frac{M_s(\xi_0)}{\xi^2},
\end{equation}
and
\begin{equation}\label{eq:SecondOrderEstimates}
\abs{\nabla^2 (u, v)}\leq \frac{M_s(\xi_0)}{\xi^2}.
\end{equation}
\end{thm}
\begin{proof}
First, by \eqref{eq:SupersonicWallCondition}, the existence of the solution is a direct conclusion in Section \ref{subsect:GlobalSolutionwithoutConvexity}. Moreover, we have the following
\begin{equation}\label{eq:derivativeBoundness}
\abs{\bar{\partial}^+c}+\abs{\bar{\partial}^-c}+\abs{\nabla \vec{U}}\leq \frac{M_s}{\xi}
\end{equation}
and
\begin{equation*}
\abs{P_wP^s}(\xi)\leq \epsilon M_0(\xi_0)\xi
\end{equation*}
for $\xi\leq\xi_0$.

We first deduce the higher order differential relations on $\mathsf{W}$ and $\mathsf{S}$. Put $\bar{\partial}^s$ on \eqref{eq:ShockRelation} to have
\begin{equation*}
\bar{\partial}^s\qnt{\bar{\partial}^+c}=\bar{\partial}^s\qnt{g(u, \epsilon)\bar{\partial}^-c}.
\end{equation*}
Recalling \eqref{eq:sderivative} and \eqref{eq:SonicExpressions}, we have
\begin{equation*}
\bar{\partial}^+\bar{\partial}^+c=\frac{t_-}{t_+}g(u, \epsilon)\qnt{\bar{\partial}^-\bar{\partial}^-c}+R_1(\vec{U}, \epsilon)\bar{\partial}^+\bar{\partial}^-c+R_2(\vec{U}, \epsilon)\bar{\partial}^-\bar{\partial}^+c+R_3(\vec{U}, \epsilon)(\bar{\partial}^+c, \bar{\partial}^-c),
\end{equation*}
where $R_1$, $R_2$ are functions of $\vec{U}$ and $\epsilon$, $R_3$ is quadratic form of $(\bar{\partial}^+c, \bar{\partial}^-c)$ and the coefficients are functions of $(\vec{U}, \epsilon)$. Owing to the characteristic decomposition \eqref{eq:CharacteristicDecomposition}, $R_1, R_2$ and $R_3$ can be summarised to
\begin{equation}\label{eq:shockrelation2nd}
\bar{\partial}^+\bar{\partial}^+c=\frac{t_-}{t_+}g(u, \epsilon)\qnt{\bar{\partial}^-\bar{\partial}^-c}+R_4(\vec{U}, \epsilon)(\bar{\partial}^+c, \bar{\partial}^-c),
\end{equation}
where $R_4$ is quadratic form of $(\bar{\partial}^+c, \bar{\partial}^-c)$ and its coefficients are depending on $(\vec{U}, \epsilon)$.

Similarly, we apply $\bar{\partial}^0$ on \eqref{eq:WallRelationCharacteristic} to have
\begin{equation*}
\bar{\partial}^0\qnt{\bar{\partial}^-c-\bar{\partial}^+c}=\bar{\partial}^0\qnt{(\gamma-1)q\kappa_w}.
\end{equation*}
Recalling $\bar{\partial}^0=\frac{\bar{\partial}^++\bar{\partial}^-}{2\cos\omega}$, we have
\begin{equation*}
\bar{\partial}^-\bar{\partial}^-c-\bar{\partial}^+\bar{\partial}^+c=\bar{\partial}^-\bar{\partial}^+c-\bar{\partial}^+\bar{\partial}^-c+\qnt{R_5(\vec{U}, \epsilon)\bar{\partial}^+c+R_6(\vec{U}, \epsilon)\bar{\partial}^-c}\kappa_w+R_7(\vec{U}, \epsilon)\frac{\dif \kappa_w}{\dif \xi}.
\end{equation*}
Similarly, the above can be reduced to
\begin{equation}\label{eq:wallrelation2nd}
\bar{\partial}^-\bar{\partial}^-c-\bar{\partial}^+\bar{\partial}^+c=R_8(\vec{U}, \epsilon)(\bar{\partial}^+c, \bar{\partial}^-c)+\qnt{R_5(\vec{U}, \epsilon)\bar{\partial}^+c+R_6(\vec{U}, \epsilon)\bar{\partial}^-c}\kappa_w+R_7(\vec{U}, \epsilon)\frac{\dif \kappa_w}{\dif \xi}.
\end{equation}
As the Proposition 3.8 in \cite{LiYangZheng2011JDE}, we can also deduce the higher order characteristic decomposition
\begin{equation}\label{eq:CharacteristicDecompositionSecond}
\begin{cases}
\bar{\partial}^+\qnt{\bar{\partial}^-\bar{\partial}^-c}+R_9(\vec{U})\bar{\partial}^-\bar{\partial}^-c=R_{10}(\vec{U}, \bar{\partial}^+c, \bar{\partial}^-c);\\
\bar{\partial}^-\qnt{\bar{\partial}^+\bar{\partial}^+c}+R_{11}(\vec{U})\bar{\partial}^+\bar{\partial}^+c=R_{12}(\vec{U}, \bar{\partial}^+c, \bar{\partial}^-c).
\end{cases}
\end{equation}
Finally, by \eqref{eq:SonicCharacteristicBoundInductive} and \eqref{eq:VelocityBoundInductive}, for $R_i$ in above, we have
\begin{equation*}
\abs{R_i}\leq \frac{M_s(\xi_0)}{\xi^2}.
\end{equation*}

Then, to show the estimates of $\bar{\partial}^+\bar{\partial}^+c$ and $\bar{\partial}^-\bar{\partial}^-c$, we apply a characteristic reflection scheme as Proposition \ref{lem:DerivativeDecay}. Indeed, by Theorem \ref{thm:characteristicestimate}, for any characteristic between $\textsf{W}$ and $\textsf{S}$, take an example $\widehat{\underline{P}_1\bar{P}_1}$ in Figure \ref{fig:CharacteristicReflect}, we have
$$\xi(\underline{P}_1)\leq\xi(\bar{P}_1)-\frac{8}{\tan\omega_0}\abs{\bar{P}_1(\bar{P}_1)_w}\leq\xi(\bar{P}_1)-\frac{8}{\tan\omega_0}\epsilon M_1\xi(\bar{P}_1).$$
Thus, for fixed $\epsilon>0$, the characteristic finally approaches the origin $\xi=0$, otherwise it leads to a contradiction. We obtain a series $\bar{P}_n$ on the shock $\textsf{S}$ and a series $\underline{P}_n$ on the wedge boundary $\textsf{W}$ with the $C^\pm_n$ characteristics.

Along $C^-_1$, we can integrate the second equation of \eqref{eq:CharacteristicDecompositionSecond} and utilise Gronwall inequality to have
\begin{equation}\label{eq:iteration2nd1}
\abs{\bar{\partial}^+\bar{\partial}^+c(P)}\leq(1+\epsilon M_s(\xi_0))\abs{\bar{\partial}^+\bar{\partial}^+c(\bar{P}_1)}+\frac{\epsilon M_s(\xi_0)}{\xi},
\end{equation}
where we employ the narrow estimate
\begin{equation*}
\mbox{length of $C^-_1$}\leq M_s\abs{P_wP^s}\leq \epsilon M_s\xi(P)\leq \epsilon M_s(\xi_0),
\end{equation*}
which is a conclusion of Theorem \ref{thm:characteristicestimate}.

For \eqref{eq:shockrelation2nd}, we point out that the coefficient
\begin{equation}\label{eq:gSecondS}
\lim\limits_{\epsilon\rightarrow0}\frac{t_-}{t_+}g=\frac{t_-}{t_+}g|_{\epsilon=0}=g(\vec{U}_s, 0),
\end{equation}
where we use \eqref{eq:ShockRelationT}. So, as \eqref{eq:gRange}, by approximate estimate \eqref{eq:SpecialSolutionPerturbationEstimate3}, we can show
\begin{equation*}
\sup_{\xi\in[0, \xi_0]}\abs{\frac{t_-}{t_+}g(\vec{U}(\xi), \epsilon)}\leq \frac{\bar{K}+1}{2},\quad\mbox{for}\quad\xi\in[0, \xi_0]~\mbox{and small $\epsilon$}.
\end{equation*}
So \eqref{eq:shockrelation2nd} yields
\begin{equation}\label{eq:iteration2nd2}
\abs{\bar{\partial}^+\bar{\partial}^+c(\bar{P}_1)}\leq \frac{\bar{K}+1}{2}\abs{\bar{\partial}^-\bar{\partial}^-c(\bar{P}_1)}+\frac{M_s(\xi_0)}{\xi^2}.
\end{equation}
Along $C^+_1$, as \eqref{eq:iteration2nd1} we have
\begin{equation}\label{eq:iteration2nd3}
\abs{\bar{\partial}^-\bar{\partial}^-c(\bar{P}_1)}\leq(1+\epsilon M_s(\xi_0))\abs{\bar{\partial}^-\bar{\partial}^-c(\underline{P}_1)}+\epsilon M_s(\xi_0)\xi.
\end{equation}
Then \eqref{eq:wallrelation2nd} yields
\begin{equation}\label{eq:iteration2nd4}
\abs{\bar{\partial}^-\bar{\partial}^-c(\underline{P}_1)}\leq\abs{\bar{\partial}^+\bar{\partial}^+c(\underline{P}_1)}+\frac{M_s(\xi_0)}{\xi^2},
\end{equation}
where we estimate $\frac{\dif \kappa_w}{\dif \xi}$ by \eqref{eq:ThirdOrderRegularW}. By choosing $\epsilon$ small, we can summarise \eqref{eq:iteration2nd1}-\eqref{eq:iteration2nd4} to have
\begin{equation*}
\abs{\bar{\partial}^+\bar{\partial}^+c(P)}\leq \frac{\bar{K}+3}{4}\abs{\bar{\partial}^+\bar{\partial}^+c(\underline{P}_1)}+\frac{M_s(\xi_0)}{\xi^2}.
\end{equation*}
Iterating this inequality, since $\bar{K}\in(0,1)$, we have
\begin{equation*}
\abs{\bar{\partial}^+\bar{\partial}^+c(P)}\leq \frac{M_s(\xi_0)}{\xi^2}.
\end{equation*}
Similarly, we have
\begin{equation*}
\abs{\bar{\partial}^-\bar{\partial}^-c(P)}\leq \frac{M_s(\xi_0)}{\xi^2}.
\end{equation*}
By the characteristic decomposition \eqref{eq:CharacteristicDecomposition}, we have
\begin{equation*}
\abs{\bar{\partial}^+\bar{\partial}^-c(P)}+\abs{\bar{\partial}^-\bar{\partial}^+c(P)}\leq \frac{M_s(\xi_0)}{\xi^2}.
\end{equation*}
By characteristic expressions, $\nabla^2 (u, v)$ can be expressed by $\bar{\partial}^\pm\bar{\partial}^\pm c, \bar{\partial}^\pm c$ and $\vec{U}$. Thus, we have \eqref{eq:SecondOrderEstimates}.
\end{proof}

Next, we turn to the limit of the derivatives of the solution as $\epsilon\rightarrow0$. Define the limit formal derivatives $\bar{\parbar}^\pm c_s(P_w)$ via the following equations
\begin{equation}\label{eq:FormalDerivativesEquations}
\begin{cases}
\bar{\parbar}^+c_s(P_w)=g(\vec{U}_s(P_w), 0)\bar{\parbar}^-c_s(P_w),\\
\bar{\parbar}^-c_s(P_w)-\bar{\parbar}^+c_s(P_w)=(\gamma-1)q_s(P_w)\kappa_w(P_w).
\end{cases}
\end{equation}
Then, we have

\begin{thm}[The higher order approximation]\label{thm:LocalSolutionApproximateEstimate}
Under the assumption of Theorem \ref{thm:HigherSmoothnessoftheSolution}, if $f'|_{\xi_0}>0$, then there exists $\epsilon_2\leq\epsilon_1$, such that Problem 2 can be solved for $0<\epsilon<\epsilon_2$ and $\xi\in[0, \xi_0]$. In addition,
\begin{equation}\label{eq:CPerturbation}
\abs{\bar{\partial}^\pm c(P)|_{P\in P^sP_w}-\bar{\parbar}^\pm c_s(P_w)}\leq \frac{\epsilon M_s(\xi_0)}{\xi}.
\end{equation}
\end{thm}

\begin{proof}This theorem can be proved in a similar way as \eqref{eq:SpecialSolutionPerturbationEstimate3} in Theorem \ref{thm:SpecialSolutionPerturbationEstimate}. We give an outline here.

By the narrow estimate, we have
\begin{equation}\label{eq:Fluctuation}
\abs{\bar{\partial}^\pm c(P^s)-\bar{\partial}^\pm c(P)}+\abs{\bar{\partial}^\pm c(P_w)-\bar{\partial}^\pm c(P)}\leq \abs{\nabla^2 (u, v)}\abs{P^sP_w}\leq \frac{\epsilon M_s(\xi_0)}{\xi}.
\end{equation}
On $P^s$, in view of \eqref{eq:ShockPolar}, we have
\begin{equation*}
\bar{\partial}^+c(P^s)=g(u, \epsilon)(P^s)\bar{\partial}^-c(P^s).
\end{equation*}
By \eqref{eq:SpecialSolutionPerturbationEstimate3}, \eqref{eq:SonicCharacteristicBoundInductive}, \eqref{eq:Fluctuation} and above, it yields
\begin{equation}\label{eq:Fluctuation1}
\begin{split}
\abs{\bar{\partial}^+c(P)-g(\vec{U}_s(P_w), 0)\bar{\partial}^-c(P)}&\leq \abs{g(u, \epsilon)(P^s)-g(\vec{U}_s(P_w), 0)}\abs{\bar{\partial}^-c(P)}\\
&\quad+\abs{\bar{\partial}^+c(P)-\bar{\partial}^+c(P^s)}\\
&\quad+\abs{g(u, \epsilon)(P^s)}\abs{\bar{\partial}^-c(P^s)-\bar{\partial}^-c(P)}\\
&\leq \frac{\epsilon M_s(\xi_0)}{\xi}.
\end{split}
\end{equation}
On $P_w$, in view of \eqref{eq:WallRelationCharacteristic}, we have
\begin{equation*}
\bar{\partial}^-c(P_w)-\bar{\partial}^+c(P_w)=(\gamma-1)q(P_w)\kappa_w(P_w).
\end{equation*}
As \eqref{eq:Fluctuation1}, by \eqref{eq:Fluctuation}, it yields
\begin{equation}\label{eq:Fluctuation2}
\abs{\bar{\partial}^-c(P)-\bar{\partial}^+c(P)-(\gamma-1)q_s(P_w)\kappa_w(P_w)}\leq \frac{\epsilon M_s(\xi_0)}{\xi}.
\end{equation}
\eqref{eq:FormalDerivativesEquations} can be viewed as linear equations of $(\bar{\parbar}^+ c_s(P_w), \bar{\parbar}^- c_s(P_w))$. \eqref{eq:Fluctuation1} and \eqref{eq:Fluctuation2} indicate that $(\bar{\partial}^+c(P), \bar{\partial}^-c(P))$ serves the linear equations with the same coefficients, while their inhomogeneous terms get perturbations with $\epsilon M_s(\xi_0)$ estimates. This observation yields the perturbation estimate \eqref{eq:CPerturbation} between $(\bar{\parbar}^+ c_s(P_w), \bar{\parbar}^- c_s(P_w))$ and $(\bar{\partial}^+c(P), \bar{\partial}^-c(P))$.
\end{proof}
\begin{comment}
\begin{prop}
When the Mach number of the incoming flow goes to $\infty$, the shock front approaches the wall. Moreover, for a fixed point $P_w\in W$, the state $|(u, v)(P) - (u_s, v_s)(P_w)|<\delta$ for each $P\in P_wP^s$.?????????????
\end{prop}

\subsection{Proof of Theorem \ref{mth1}}
\begin{proof}[Proof of Theorem \ref{mth1}]
By Lemma \ref{lemex} we have the existence of the solution. The asymptotic property follows from Proposition \ref{prop4.8}. As $\underline{M}\rightarrow\infty$, $\epsilon\rightarrow0$, we have by \eqref{eq:NarrowEstimate} that the distance between the shock and the wedge goes to zero.We complete the proof of Theorem \ref{mth1}.
\end{proof}

\begin{rem}
Although we get Theorem \ref{th:LocalSolutionApproximateEstimate}, it means that the solution convergesuniformly to a function depending only on $\xi$ as $\epsilon\rightarrow0$, and its derivatives also converge uniformly. What does this mean? How we describe this kind of solution? Should we name it?
\end{rem}
\end{comment}

\section{Convex wedge and the extinction of shock wave}\label{sect:ConvexWedge}
In this section, we prove Theorem \ref{thm:MainTheorem2}. Under the assumption of convex wedge, in order to show the global existence of the solution, unlike obtaining estimates above, we now need to calculate the sign of the key quantities clearly. In the previous proofs, we calculated the quantities near the limit solution and obtained the required estimate through continuity. For the case $f_\infty=0$, this approach is not feasible here, because many quantities have complex singularities that are difficult to calculate. The handling of this part is the key to our article. We believe that our disposing of this part will provide inspiration for solving general problems.

\subsection{Some sign results along the shock polar}\label{subsect:SignofShockStates}
\begin{thm}\label{thm:ShockPolarI}
For any positive constants $\delta_1>0$, there exists $\bar{\epsilon}_1>0$, such that for any $0<\epsilon<\bar{\epsilon}_1$, $\frac{\gamma-1}{2}\bar{q}+\delta_1<u<\bar{q}$ and $v>\delta_1$, we have
\begin{equation}\label{eq:InnerCondition}
\begin{cases}
(\sin\alpha, -\cos\alpha)\cdot(-G_u, -G_v)>0,\\
(\sin\beta, -\cos\beta)\cdot(-G_u, -G_v)>0,
\end{cases}
\end{equation}
on the shock polar \eqref{eq:ShockPolar}.
\end{thm}
\begin{proof}
On the shock polar \eqref{eq:ShockPolar}, we employ $v$ as the parameter around $(\underline{u}, 0)$ and define
\begin{equation}\label{eq:HJExpressions}
\begin{cases}
H(v, \epsilon):=(\sin\alpha, -\cos\alpha)\cdot(-G_u, -G_v),\\
J(v, \epsilon):=(\sin\beta, -\cos\beta)\cdot(-G_u, -G_v).
\end{cases}
\end{equation}
For $\epsilon=0$, by \eqref{eq:ShockPolarS}
\begin{equation*}
G_u=2u - \bar{q},\quad G_v=2v.
\end{equation*}
In such a case, noticing that \eqref{eq:ShockPolarS} is a circle, we have $G^2_u+G^2_v=\bar{q}^2$, 
\begin{equation*}
\qnt{G_u, G_v}=\bar{q}\qnt{\cos2\tau, \sin2\tau}
\end{equation*}
and \eqref{eq:HJExpressions} becomes
\begin{equation*}
\begin{cases}
H(v, 0)=\bar{q}\sin\beta,\\
J(v, 0)=\bar{q}\sin\alpha.
\end{cases}
\end{equation*}

By Theorem \ref{thm:gRangeS}, we have $\sin\alpha>0$ and $\sin\beta>0$ for the points on the shock polar \eqref{eq:ShockPolarS} satisfying $\frac{\gamma-1}{2}\bar{q}<u<\bar{q}$ and $v>0$. So, by continuity, there exists $\bar{\epsilon}_1>0$, which is depending on $\delta_1$, for the points on the shock polar \eqref{eq:ShockPolarS} satisfying $\epsilon<\bar{\epsilon}_1$, $\frac{\gamma-1}{2}\bar{q}+\delta_1\leq u$ and $v\geq \delta_1$, we have
\begin{equation*}
H(v, \epsilon)>m_g>0,\quad\mbox{and}\quad J(v, \epsilon)>m_g>0
\end{equation*}
where $m_g$ is a constant.
\end{proof}

We can see in the above proof that the lower positive bound of $v$ plays an important role to ensure the validity of \eqref{eq:InnerCondition}. If the wedge boundary approaches a flat, this assumption on $v$ will be invalid. It makes the estimates similar to \eqref{eq:InnerCondition} become much more complex. For this case, we have the following

\begin{thm}\label{thm:ShockPolarIV}
For the adiabatic index $\gamma=2$, there exist positive constants $K_{\bar{q}}<1$, $\delta_2$ and $\bar{\epsilon}_2$, such that for any $0<\epsilon<\bar{\epsilon}_2$, $\bar{q}-\delta_2<u<\underline{u}$ and $0<v<\delta_2$, we have both \eqref{eq:InnerCondition} and 
\begin{equation}\label{eq:gBoundV}
0<g\leq K_{\bar{q}}
\end{equation}
on the shock polar \eqref{eq:ShockPolar}.
\end{thm}
By the theory in quasilinear hyperbolic system, the shock wave curve and rarefaction wave curve are tangent to each other in the second order at the generation point, see \cite[(5.36) in Page 98]{Bressan2000OLSMA}. Thus, to show the first inequality of \eqref{eq:InnerCondition} near the $(\underline{u}, 0)$, we should calculate the third-order difference of the wave curves. It is a very complicated computation. For a fixed supersonic incoming speed $\underline{u}$, \eqref{eq:InnerCondition} has been proved by Taylor expansion for the point on shock polar near $(\underline{u}, 0)$ in \cite[Lemma 3.2]{Chen1998CAMS}. But, for our case, we need to show \eqref{eq:InnerCondition} and \eqref{eq:gBoundV} uniformly for the point around $(\bar{q}, 0)$, which is a vacuum point and $g$ has complicated singularity at this point. The implicit function theorem is not valid and Taylor expansion is difficult to apply in this situation. Here, we choose the adiabatic index $\gamma=2$. Under this assumption, many quantities become polynomials of $(u, v, \underline{u})$. They are easily handled with the software Mathematica.
\begin{proof}
(1) Under the assumption $\gamma=2$, we have the states function $p=\rho^2$ and $c^2=2\rho$. Let $\bar{q}=1$, then the Bernoulli's law takes the form
\begin{equation}\label{eq:BernoulliLaw2}
\frac{u^2+v^2}{2}+2\rho=\frac{1}{2}.
\end{equation}
We consider the quantity $\mathcal{F}$ defined as follows
\begin{equation}\label{eq:Fexpressionuvu}
\mathcal{F}(u, v, \underline{u}):=(u^2 - c^2)\mathcal{G}_v^2- 2 u v \mathcal{G}_u \mathcal{G}_v + (v^2 - c^2)\mathcal{G}_u^2,
\end{equation}
where
\begin{equation}\label{eq:Gexpressionuvu}
\mathcal{G}(u, v, \underline{u})=(\rho u-\epsilon(\underline{u}))(u-\underline{u})+\rho v^2.
\end{equation}

Then the shock polar \eqref{eq:ShockPolar} is equivalent to $\mathcal{G}=0$. The function $\epsilon(\underline{u})$ is a cubic polynomial of $\underline{u}$
\begin{equation}\label{eq:Eexpressionu}
\epsilon(\underline{u})=\underline{\rho}\underline{u}=\frac{1}{4}\qnt{1-\underline{u}^2}\underline{u}.
\end{equation}
Solving \eqref{eq:BernoulliLaw2} gives $\rho=\frac14-\frac{u^2+v^2}{4}$. Inserting it and \eqref{eq:Eexpressionu} into \eqref{eq:Fexpressionuvu} and \eqref{eq:Gexpressionuvu} respectively, it follows that both $\mathcal{F}$ and $\mathcal{G}$ are polynomials of $(u, v, \underline{u})$
\begin{equation*}
\begin{split}
\mathcal{F}&=\frac{1}{32} (-4 u^2 + 20 u^4 - 32 u^6 + 16 u^8 - 4 v^2 + 40 u^2 v^2\\
&\quad- 96 u^4 v^2 + 64 u^6 v^2 + 20 v^4 - 96 u^2 v^4 + 96 u^4 v^4\\
&\quad - 32 v^6 + 64 u^2 v^6 + 16 v^8 + 8 u \underline{u} - 36 u^3 \underline{u} + 52 u^5 \underline{u}\\
&\quad - 24 u^7 \underline{u} - 36 u v^2 \underline{u} + 104 u^3 v^2 \underline{u} - 72 u^5 v^2 \underline{u}\\
&\quad + 52 u v^4 \underline{u} - 72 u^3 v^4 \underline{u} - 24 u v^6 \underline{u} - 4 \underline{u}^2 + 16 u^2 \underline{u}^2\\
&\quad - 21 u^4 \underline{u}^2 + 9 u^6 \underline{u}^2 + 16 v^2 \underline{u}^2 - 34 u^2 v^2 \underline{u}^2\\
&\quad + 21 u^4 v^2 \underline{u}^2 - 13 v^4 \underline{u}^2 + 15 u^2 v^4 \underline{u}^2 + 3 v^6 \underline{u}^2\\
&\quad - 4 u \underline{u}^3 + 12 u^3 \underline{u}^3 - 8 u^5 \underline{u}^3 + 12 u v^2 \underline{u}^3 - 16 u^3 v^2 \underline{u}^3\\
&\quad - 8 u v^4 \underline{u}^3 + 4 \underline{u}^4 - 10 u^2 \underline{u}^4 + 6 u^4 \underline{u}^4 - 14 v^2 \underline{u}^4\\
&\quad + 12 u^2 v^2 \underline{u}^4 + 6 v^4 \underline{u}^4 - \underline{u}^6 + u^2 \underline{u}^6 + 3 v^2 \underline{u}^6),
\end{split}
\end{equation*}
and
\begin{equation*}
\mathcal{G}=\frac{1}{4} (u^2 - u^4 + v^2 - 2 u^2 v^2 - v^4 - 2 u \underline{u} + u^3 \underline{u} + u v^2 \underline{u} + 
\underline{u}^2 + u \underline{u}^3 - \underline{u}^4).
\end{equation*}

Introduce the new variables $\qnt{U, V, Y}$ as follows
\begin{equation}\label{eq:uvyUVY}
u=1-U,\quad v^2=V,\quad \underline{u}=1-Y.
\end{equation}
It is obviously that 
\begin{equation}\label{eq:UVYCorner1}
V\geq0.
\end{equation} 
Since $u\leq\underline{u}\leq\bar{q}=1$, we have 
\begin{equation}\label{eq:UVYCorner2}
0\leq Y\leq U\leq 1.
\end{equation} 
To consider the states bounded by the shock polar \eqref{eq:ShockPolarS}, we have
\begin{equation*}
u^2-u+v^2\leq 0.
\end{equation*}
Inserting \eqref{eq:uvyUVY} into above, it yields 
\begin{equation}\label{eq:UVYCorner3}
V\leq U(1-U)\leq U.
\end{equation}
Simultaneously, $\mathcal{F}$ and $\mathcal{G}$ become
\begin{equation*}
\begin{split}
\mathcal{F}&=\frac{1}{32} (-32 U^3 + 160 U^4 - 298 U^5 + 257 U^6 - 104 U^7 + 16 U^8\\
&\quad + 72 U^2 V - 348 U^3 V + 525 U^4 V - 312 U^5 V + 64 U^6 V\\
&\quad - 50 U V^2 + 279 U^2 V^2 - 312 U^3 V^2 + 96 U^4 V^2 + 11 V^3\\
&\quad - 104 U V^3 + 64 U^2 V^3 + 16 V^4 + 64 U^2 Y - 272 U^3 Y\\
&\quad + 448 U^4 Y - 368 U^5 Y + 150 U^6 Y - 24 U^7 Y - 64 U V Y\\
&\quad + 320 U^2 V Y - 496 U^3 V Y + 318 U^4 V Y - 72 U^5 V Y + 16 V^2 Y\\
&\quad - 128 U V^2 Y + 186 U^2 V^2 Y - 72 U^3 V^2 Y + 18 V^3 Y\\
&\quad - 24 U V^3 Y - 32 U Y^2 + 64 U^2 Y^2 - 36 U^3 Y^2 + 30 U^4 Y^2\\
&\quad - 30 U^5 Y^2 + 9 U^6 Y^2 + 24 V Y^2 - 52 U V Y^2 + 20 U^2 V Y^2\\
&\quad - 36 U^3 V Y^2 + 21 U^4 V Y^2 + 14 V^2 Y^2 - 6 U V^2 Y^2\\
&\quad + 15 U^2 V^2 Y^2 + 3 V^3 Y^2 + 48 U Y^3 - 80 U^2 Y^3 + 28 U^3 Y^3\\
&\quad + 16 U^4 Y^3 - 8 U^5 Y^3 - 48 V Y^3 + 60 U V Y^3 - 16 U^3 V Y^3\\
&\quad - 16 V^2 Y^3 - 8 U V^2 Y^3 - 34 U Y^4 + 41 U^2 Y^4 - 24 U^3 Y^4\\
&\quad + 6 U^4 Y^4 + 43 V Y^4 - 24 U V Y^4 + 12 U^2 V Y^4 + 6 V^2 Y^4\\
&\quad + 12 U Y^5 - 6 U^2 Y^5 - 18 V Y^5 - 2 U Y^6 + U^2 Y^6 + 3 V Y^6),
\end{split}
\end{equation*}
and
\begin{equation*}
\begin{split}
\mathcal{G}&=\frac{1}{4} (-2 U^2 + 3 U^3 - U^4 + 3 U V - 2 U^2 V - V^2 + 4 U Y - 3 U^2 Y\\
& + U^3 Y - V Y + U V Y - 2 Y^2 - 3 U Y^2 + 3 Y^3 + U Y^3 - Y^4).
\end{split}
\end{equation*}

For $\mathcal{F}$ and $\mathcal{G}$, with the constrain \eqref{eq:UVYCorner1}-\eqref{eq:UVYCorner3}, we claim that there is no real solution for 
\begin{equation}\label{eq:FGeq}
\mathcal{F}=0,\quad\mbox{and}\quad\mathcal{G}=0
\end{equation} 
around $(U, V, Y)=(0, 0, 0)$ with
\begin{equation}\label{eq:UVYCorner}
0< Y< U,\quad\mbox{and}\quad 0< V< U.
\end{equation}

In fact, treating $V$ as the independent variable, we divide $\mathcal{F}$ by $\mathcal{G}$ to obtain
\begin{equation*}
\mathcal{F}=\mathcal{P}\mathcal{G}+\mathcal{R}
\end{equation*}
where
\begin{equation*}
\begin{split}
\mathcal{P}&=\frac{1}{8} (17 U - 57 U^2 + 56 U^3 - 16 U^4 - 11 V + 56 U V - 32 U^2 V\\
&\quad -16 V^2 - 5 Y - 9 U Y - 22 U^2 Y + 8 U^3 Y - 2 V Y + 8 U V Y\\
&\quad +20 Y^2 + 35 U Y^2 - U^2 Y^2 - 3 V Y^2 - 29 Y^3 - 11 U Y^3 + 10 Y^4)
\end{split}
\end{equation*}
and
\begin{equation*}
\begin{split}
\mathcal{R}&=\frac{1}{32} (2 U^3 - 5 U^4 + 2 U^5 - U^2 V + 2 U^3 V - 14 U^2 Y + 4 U^3 Y\\
&\quad + 14 U^4 Y - 6 U^5 Y + 12 U V Y + 2 U^2 V Y - 6 U^3 V Y + 22 U Y^2\\
&\quad + 62 U^2 Y^2 - 19 U^3 Y^2 - 10 U^4 Y^2 + 6 U^5 Y^2 - 3 V Y^2\\
&\quad - 29 U V Y^2 + 2 U^2 V Y^2 + 6 U^3 V Y^2 - 10 Y^3 - 116 U Y^3\\
&\quad - 135 U^2 Y^3 + 21 U^3 Y^3 - 2 U^4 Y^3 - 2 U^5 Y^3 + V Y^3\\
&\quad + 27 U V Y^3 - 6 U^2 V Y^3 - 2 U^3 V Y^3 + 55 Y^4 + 261 U Y^4\\
&\quad + 139 U^2 Y^4 - 7 U^3 Y^4 + 3 U^4 Y^4 + 3 V Y^4 - 11 U V Y^4\\
&\quad + 3 U^2 V Y^4 - 123 Y^5 - 271 U Y^5 - 63 U^2 Y^5 - U^3 Y^5 - V Y^5\\
&\quad + U V Y^5 + 127 Y^6 + 125 U Y^6 + 11 U^2 Y^6 - 59 Y^7 - 21 U Y^7\\
&\quad + 10 Y^8).
\end{split}
\end{equation*}

Suppose to the contrary that \eqref{eq:FGeq} admits real solutions around $(U, V, Y)=(0, 0, 0)$ satisfying \eqref{eq:UVYCorner}. Then it follows from \eqref{eq:Fexpressionuvu} that 
\begin{equation*}
\mathcal{G}=0,\quad\mbox{and}\quad\mathcal{R}=0.
\end{equation*}
Note that $\mathcal{G}=0$ is a quadratic equation of $V$ and $\mathcal{R}=0$ is a linear equation of $V$. We obtain by $\mathcal{G}=0$ that 
\begin{equation*}
V=\frac{\mathcal{G}_r\pm\sqrt{\mathcal{G}_i}}{2},
\end{equation*}
where
\begin{equation*}
\mathcal{G}_r=3 U - 2 U^2 - Y + U Y,
\end{equation*}
and
\begin{equation*}
\mathcal{G}_i=U^2 + 10 U Y - 2 U^2 Y - 7 Y^2 - 14 U Y^2 + U^2 Y^2 + 12 Y^3 + 
4 U Y^3 - 4 Y^4.
\end{equation*}

Next, we go on to solve $V$ from $\mathcal{R}=0$. Note that $\mathcal{R}=\mathcal{R}_1V+\mathcal{R}_0$, where 
\begin{equation*}
\begin{split}
\mathcal{R}_1&=-U^2 + 2 U^3 + 12 U Y + 2 U^2 Y - 6 U^3 Y - 3 Y^2 - 29 U Y^2\\
&\quad + 2 U^2 Y^2 + 6 U^3 Y^2 + Y^3 + 27 U Y^3 - 6 U^2 Y^3 - 2 U^3 Y^3\\
&\quad + 3 Y^4 - 11 U Y^4 + 3 U^2 Y^4 - Y^5 + U Y^5,
\end{split}
\end{equation*}
and 
\begin{equation*}
\begin{split}
\mathcal{R}_0&=2 U^3 - 5 U^4 + 2 U^5 - 14 U^2 Y + 4 U^3 Y + 14 U^4 Y - 6 U^5 Y\\
&\quad + 22 U Y^2 + 62 U^2 Y^2 - 19 U^3 Y^2 - 10 U^4 Y^2 + 6 U^5 Y^2\\
&\quad - 10 Y^3 - 116 U Y^3 - 135 U^2 Y^3 + 21 U^3 Y^3 - 2 U^4 Y^3\\
&\quad - 2 U^5 Y^3 + 55 Y^4 + 261 U Y^4 + 139 U^2 Y^4 - 7 U^3 Y^4\\
&\quad + 3 U^4 Y^4 - 123 Y^5 - 271 U Y^5 - 63 U^2 Y^5 - U^3 Y^5 + 127 Y^6\\
&\quad + 125 U Y^6 + 11 U^2 Y^6 - 59 Y^7 - 21 U Y^7 + 10 Y^8.
\end{split}
\end{equation*}
The equation $\mathcal{R}=0$ can be solved for two cases. 

{\bf Case 1}:
\begin{equation*}
\mathcal{R}_1=0 \quad \text{and}\quad \mathcal{R}_0=0.
\end{equation*}

{\bf Case 2}:
\begin{equation*}
V=-\frac{\mathcal{R}_0}{\mathcal{R}_1}.
\end{equation*}

For {\bf Case 1}, as $(U, Y)\rightarrow(0, 0)$, in view of \eqref{eq:UVYCorner}, the solution $(U, Y)$ of $\mathcal{R}_1=0$ obeys
\begin{equation*}
\frac{U}{Y}\rightarrow \mbox{a root of} \quad{x^2-12x+3=0},
\end{equation*}
while the solution $(U, Y)$ of $\mathcal{R}_0=0$ obeys
\begin{equation*}
\frac{U}{Y}\rightarrow \mbox{a root of} \quad{x^3-7x^2+11x-5=0}.
\end{equation*}
These two cannot coexist for $(U, Y)$ near $(0, 0)$.

For {\bf Case 2}, we have
\begin{equation*}
\frac{\mathcal{G}_r\pm\sqrt{\mathcal{G}_i}}{2}=-\frac{\mathcal{R}_0}{\mathcal{R}_1}.
\end{equation*}
This leads to
\begin{equation*}
\begin{split}
&-4 (U - Y)^4 (-2 + Y) (-1 + Y)^3 Y (-U - 11 Y + U Y + 5 Y^2) \\
&\quad\cdot(1 - 
2 U - 5 Y - 4 U Y + 13 Y^2 + 2 U Y^2 - 5 Y^3)=0.
\end{split}
\end{equation*}
It cannot hold for $0<Y<U<<1$.

So we have $\mathcal{F}\neq0$ around $(U, V, Y)=(0, 0, 0)$ provided $\mathcal{G}=0$ and \eqref{eq:UVYCorner}. Since $\mathcal{G}=0$ is the shock polar, we have
\begin{equation*}
\frac{\mathcal{G}_v}{\mathcal{G}_u}=\frac{G_v}{G_u}=\tan \qnt{k+\frac{\pi}{2}}.
\end{equation*}
Then
\begin{equation*}
\begin{split}
&\sgn\qnt{\qnt{\tan \qnt{k+\frac{\pi}{2}}-\tan\alpha}\qnt{\tan \qnt{k+\frac{\pi}{2}}-\tan\beta}}\\
&=\sgn\qnt{\qnt{\frac{G_v}{G_u}-\lambda_+}\qnt{\frac{G_v}{G_u}-\lambda_-}}\\
&=\sgn\qnt{\qnt{\frac{\mathcal{G}_v}{\mathcal{G}_u}-\lambda_+}\qnt{\frac{\mathcal{G}_v}{\mathcal{G}_u}-\lambda_-}}\\
&=\sgn\mathcal{F}.
\end{split}
\end{equation*}
By Theorem \ref{thm:ShockPolarI}, since
\begin{equation*}
\tan \qnt{k+\frac{\pi}{2}}-\tan\alpha>0\quad\mbox{and}\quad\tan \qnt{k+\frac{\pi}{2}}-\tan\beta>0,
\end{equation*}
for $(u, v)$ on $\mathcal{C}_s$ with $u>\frac{\gamma-1}{2}\bar{q}$ and $v>0$, we have $\mathcal{F}>0$. On $\frac{\gamma-1}{2}\bar{q}<u<\bar{q}$ and $v=0$, we have
\begin{equation*}
\tan \qnt{k+\frac{\pi}{2}}-\tan\alpha=0\quad\mbox{and}\quad\tan \qnt{k+\frac{\pi}{2}}-\tan\beta>0.
\end{equation*}
Therefore, $\tan\qnt{k+\frac{\pi}{2}}-\tan\alpha$ and $\tan\qnt{k+\frac{\pi}{2}}-\tan\beta$ cannot change sign around $(u, v)=(\bar{q}, 0)$ for $u<\bar{q}$ and $v>0$. (1) is proved.

(2) Although it holds Theorem \ref{thm:gRangeS}, which confirms that $g$ obeys \eqref{eq:gEstimateS1} for $\epsilon=0$ and $v>0$ with \eqref{eq:gEstimateF1}, we cannot apply the continuity of $g(\vec{U}, \epsilon)$ to get its bound around $\vec{U}=(u, v)=(\bar{q}, 0)$ for $\epsilon<<1$. Because $g(\vec{U}, \epsilon)$ has singularity at $\vec{U}=(\bar{q}, 0)$ and is multivalued near $\vec{U}=(\bar{q}, 0)$. To see this, on the one hand, for $\epsilon=0$ and $u\rightarrow\bar{q}-0$, we have \eqref{eq:gLimitS}. On the other hand, by the results in \cite{CourantFriedrichs1976AMS}, we have
\begin{equation*}
\lim\limits_{v\rightarrow 0+,\epsilon>0}\cot k=\lim\limits_{v\rightarrow 0+,\epsilon>0}\frac{G_v}{G_u}=\tan\alpha.
\end{equation*}
This yields that 
\begin{equation*}
\lim\limits_{v\rightarrow 0+,\epsilon>0}g(\vec{U}, \epsilon)=\lim\limits_{v\rightarrow 0+,\epsilon>0}\frac{\sin(s-\alpha)\cos(k-\alpha)}{\sin(\beta-s)\cos(k-\beta)}=0.
\end{equation*}
$g$ is presented in a way similar to the function $\frac{y}{x+y}$ as follows
\begin{equation*}
\lim\limits_{x=0,y\rightarrow0+}\frac{y}{x+y}=1\neq0=\lim\limits_{y=0,x\rightarrow0+}\frac{y}{x+y}.
\end{equation*}

Now, we begin to get the bound of $g$. On the one hand, we have
\begin{equation*}
\beta\leq\tau\leq s\leq\alpha,
\end{equation*}
which yields
\begin{equation}\label{eq:sabBound}
0\leq\frac{\sin(s-\alpha)}{\sin(\beta-s)}\leq1.
\end{equation}
On the other hand, noting that $n:=\frac{\pi}{2}+k$ is the angle of the normal direction of the shock polar, we have
\begin{equation*}
\frac{\cos(k-\alpha)}{\cos(k-\beta)}=\frac{\sin(n-\alpha)}{\sin(n-\beta)}.
\end{equation*}
By the first part of this theorem, we have
\begin{equation*}
\frac{\sin(n-\alpha)}{\sin(n-\beta)}\in(0, 1).
\end{equation*}
Our mission is to show that there exists a constant $K_{\bar{q}}\in(0, 1)$ such that
\begin{equation}\label{eq:nabBound}
0<\frac{\sin(n-\alpha)}{\sin(n-\beta)}<K_{\bar{q}},
\end{equation}
for $(u, v)$ around $(\bar{q}, 0)$. Then, as a result, \eqref{eq:gBoundV} would be proved by combining \eqref{eq:sabBound} and \eqref{eq:nabBound}. To do it, first of all, we have $\lim\limits_{(u, v)\rightarrow(\bar{q}, 0)}\sin(n-\alpha)=0$. Then for $(u, v)$ around $(\bar{q}, 0)$, $c=\sqrt{\frac{\bar{q}^2}{2}-\frac{u^2+v^2}{2}}$ is small. So does the quantity $\omega=\arcsin\frac{c}{\sqrt{u^2+v^2}}$. Due to $\beta=\alpha-2\omega$, we have by equivalent infinitesimals the following
\begin{equation*}
\frac{\sin(n-\alpha)}{\sin(n-\beta)}\sim \frac{n-\alpha}{n-\beta}=\frac{n-\alpha}{n-\alpha+2\omega}.
\end{equation*}
Therefore, \eqref{eq:nabBound} is equivalent to 
\begin{equation*}
0\leq\frac{n-\alpha}{\omega}\leq M<\infty
\end{equation*}
for some positive constant $M$ around $(\bar{q}, 0)$. By equivalent infinitesimals again, to avoid the difficulties associated with the square root terms as the first part, we instead show the boundedness of $\mathcal{K}$ defined as follows
\begin{equation*}
\begin{split}
\mathcal{K}&=\frac{\qnt{\tan n-\tan\alpha}\qnt{\tan n-\tan\beta}}{\tan^2\omega}\\
&=\frac{\mathcal{G}_u^2(u^2-c^2)\qnt{\tan \qnt{k+\frac{\pi}{2}}-\tan\alpha}\qnt{\tan \qnt{k+\frac{\pi}{2}}-\tan\beta}}{\mathcal{G}_u^2(u^2-c^2)\tan^2\omega}\\
&=\frac{\mathcal{F}}{\mathcal{G}_u^2(u^2-c^2)\tan^2\omega}.
\end{split}
\end{equation*}

Firstly, for $(u, v)$ around $(\bar{q}, 0)=(1, 0)$, we have $(U, V, Y)$ around $(0, 0, 0)$ and
\begin{equation}\label{eq:omegaU}
(u^2-c^2)\tan^2\omega\sim c^2=\frac{1}{2}(1-u^2-v^2)=\frac{1}{2}(2U-U^2-V)\geq \frac{U}{4}
\end{equation}
because of \eqref{eq:UVYCorner}.

Secondly, for $\mathcal{G}$, on the one hand, around $(U, V, Y)=(0, 0, 0)$, we complete the square of its principal part (second order terms) as follows
\begin{equation*}
\begin{split}
&\frac{1}{4} (-2 U^2 + 3 U V - V^2 + 4 U Y - V Y - 2 Y^2)\\
&=\frac{1}{8}(-4 U^2 + 6 U V - 2V^2 + 8 U Y - 2V Y - 4 Y^2)\\
&=\frac{1}{8}(-3 U^2- U^2 + 6 U V - 2V^2 + 8 U Y - 2V Y - 4 Y^2)\\
&=\frac{1}{8}(-3 U^2- (U^2 - 6 U V- 8 U Y)- 2V^2 - 2V Y - 4 Y^2)\\
&=\frac{1}{8}(-3 U^2- (U^2 - 6 U V- 8 U Y)- 2V^2 - 2V Y - 4 Y^2)\\
&=\frac{1}{8}(-3 U^2- (U^2 - 6 U V- 8 U Y+(3V+4Y)^2)\\
&\quad+(3V+4Y)^2- 2V^2 - 2V Y - 4 Y^2)\\
&=\frac{1}{8}(-3 U^2- (U^2 - 6 U V- 8 U Y+(3V+4Y)^2)+7V^2 +22V Y +12 Y^2)\\
&=\frac{1}{8}(-3 U^2- (U-3V-4Y)^2+7V^2 +22V Y +12 Y^2).
\end{split}
\end{equation*}
So for $U^2+V^2\leq\Delta^2$ where $\Delta<<1$, repeat above algebraic manipulating and absorb the higher order terms, we have
\begin{equation*}
\mathcal{G}=\frac{1}{8}(-\tilde{3} U^2- (\tilde{1}U-\tilde{3}V-\tilde{4}Y)^2+\tilde{7}V^2 +\tilde{22}V Y +\tilde{12} Y^2),
\end{equation*}
where we denotes
\begin{equation*}
\tilde{i}=i+\mathcal{O}(\Delta)
\end{equation*}
and $\mathcal{O}(\Delta)$ is the infinitesimals of the same order of $\Delta$. Thus $\mathcal{G}=0$ leads to
\begin{equation*}
\tilde{3} U^2+ (\tilde{1}U-\tilde{3}V-\tilde{4}Y)^2=\tilde{7}V^2 +\tilde{22}V Y +\tilde{12} Y^2
\end{equation*}
Since $U>V>0$ and $U>Y>0$, above yields
\begin{equation*}
2U^2\leq\tilde{3} U^2\leq \tilde{7}V^2 +\tilde{22}V Y +\tilde{12} Y^2\leq 18(V+Y)^2,
\end{equation*}
provided $\Delta<<1$. Finally, due to $0<Y<U$ as in \eqref{eq:UVYCorner}, we get
\begin{equation}\label{eq:VYequU}
U\leq3(V+Y).
\end{equation}
On the other hand,
\begin{equation*}
\begin{split}
\mathcal{G}&=\frac{1}{4}\qnt{V (3 U - 2 U^2 - Y + U Y-V)-(U - Y)^2 (2 - 3 U + U^2 - 3 Y + U Y + Y^2)}\\
&=0.
\end{split}
\end{equation*}
So we have the expression
\begin{equation*}
(U - Y)^2=\frac{V (3 U - 2 U^2 - Y + U Y-V)}{2 - 3 U + U^2 - 3 Y + U Y + Y^2},
\end{equation*}
where $0<Y<U$ and $0<V<U$. 
Due to \eqref{eq:UVYCorner}, we have
\begin{equation}\label{eq:UmYequVU}
\frac{VU}{6}\leq(U - Y)^2\leq6VU
\end{equation}
provided $\Delta<<1$.

Thirdly, to calculate $\mathcal{G}_u^2$, we have
\begin{equation}\label{eq:Gu2decomposition}
\mathcal{G}_u^2=\mathcal{G}_u^2+2\mathcal{G}=\qnt{\mathcal{G}_u^2}_1+\qnt{\mathcal{G}_u^2}_2
\end{equation}
where
\begin{equation*}
\begin{split}
\qnt{\mathcal{G}_u^2}_1&=\frac{1}{16} (U - Y)^2 (-48 U + 105 U^2 - 72 U^3 + 16 U^4 + 54 U Y\\
&\quad - 
42 U^2 Y + 8 U^3 Y + 9 Y^2 - 24 U Y^2 + 9 U^2 Y^2 - 6 Y^3\\
&\quad + 
2 U Y^3 + Y^4),
\end{split}
\end{equation*}
and
\begin{equation*}
\begin{split}
\qnt{\mathcal{G}_u^2}_2&=\frac{1}{16} V (70 U^2 - 96 U^3 + 32 U^4 + V - 24 U V + 16 U^2 V + 16 Y\\
&\quad - 
68 U Y + 84 U^2 Y - 32 U^3 Y + 6 V Y - 8 U V Y - 10 Y^2\\
&\quad + 
12 U Y^2 + 6 U^2 Y^2 + V Y^2 - 8 U Y^3 + 2 Y^4).
\end{split}
\end{equation*}
By \eqref{eq:UmYequVU},
\begin{equation*}
\abs{\qnt{\mathcal{G}_u^2}_1}\leq CVU^2.
\end{equation*}
By \eqref{eq:VYequU}, noting the principal part in the bracket of $\qnt{\mathcal{G}_u^2}_2$ is $V+16Y$, we have
\begin{equation*}
\qnt{\mathcal{G}_u^2}_2\geq \frac{1}{100}VU.
\end{equation*}
Combining above two, for \eqref{eq:Gu2decomposition}, we have
\begin{equation*}
\mathcal{G}_u^2\geq \frac{1}{200}VU.
\end{equation*}
Thus, for the denominator part of $\mathcal{K}$, combining \eqref{eq:omegaU} with above, we have
\begin{equation}\label{eq:Kdenominator}
\mathcal{G}_u^2(u^2-c^2)\tan^2\omega\geq \frac{1}{1000}VU^2.
\end{equation}

Fourthly, for $\mathcal{F}$, we have
\begin{equation*}
\begin{split}
\mathcal{F}&=\frac{1}{32} (-2 + U) U (U - Y)^2 (4 - 9 U + 4 U^2 - 3 Y + U Y + Y^2)^2\\
&\quad+\frac{1}{32} V (72 U^2 - 348 U^3 + 525 U^4 - 312 U^5 + 64 U^6 - 50 U V\\
&\quad +279 U^2 V - 312 U^3 V + 96 U^4 V + 11 V^2 - 104 U V^2\\
&\quad +64 U^2 V^2 + 16 V^3 - 64 U Y + 320 U^2 Y - 496 U^3 Y + 318 U^4 Y\\
&\quad -72 U^5 Y + 16 V Y - 128 U V Y + 186 U^2 V Y - 72 U^3 V Y\\
&\quad +18 V^2 Y - 24 U V^2 Y + 24 Y^2 - 52 U Y^2 + 20 U^2 Y^2\\
&\quad -36 U^3 Y^2 + 21 U^4 Y^2 + 14 V Y^2 - 6 U V Y^2 + 15 U^2 V Y^2\\
&\quad +3 V^2 Y^2 - 48 Y^3 + 60 U Y^3 - 16 U^3 Y^3 - 16 V Y^3\\
&\quad -8 U V Y^3 + 43 Y^4 - 24 U Y^4 + 12 U^2 Y^4 + 6 V Y^4 - 18 Y^5\\
&\quad +3 Y^6)
\end{split}
\end{equation*}
Inserting \eqref{eq:UmYequVU} into the first term on the right hand side of above and applying \eqref{eq:UVYCorner} directly to the second term, we have
\begin{equation*}
\abs{\mathcal{F}}\leq 1000VU^2,
\end{equation*}
for $0<V<U<<1$ and $0<Y<U<<1$. In view of \eqref{eq:Kdenominator}, we have
\begin{equation*}
\abs{\mathcal{K}}\leq 10^6.
\end{equation*}
We complete our proof.
\end{proof}
Combine Theorem \ref{thm:ShockPolarI} and \ref{thm:ShockPolarIV}, we have 
\begin{thm}\label{thm:gRange0}
For any positive constants $\delta_0\leq\delta_2$, there exists $\bar{\epsilon}_0(\delta_0)<\min\set{\bar{\epsilon}_1, \bar{\epsilon}_2}$, which is depending on $\delta_0$, and constant $\underline{K}\in(0, 1)$, such that for any $0<\epsilon<\bar{\epsilon}_0$, $\frac{\gamma-1}{2}\bar{q}+\delta_0<u<\underline{u}$ and $v>0$, we have

(1) \eqref{eq:InnerCondition} holds on the shock polar \eqref{eq:ShockPolar};

(2) $g$ satisfies
\begin{equation}\label{eq:gBoundV2}
0<g\leq \underline{K}<1.
\end{equation}
\end{thm}

Then, with the above theorem, we can prove the following

\begin{thm}[No vacuum]\label{thm:NoVacuum}
Under the assumption of Theorem \ref{thm:gRange0}, we consider the states $(u, v)$ along the $C^\pm$ characteristics issuing from the shock with $v>0$. 

(1) Along each $C^-$ characteristic, if $\tau$ decreases, then the sonic $c$ decreases and the states $(u, v)$ locate in the domain on the $(u, v)$ plane bounded by the shock polar \eqref{eq:ShockPolar} provided $\tau\geq0$;

(2) Along each $C^+$ characteristic, if $\tau$ decreases, then the sonic $c$ increases and the states $(u, v)$ locate in the domain on the $(u, v)$ plane bounded by the shock polar \eqref{eq:ShockPolar} provided $\tau\geq0$.
\end{thm}

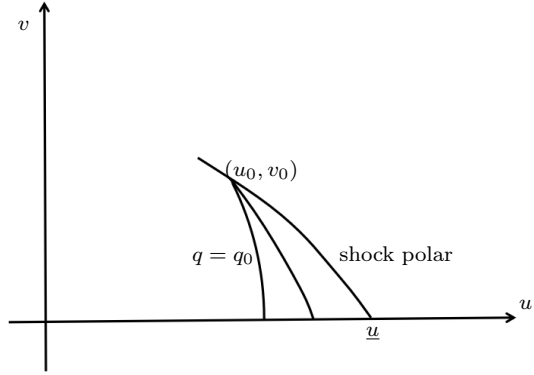
\begin{figure}[htbp]
\setlength{\unitlength}{1bp}
\begin{center}
\begin{picture}(200,150)
\put(-30,-10){\includegraphics[scale=0.66]{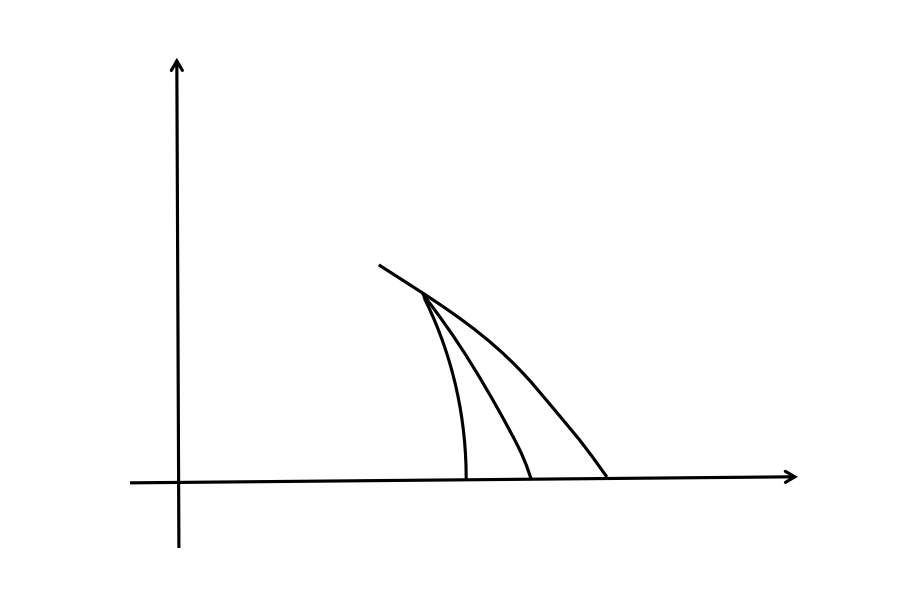}}
\put(75,43){\tiny$q=q_0$}
\put(197,25){\tiny$u$}
\put(87,75){\tiny$(u_0, v_0)$}
\put(130,43){\tiny$\mbox{shock polar}$}
\put(140,15){\tiny$\underline{u}$}
\put(10,130){\tiny$v$}
\end{picture}
\end{center}
\caption{bounded by the shock polar}\label{fig:ShockpolarInner}
\end{figure}

\begin{proof}
We use Riemann invariants of \eqref{eq:EulerEquations} (see \cite{CourantFriedrichs1976AMS},\cite{Chen2014Interaction}) to prove this theorem. Consider a Riemann invariant $F(u, v)$ for $C^-$ characteristic, we have
\begin{equation*}
\bar{\partial}^-F(u, v)=F_u\bar{\partial}^-u+F_v\bar{\partial}^-v=0.
\end{equation*}
So $$F(u, v)=\rm{constant}$$ along the $C^-$ characteristic. By the equation $\bar{\partial}^-u+\lambda_+\bar{\partial}^-v=0$, we have 
$$F_u:F_v=1:\lambda_+.$$ 
By the above equation, contour lines $$F(u, v)=\rm{constant}$$ are the integral curves defined by the vector fields $(\lambda_+, -1)$ or $(\sin\alpha, -\cos\alpha)$. The conclusions follow by a comparison of $(\sin\alpha, -\cos\alpha)$ with the following vectors:
\begin{itemize}
\item the tangent vector of circle $q=\mbox{constant}$: $(\sin\tau, -\cos\tau)$;
\item the inner normal vector of circle $q=\mbox{constant}$: $(-\cos\tau, -\sin\tau)$;
\item the inner normal vector of the shock polar: $(-G_u, -G_v)$.
\end{itemize}
In fact, taking inner products of $(\sin\alpha, -\cos\alpha)$ and the above three vectors respectively, we have 
\begin{equation}\label{eq:RIDirectionsCompare1}
\begin{split}
(\sin\alpha, -\cos\alpha)\cdot(\sin\tau, -\cos\tau)=\cos\omega&>0;\\
(\sin\alpha, -\cos\alpha)\cdot(-\cos\tau, -\sin\tau)=-\sin\omega&<0;\\
(\sin\alpha, -\cos\alpha)\cdot(-G_u, -G_v)&>0,
\end{split}
\end{equation}
where Theorem \ref{thm:ShockPolarI} yields the third inequality. The first inequality of \eqref{eq:RIDirectionsCompare1} indicates that $\tau$ decreases in $(\sin\alpha, -\cos\alpha)$ direction. Indeed, it follows from the fact that
\begin{equation*}
(\sin\alpha, -\cos\alpha)\nabla_{(u, v)}\tau=(\sin\alpha, -\cos\alpha)\cdot\frac{1}{q\cos\tau}\qnt{-\sin\tau, \cos\tau}=-\frac{\cos\omega}{q\cos\tau}<0.
\end{equation*}
In a similar way, the second inequality of \eqref{eq:RIDirectionsCompare1} indicates that $q$ increases in $(\sin\alpha, -\cos\alpha)$ direction. The third inequality of \eqref{eq:RIDirectionsCompare1} indicates that $(\sin\alpha, -\cos\alpha)$ inserts the domain bounded by the shock polar. Since the curve $F(u, v)=\rm{constant}$ is generated by the vector fields $(\sin\alpha, -\cos\alpha)$, combine the first two conclusions above, along these curves, if $\tau$ decreases, then $q$ increases and, consequently, $c$ decreases.

Meanwhile, in view of the third conclusion of \eqref{eq:RIDirectionsCompare1}, the curve must insert the shock polar provided $c$ decreases. Recall that along each $C^-$ characteristic, the states must lie on the curve $F(u, v)=\rm{constant}$. Summarising the above, we arrive at the first conclusion of this theorem.

The second conclusion can be proved in the same way. Thus we complete the proof.
\end{proof}

\subsection{Second initial value condition}
\begin{lem}\label{lem:SecondInitialValueConditionS}
Under the assumption of {\bf Case B}, for some $x_0>0$ with $\xi_0=\xi(x_0, f(x_0))$, we have 
\begin{equation}\label{eq:SecondInitialPointCondition}
0<f'(x_0)<\sqrt{\frac{3-\gamma}{\gamma-1}},\quad\mbox{and}\quad f''(x_0)<0,
\end{equation} 
then there exist positive constants $\delta_f$ and $\epsilon_s\leq\bar{\epsilon}_0(\delta_f)$, so that we have

(1)
\begin{equation}\label{eq:SecondInitialValueConditionuS}
u_s(\xi_0)>\frac{\gamma-1}{2}\bar{q}+2\delta_f,\quad v_s(\xi_0)>2\delta_f
\end{equation}
and
\begin{equation}\label{eq:SecondInitialValueConditionqS}
q_s(\xi_0)>c_0+2\delta_f,
\end{equation}
where $c_0:=\frac{\gamma-1}{\gamma+1}\bar{q}$ is the critical sonic.

(2) We have
\begin{equation}\label{eq:SecondInitialValueConditionSonicS}
\bar{\parbar}^\pm c_s(\xi_0)<0.
\end{equation}

(3) For $\underline{K}\in(0, 1)$ introduced in Theorem \ref{thm:gRange0} and $\epsilon<\epsilon_s$,
\begin{equation}\label{eq:PositiveG}
g(\vec{U}, \epsilon)\in(0, \underline{K}),
\end{equation}
where $u\in(\frac{\gamma-1}{2}\bar{q}+\delta_f, \underline{u})$ and $v>0$ satisfy \eqref{eq:ShockPolar}.

(4) We have \eqref{eq:InnerCondition} for any $\epsilon<\epsilon_s$, $v>0$ and $\frac{\gamma-1}{2}\bar{q}+\delta_f<u<\underline{u}$.
\end{lem}

\begin{proof}
The existence of $x_0$ is a direct result of the assumption for {\bf Case B}. 

(1) By the expression \eqref{eq:LimitSolution}, it is trivial.

(2) In \eqref{eq:FormalDerivativesEquations}, $g(\vec{U}_s(P), 0)\in(0, 1)$ and $\kappa_w(P)<0$ lead to \eqref{eq:SecondInitialValueConditionSonicS}. Consequently, it is a direct conclusion of Theorem \ref{thm:gRangeS}.

(3) and (4) follow directly from Theorem \ref{thm:gRange0}.

We have complete the proof.
\end{proof}

By Lemma \ref{lem:SecondInitialValueConditionS}, the approximation estimates \eqref{eq:SpecialSolutionPerturbationEstimate3} and \eqref{eq:CPerturbation} yield the following

\begin{thm}[Second initial condition]\label{thm:SecondInitialCondition}
There exists a point $\epsilon_p\leq\epsilon_s$, such that for $0<\epsilon\leq\epsilon_p$, we have

(1) {\bf Problem 2} can be solved in $\xi\in[0, \xi_0]$ with
\begin{equation}\label{eq:SecondOriginCondition}
u(P)|_{\xi(P)=\xi_0}>\frac{\gamma-1}{2}\bar{q}+\delta_f,\quad v(P)|_{\xi(P)=\xi_0}>\delta_f,
\end{equation}
and
\begin{equation}\label{eq:SecondOriginConditionSupersonic}
q(P)|_{\xi(P)=\xi_0}>c_0+\delta_f.
\end{equation}

(2) We have
\begin{equation}\label{eq:SecondOriginConditionSonic}
\bar{\partial}^\pm c(P)|_{\xi(P)=\xi_0}<0.
\end{equation}
\end{thm}

\subsection{Global solution}
In this part, we prove the global existence of the solution for {\bf Case B} by showing {\it a prioir} estimates.
\begin{thm}\label{thm:ExistenceGlobalSolution}
If $\epsilon<\epsilon_p$, then {\bf Case B} of {\bf Problem 2} admits a global piecewise smooth solution.
\end{thm}
\begin{proof}
The proof relies on the following lemmas. With Theorem \ref{thm:SecondInitialCondition} in hand, our task is to get the global solution for $\xi\geq\xi_0$. The key is to show $\bar{\partial}^\pm c<0$ for $\xi\geq\xi_0$. We first prove the following geometric lemma.

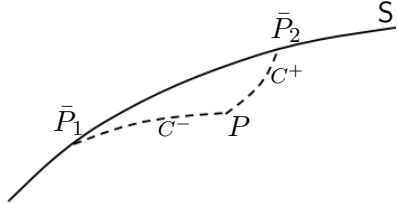
\begin{figure}[htbp]
\setlength{\unitlength}{1bp}
\begin{center}
\begin{picture}(200,90)
\put(-10,0){\includegraphics[scale=0.7]{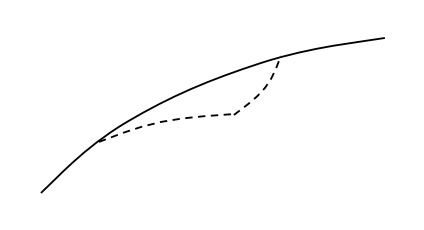}}
\put(89,40){\small$P$}
\put(105,77){\small$\bar{P}_2$}
\put(145,83){\small$\mathsf{S}$}
%\put(22,63){\tiny$\bar{\partial}^\pm c<0$}
\put(63,40){\tiny$C^-$}
\put(105,60){\tiny$C^+$}
\put(24,42){\small$\bar{P}_1$}
\end{picture}
\end{center}
\caption{convexity}\label{fig:convexshock}
\end{figure}
\begin{lem}\label{lem:ConvexShock}
The shock is convex provided $\bar{\partial}^\pm c<0$.
\end{lem}
\begin{proof}
As illustrated in Figure \ref{fig:convexshock}, we consider the pattern on $\mathsf{S}$ locally. Let $\bar{P}_1$ and $\bar{P}_2$ on $\mathsf{S}$ be two adjacent points connecting the $C^\pm$ characteristics with the intersection $P$. By $\bar{\partial}^\pm c<0$, we have $c|_{\bar{P}_1}>c|_{P}>c|_{\bar{P}_2}$ locally. This indicates $q|_{\bar{P}_1}>q|_{\bar{P}_2}$. So the strength of the shock is weakened locally along the shock $\mathsf{S}$ as $\xi$ increases. Moreover, $s$ is decreasing. Thus, the shock is convex locally. We complete the proof.
\end{proof}

Then, to show the global estimates, we prove the following continuity lemma. As illustrated in Figure \ref{fig:characteristicSW}, we assume that the $C^-$ characteristic connects $P^s$ on $\mathsf{S}$ with $\underline{P}_1$ on $\mathsf{W}$ and the $C^+$ characteristic connects $P_w$ on $\mathsf{W}$ with $\bar{P}_1$ on $\mathsf{S}$. Therefore, the domain is divided into $\Omega_1$, $\bar{\Omega}_1$ and $\underline{\Omega}_1$. We have
\begin{lem}
If it holds
\begin{equation}\label{eq:gPositiveConditionSecondInitial}
u(P^s)\in(\frac{\gamma-1}{2}\bar{q}+\delta_f, \underline{u}),\quad\mbox{and}\quad v(P^s)>0;
\end{equation}
\begin{equation}\label{eq:InitialDerivativeSgnSecondInitial}
\bar{\partial}^\pm c|_P<0,\quad\mbox{for}\quad P\in P^sP_w;
\end{equation}
and
\begin{equation}\label{eq:InitialDirectionSgnSecondInitial}
\tau|_P>0,\quad\mbox{for}\quad P\in P^sP_w,
\end{equation}
then we have 
\begin{equation}\label{eq:gPositiveConditionSecond}
u(P)\in(\frac{\gamma-1}{2}\bar{q}+\delta_f, \underline{u}),\quad\mbox{and}\quad v(P)>0,
\end{equation}
where $P$ belong to the shock $\mathsf{S}$ between $P^s$ and $\bar{P}_1$;
\begin{equation}\label{eq:InitialDerivativeSgnSecond}
\bar{\partial}^\pm c|_{\Omega_1\cup\bar{\Omega}_1\cup\underline{\Omega}_1}<0;
\end{equation}
and
\begin{equation}\label{eq:InitialDirectionSgnSecond}
\tau|_{\Omega_1\cup\bar{\Omega}_1\cup\underline{\Omega}_1}>0.
\end{equation}
\end{lem}

\begin{proof}
In view of the characteristic decomposition \eqref{eq:CharacteristicDecomposition}, we see that $\bar{\partial}^\pm c$ can not change sign along the characteristic. So $\bar{\partial}^\pm c|_{P^sP_w}<0$ indicates $\bar{\partial}^- c<0$ in $\Omega_1\cup\bar{\Omega}_1$ and $\bar{\partial}^+ c<0$ in $\Omega_1\cup\underline{\Omega}_1$. On the shock $\mathsf{S}$, $\bar{\partial}^- c<0$, \eqref{eq:ShockRelation} and \eqref{eq:PositiveG} imply $\bar{\partial}^+ c<0$ on $\mathsf{S}$ provided that \eqref{eq:gPositiveConditionSecond} holds. Since we have \eqref{eq:gPositiveConditionSecondInitial}, \eqref{eq:gPositiveConditionSecond} holds around $P^s$ on $\mathsf{S}$. By Lemma \ref{lem:ConvexShock}, $\bar{\partial}^\pm c<0$ guarantee the convexity of the shock $\mathsf{S}$. On the one hand, the convexity makes $u$ increase along $\mathsf{S}$. So $u|_{\mathsf{S}\cap\set{\xi\geq\xi(P^s)}}\geq u(P^s)$. On the other hand, \eqref{eq:SonicExpressions} indicates $c\bar{\partial}^+\tau=-\frac{\sin(2\omega)}{2\kappa}\bar{\partial}^+c>0$ provided $\bar{\partial}^+c<0$. So $\tau|_{\mathsf{S}\cap\set{\xi\geq\xi(P^s)}}\geq\tau|_{P^sP_w}>0$. These two reinforce the condition \eqref{eq:gPositiveConditionSecond} along $\mathsf{S}$. Therefore, \eqref{eq:gPositiveConditionSecond} must hold along the shock $\mathsf{S}$ between $P^s$ and $\bar{P}_1$. We have $\bar{\partial}^+ c<0$ in $\bar{\Omega}_1$. On the other hand, $\bar{\partial}^+c<0$, convex condition $\kappa_w\leq0$ and \eqref{eq:WallRelationCharacteristic} imply $\bar{\partial}^- c<0$ on $\mathsf{W}$. Finally, we get $\bar{\partial}^- c<0$ in $\underline{\Omega}_1$. Since \eqref{eq:InitialDirectionSgnSecondInitial} and $\tau|_{\mathsf{W}}\geq0$, \eqref{eq:InitialDerivativeSgnSecond} and $c\bar{\partial}^-\tau=\frac{\sin(2\omega)}{2\kappa}\bar{\partial}^-c<0$ in \eqref{eq:SonicExpressions} lead to \eqref{eq:InitialDirectionSgnSecond}.
\end{proof}

Now, to show that the sign condition $\bar{\partial}^\pm c<0$ can be extended to infinity, we need the following
\begin{lem}
All the characteristics connect the shock $\mathsf{S}$ and the wedge $\mathsf{W}$.
\end{lem}
\begin{proof}
We consider the set
\begin{equation*}
\Gamma=\set{\xi|~\mathsf{W}_{[0, \xi]} \mbox{ arrived by the $C^-$ characteristics from the shock $\mathsf{S}$.}}.
\end{equation*}
Define $\bar{\xi}:=\sup\Gamma$. We assume (with the aim of arriving at a contradiction) $\bar{\xi}\neq+\infty$. Then the $C^-$ characteristic is tangent to the wedge $\mathsf{W}$ on the point $\mathsf{W}|_{\bar{\xi}}$. Otherwise, we can extend the solution forward a little and it contradicts $\bar{\xi}=\sup\Gamma$. So we have $\beta=\tau$ at the point $\mathsf{W}|_{\bar{\xi}}$. This indicates $\omega=\tau-\beta=0$ and the flow is vacuum. But, by \eqref{eq:SonicExpressions}, we have $\bar{\partial}^-\tau\leq0$. So $\tau|_{P^s}\geq\tau|_{C^-}\geq\tau|_{\underline{P}_1}\geq0$. In view of Theorem \ref{thm:NoVacuum}, we can derive that $(u, v)|_{C^-}$ locates in the upper $(u, v)$ plane and is bounded by the curves: (1) the line $v=0$; (2) the circle $q= q|_{P^s}$; (3) the shock polar $G(\vec{U}, \epsilon)=0$. So the states $(u, v)$ on the point $\mathsf{W}|_{\bar{\xi}}$ must locate inside the $(u, v)$ domain bounded by the shock polar $G(\vec{U}, \epsilon)=0$. It indicates that the flow can not touch the vacuum along $C^-$. Thus we get contradiction and $\bar{\xi}=+\infty$. 

In a similar way, we can show that all the $C^+$ characteristics issuing from the wedge $\mathsf{W}$ arrive at the shock $\mathsf{S}$. To see this, we assume to the contrary that there exists a $C^+$ characteristic tangent to the shock $\mathsf{S}$ on a point $\mathsf{S}|_{\bar{\xi}}$. So, the shock is degenerate on $\mathsf{S}|_{\bar{\xi}}$. On the one hand, $\mathsf{S}$ is convex for $\xi\in[\xi_0, \bar{\xi}]$. On the other hand, \eqref{eq:SonicExpressions} yields $c\bar{\partial}^+\alpha=-\nu\Xi\sin(2\omega)\bar{\partial}^+c>0$, where the hypersonic incoming flow assumption guarantee $\Xi>0$. The opposing convexity of $\mathsf{S}$ and the $C^+$ characteristic at the tangent point $\mathsf{S}|_{\bar{\xi}}$ leads to a contradiction.
\end{proof}

Above lemma leads to $\bar{\partial}^\pm c<0$ for $\xi\geq\xi_0$. Finally, we give the lower boundedness of $\bar{\partial}^\pm c$. As illustrated in Figure \ref{fig:CharacteristicReflect}, we let $P\in\Omega$ and consider $\bar{\partial}^+c(P)$. By $\bar{\partial}^\pm c<0$, the terms in brackets of \eqref{eq:CharacteristicDecomposition} are negative, we have
\begin{equation*}
\bar{\partial}^+c(\bar{P}_1)<\bar{\partial}^+c(P)<0,
\end{equation*}
and
\begin{equation*}
\bar{\partial}^-c(\underline{P}_1)<\bar{\partial}^-c(\bar{P}_1)<0.
\end{equation*}
On the shock $\mathsf{S}$, we have
\begin{equation*}
\bar{\partial}^+c(\bar{P}_1)=g(\vec{U}, \epsilon)\bar{\partial}^-c(\bar{P}_1).
\end{equation*}
On the wall $\mathsf{W}$, we have
\begin{equation*}
\bar{\partial}^-c(\underline{P}_1)-\bar{\partial}^+c(\underline{P}_1)=(\gamma-1)q(\underline{P}_1)\kappa_w(\underline{P}_1)\leq0.
\end{equation*}
Summarising above four, we conclude
\begin{equation*}
\begin{split}
0&>\bar{\partial}^+c(P)\\
&>\bar{\partial}^+c(\bar{P}_1)\\
&=g(\vec{U}, \epsilon)\bar{\partial}^-c(\bar{P}_1)\\
&>g(\vec{U}, \epsilon)\bar{\partial}^-c(\underline{P}_1)\\
&=g(\vec{U}, \epsilon)\qnt{\bar{\partial}^+c(\underline{P}_1)+(\gamma-1)q(\underline{P}_1)\kappa_w(\underline{P}_1)}\\
&\geq\underline{K}\qnt{\bar{\partial}^+c(\underline{P}_1)+(\gamma-1)\bar{q}\inf_{P\in \mathsf{W}}\kappa_w(P)},
\end{split}
\end{equation*}
in the last inequality of which we have applied \eqref{eq:PositiveG} to control $g(\vec{U}, \epsilon)$. Iterating this inequalities, we can easily prove that
\begin{equation*}
0>\bar{\partial}^+c(P)\geq \frac{K+1}{1-K}(\gamma-1)\bar{q}\inf_{P\in \mathsf{W}}\kappa_w(P)+\inf_{\xi(P)=\xi_0}\bar{\partial}^+c(P).
\end{equation*}
For small $\epsilon$, the right-hand side of above inequality only depends on the wall $\mathsf{W}$. It means
\begin{equation*}
0>\bar{\partial}^+c(P)\geq -M_s.
\end{equation*}
In a similar way, we can show
\begin{equation*}
0>\bar{\partial}^-c(P)\geq -M_s.
\end{equation*}

Therefore, we conclude that $\bar{\partial}^\pm c$ is bounded and the speed $q=\sqrt{u^2+v^2}$ is less than $\underline{u}$. In view of $\bar{\partial}^\pm c<0$ and Theorem \ref{thm:NoVacuum}, for any point $P$ with $\xi(P)>\xi_0$, we have
\begin{equation*}
q(P)>q(Q)|_{\xi(Q)=\xi_0}>c_0+\delta_f.
\end{equation*}
Thus $(u, v)$ keeps a distance of $\min\set{\delta_f, \bar{q}-\underline{u}}>0$ from the critical sonic circle and the limit speed circle for fixed $\epsilon>0$. As a consequence, the Mach angle $\omega\in[\omega_1, \frac{\pi}{2}-\omega_1]$ and the derivatives $(\bar{\partial}^+, \bar{\partial}^-)$ are equivalent to $(\partial_x, \partial_y)$. Moreover, in view of the sonic expressions \eqref{eq:SonicExpressions}, we have
\begin{equation}\label{eq:SonicExpressionsEquiv0}
\abs{\nabla_{(x, y)} \vec{U}}\leq C(\epsilon)\qnt{\abs{\bar{\partial}^+\vec{U}}+\abs{\bar{\partial}^-\vec{U}}}\leq C(\epsilon)\qnt{\abs{\bar{\partial}^+c}+\abs{\bar{\partial}^-c}}\leq C(\epsilon)M_s,
\end{equation}
where $C(\epsilon)$ is a constant depending on $\epsilon$. So the solution can be extended to infinity for any fixed $\epsilon<\epsilon_p$. The entropy condition is a conclusion of Theorem \ref{thm:NoVacuum}. We complete our proof.
\end{proof}

\subsection{Asymptotic behaviour}
In this section, we end up our proof for the main theorem by the following
\begin{thm}[Dissipation of the shock wave]\label{thm:ExtinctionShock}
For {\bf Case B}, in $\Omega$, we have
\begin{equation*}
\lim\limits_{x\rightarrow+\infty}(u, v)=(u_\infty, v_\infty),
\end{equation*}
where $(u_\infty, v_\infty)$ is the supersonic intersection of the line $\frac{v}{u}=f_\infty$ and shock polar \eqref{eq:ShockPolar}, and
\begin{equation*}
\lim\limits_{x\rightarrow+\infty}\phi'(x)=\lim\limits_{(u, v)\rightarrow(u_\infty, v_\infty)}\phi'(u, v).
\end{equation*}
\end{thm}
\begin{proof}
By Theorem \ref{thm:ExistenceGlobalSolution}, the shock is convex and $\tau\geq0$, so $\lim\limits_{x\rightarrow+\infty}\phi'(x)=\inf\phi'(x)$ exists. In the proof there, we also deduce that $\Omega$ is covered by the $C^-$ characteristics from the shock. Therefore, on the $(u, v)$ plane, the flow field is covered by the contour line $F(u, v)=\mbox{constant}$ from the shock polar \eqref{eq:ShockPolar}, where $F$ is the Riemann invariant for the -characteristic. Then, in view of Theorem \ref{thm:NoVacuum}, it suffices to show
\begin{equation}\label{eq:ApproximateVelocityShock}
\lim\limits_{x\rightarrow+\infty}(u, v)(x, \phi(x))=(u_\infty, v_\infty).
\end{equation}

\begin{figure}[htbp]
\setlength{\unitlength}{1bp}
\begin{center}
\begin{picture}(200,165)
\put(-40,-10){\includegraphics[scale=0.7]{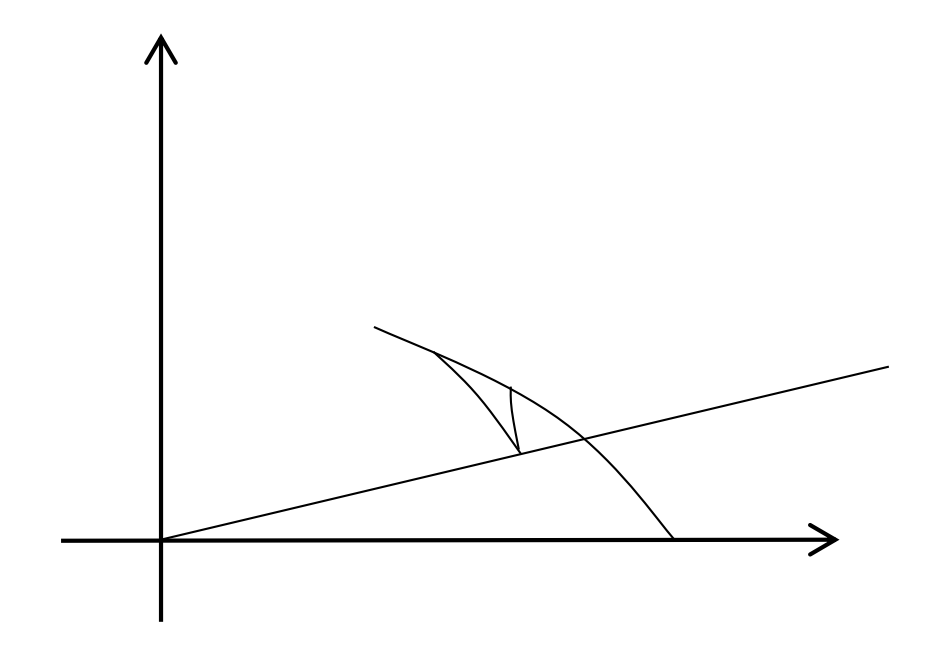}}
\put(220,16){\small$u$}
\put(0,135){\small$v$}
\put(195,60){\tiny$\frac{v}{u}=f_\infty$}
\put(90,82){\tiny$(u_m, v_m)$}
\put(110,38){\tiny$(u_c, v_c)$}
\put(111,65){\tiny$\bullet$}
\put(114,45){\tiny$\bullet$}
\put(133,49){\tiny$\bullet$}
\put(115,67){\tiny$(u_M, v_M)$}
\put(143,47){\tiny$(u_\infty, v_\infty)$}
\put(89,75){\tiny$\bullet$}
\put(15,75){\small shock polar}
\end{picture}
\end{center}
\caption{asymptotic states on $(u, v)$ plane}\label{fig:inftystates1}
\end{figure}
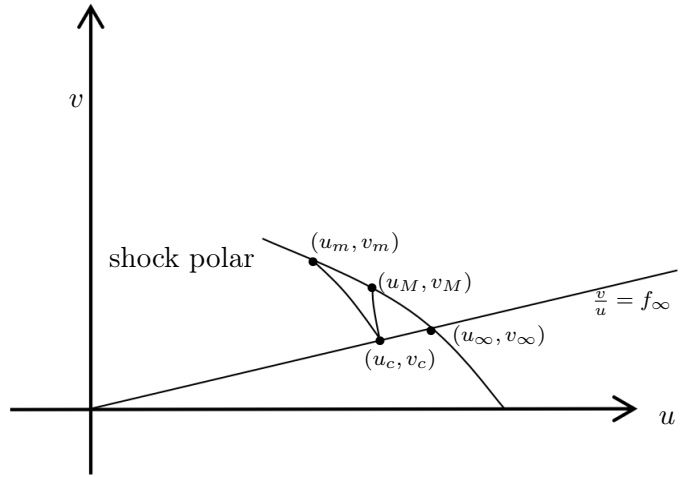

Let
\begin{equation}\label{eq:AyStates}
\lim\limits_{x\rightarrow+\infty}(u, v)(x, \phi(x))=(u_m, v_m),
\end{equation}
where $(u_m, v_m)$ is on the shock polar \eqref{eq:ShockPolar}. We assume that (with the aim of arriving at a contradiction) $v_m>v_\infty$.

In Figure \ref{fig:inftystates1}, consider the contour line $F(u, v)=F(u_m, v_m)$. This contour line inserts the shock polar and finally hits the line $\frac{v}{u}=f_\infty$ on $(u_c, v_c)$ where $v_\infty<v_c$. Then we consider the Riemann invariant $E(u, v)$ for $C^+$ characteristic and the contour line $E(u, v)=E(u_c, v_c)$ in the upper $(u, v)$ plane. By Theorem \ref{thm:NoVacuum}, this contour line threads out the shock polar at $(u_M, v_M)$, $u_M^2+v_M^2>u_c^2+v_c^2>u_m^2+v_m^2$, and then $v_M<v_m$. Thus, if the states $(u, v)(P^s)$ on the shock $\mathsf{S}$ approximates $(u_m, v_m)$ as $\xi\rightarrow+\infty$, then consider the corresponding reflection point $P^s_1$, which is on the shock $\mathsf{S}$ via a -characteristic and a +characteristic from $P^s$ (Figure \ref{fig:inftystates2}).

\begin{figure}[htbp]
\setlength{\unitlength}{1bp}
\begin{center}
\begin{picture}(200,180)
\put(-40,-10){\includegraphics[scale=0.7]{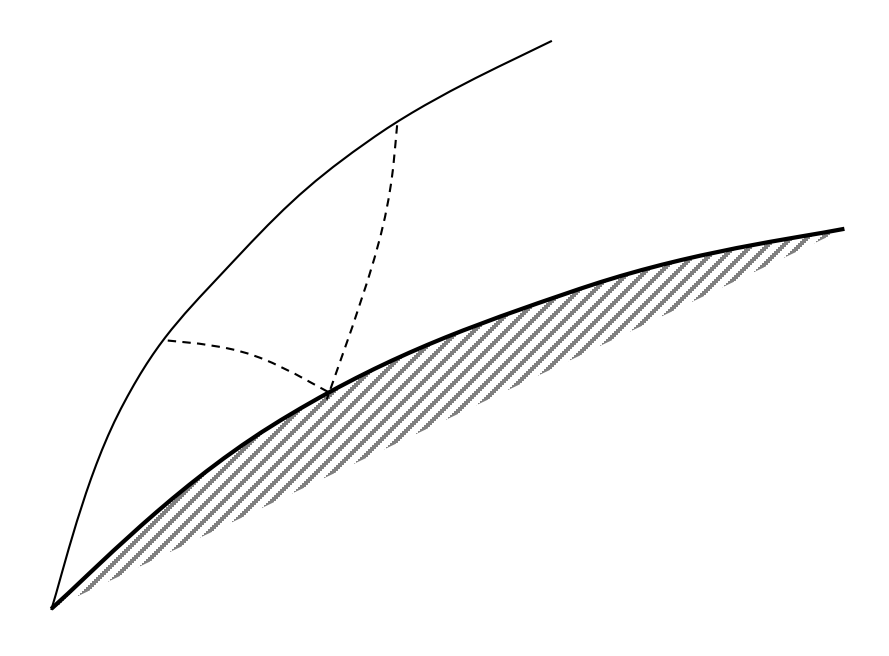}}
\put(120,160){\small$\mathsf{S}$}
\put(180,117){\small$\mathsf{W}$}
\put(23,70){\tiny$C^-$}
\put(-2,85){\small$P^s$}
\put(73,95){\tiny$C^+$}
\put(65,155){\small$P^s_1$}
\end{picture}
\end{center}
\caption{asymptotic states}\label{fig:inftystates2}
\end{figure}
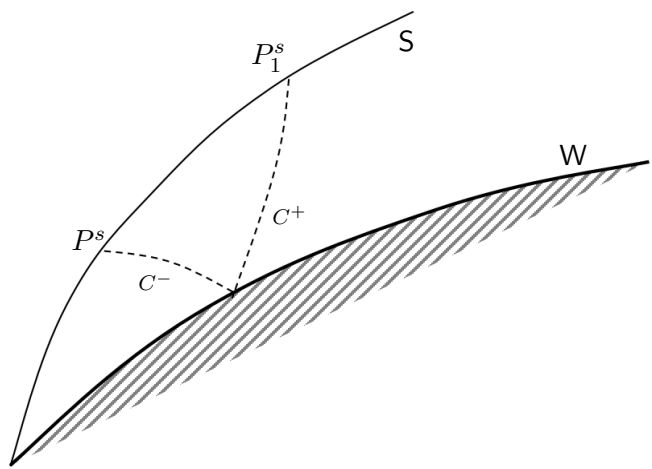
In view of the discussion above, we have that 
\begin{equation*}
(u, v)(P^s_1)\rightarrow(u_M, v_M),\quad\mbox{as}\quad(u, v)(P^s)\rightarrow(u_m, v_m).
\end{equation*}
This stands in contradiction to \eqref{eq:AyStates}. We complete the proof.
\end{proof}

\section*{Acknowledgments}
This work is supported by the National Natural Science Foundation of China under Grant No.12071298.

\bibliographystyle{plain}
\bibliography{hypersonicflowpastwedge20240611}

\end{document}